%% file: PFCH.tex
\documentclass{amsart}

\input{JHSymbol}    
\input{JHSettingArticle}

\graphicspath{ {figures_article}}

\usepackage{enumerate}
\usepackage{multirow}
\usepackage{tikz}
\usetikzlibrary{shapes, arrows}
\tikzstyle{format} = [draw, thin, align=center] 
\tikzstyle{solver} = [draw, thin, align=center, minimum height=4em]
\tikzstyle{medium} = [ellipse, draw, thin, align=center, 
minimum height=4em]
\tikzstyle{decision} = [diamond, draw, thin, align=center, 
minimum height=2.5em, aspect=2]

\usepackage{todonotes}
\usepackage{color}
\usepackage[numbers]{natbib}

\newcommand{\toggleb}{}
	
\newcommand{\diff}{{\textup{d}}}
\newcommand{\fraki}{{\mathfrak{i}}}
\newcommand{\dl}{L} 		
\newcommand{\dt}{{\mathrm{dt}}}

\newcommand{\x}{{\mathbf{x}}}

\newcommand{\gridfn}{\mathbb{H}_N }

\newcommand{\bfm}{{\mathbf{m}}}

\newcommand{\bfr}{{\mathbf{r}}}
\newcommand{\bfs}{{\mathbf{s}}}
\newcommand{\gm}{{\Omega_N}} 

\numberwithin{equation}{section}

\begin{document}
\title[Benchmark comparison of the PFC and FCH solvers]{Benchmark computations of the phase field crystal and functionalized Cahn-Hilliard equations via fully implicit, Nesterov accelerated schemes}

\author{Jea-Hyun Park}
\address[J.-H. Park]{Department of Mathematics, University of California, Santa Barbara, Santa Barabara CA 93106-3080, USA} 
\email[J.-H. Park]{jhpark1@ucsb.edu}

\author{Abner J. Salgado}
\address[A.J. Salgado]{Department of Mathematics, The University of Tennessee, Knoxville, TN 37996, USA} 
\email[A.J. Salgado]{asalgad1@utk.edu}

\author{Steven M. Wise}
\address[S.M. Wise]{Department of Mathematics, The University of Tennessee, Knoxville, TN 37996, USA} 
\email[S.M. Wise]{swise1@utk.edu}

\begin{abstract}
We introduce a fast solver for the phase field crystal (PFC) and functionalized Cahn-Hilliard (FCH) equations with periodic boundary conditions on a rectangular domain that features the preconditioned Nesterov’s accelerated gradient descent method (PAGD). With a Fourier collocation spatial descretization, we employ various second-order-in-time schemes. We observe a significant speedup with this solver compared to the preconditioned gradient descent method. With the PAGD solver, fully implicit, second-order-in-time schemes are not only feasible to solve the PFC and FCH equations, but more efficient than some semi–implicit schemes in some cases. Specifically, benchmark computations of five different schemes are conducted and indicate that, for the FCH experiments, the fully implicit schemes perform better than their IMEX versions in terms of computational cost needed to achieve a certain precision. For the PFC, the results are not as conclusive. We believe that this is due to a milder nonlinearity of the PFC compared to the FCH equation. We also discuss some practical matters in applying the PAGD: an averaged Newton preconditioner and a sweeping-friction strategy as heuristic ways to choose good preconditioner and solver parameters. The sweeping-friction strategy exhibits almost as good a performance as the case of the best manually tuned parameters.
\end{abstract}



\maketitle

	\section{Introduction}

We are interested in fast and accurate numerical solvers for initial value problems (IVPs) for nonlinear parabolic partial differential equations of the form 
\begin{equation}
	\partial_t u=M\Delta \frac{\delta \mathcal{E}}{\delta u}(u), \  t>0, \qquad u|_{t=0} = u_0,
	\label{eqn:genericPDE}
\end{equation}
supplemented with periodic boundary conditions. Here, $\frac{\delta \mathcal{E}}{\delta u}$ denotes the variational derivative of the energy 
\[
  \mathcal E(u)=\int_\Omega f(u, \nabla u, \Delta u) \diff x.
\]
The spatial domain is $\Omega$, which is assumed to be rectangular throughout this paper, and $M:\RR\to\RR$ is the so--called \emph{mobility constant}. While the mobility may depend on the unknown $u$ in general, we confine ourselves to the case of constant mobility $M\equiv1$ in this work. Two real world applications, the \emph{phase field crystal} (PFC) and \emph{functionalized Cahn-Hilliard} (FCH) equations (see Section \ref{sec:models} for more details) take this form and are of our main interest.

Our focus is on the numerical solvers. Nevertheless, for completeness, let us briefly mention existing works about the phenomena that the PFC and FCH equations model and their PDE analyses. These two equations are important models in materials science. The PFC equation describes crystal formation in a liquid bath, crack propagations in a crystal layer, and elastic and plastic deformations of a crystal lattice, to name a few. The FCH equations, on the other hand, describes network formation in a binary mixture and is a useful tool for modeling bilayer membrane formation and polymer electrolyte membrane evolution. The reader interested in applications is referred to \citep{elder2002, elder2004, emmerichPFC, asadi2015PFCreview} for the PFC, and \citep{gavish2011, gavish2012, promislow2009PEMreview} for the FCH model, respectively. There is some limited amount of work about these equations at the PDE level. For the PFC equation see \citep{conti2016PFCPDE, wang2010MPFCPDE}; whereas for the FCH see \citep{Cheng2020,dai2021FCHPDE}.   

Both the PFC and FCH are nonlinear, sixth order `parabolic' equations. As such, they share common numerical difficulties, such as accuracy and stability, and there have been efforts to overcome them; see, for example, \citep{cheng2008PFCsolver, wise2009PFCFD, hu2009PFCMG, gomez2012PFCnumerical, zhang2013PFCstepping} for the PFC, and \citep{jones2013FCHnumerical, Wise2018FCH, zhang2020benchmark, zhang2021FCHnumerical} for the FCH, respectively. If one wishes to have a long time evolution of the equations,  explicit discretization schemes in time must typically be excluded due to their stringent restriction on the time step size, ($\diff t\approx \diff x^6$), for stability. On the other hand, implicit schemes, which are more robust in terms of stability and accuracy, as a rule lead to a large, highly nonlinear system that must be solved at every time step. A substantial amount of work has been dedicated to developing schemes that mitigate the numerical difficulties or instabilities of either of these extreme approaches, fully explicit schemes, on one hand, and fully implicit schemes, on the other. Examples of this are the convex splitting technique \citep{cheng2008PFCsolver, wise2009PFCFD, hu2009PFCMG, gomez2012PFCnumerical, jones2013FCHnumerical, zhang2020benchmark, zhang2021FCHnumerical}, and the SAV technique \citep{MR4091597,MR4106756,MR4234206}, to name a few. Both of these approaches, however, are known to create larger local truncation errors than implicit schemes (\citep{xu2016disguise, zhang2020benchmark}). If a reliable, robust, and  efficient iterative solver is available to handle the nonlinear equations resulting from fully implicit schemes, a good balance between accuracy and the efficiency may be within reach.

In previous work \citep{psw2021PAGD}, we showed that the preconditioned Nesterov's accelerated gradient descent method (PAGD; see Algorithm \ref{alg:PAGD} for definition) can be applied to approximate the minimizer of a strongly convex objective that is \emph{locally Lipschitz} smooth as opposed to \emph{globally Lipschitz} smooth ones as most of the literature assumes. This significantly extends the applicability of the PAGD as a numerical PDE solver. In \citep{psw2021PAGD}, it is also  reported that the  PAGD's performance can be significantly better than that of the preconditioned gradient descent method (PGD; see Algorithm \ref{alg:PGD} for definition), especially on harder problems.  

In light of our previous discussion, the construction and analysis of efficient, time-adaptive, implicit schemes --- with the PAGD solver as the central engine --- for high--order nonlinear  parabolic equations, such as the PFC and FCH equations, is an underdeveloped subject and our main motivation in writing this contribution. Our first goal is to establish that the PAGD makes an efficient solver for real world problems (see Section \ref{sec:PAGDperform}).  But this begs the question: \emph{Does the PAGD make implicit schemes more attractive than, say, semi--implicit ones? What should one compare to answer this question?}  These questions are addressed in Section \ref{sec:scheme-comparision}. To compare schemes, we measure the computational cost needed to achieve a certain precision. Under this metric, our experiments indicate that implicit schemes are indeed a better choice when nonlinearity of the problem is ``strong.''  If one compares `dollars per digit' cost --- that is, the number of flops to achieve a desired level of accuracy in a computed solution ---  as we do in Section \ref{sec:scheme-comparision}, our experiments indicate that the implicit schemes are often a better choice over linear semi--implicit methods.

In the course of achieving our first goal, we also discuss two practical issues. One is about how to choose parameters involved in the PAGD scheme. To implement the PGD method, only the step size $s$ needs to be set. In contrast, PAGD contains an additional tunable parameter, which herein we call \emph{friction} and label $\eta$ (see \citep{psw2021PAGD} for motivation behind this naming convention). We suggest what we call the  \emph{sweeping}--$\eta$ or \emph{sweeping--friction} strategy rather than finding a single optimal constant by trial and error. It turns out that the sweeping--$\eta$ strategy is almost as efficient as the best--tuned constant friction setup and it is more robust than the latter in the sense that its performance depends less on different ranges for $\eta$ to sweep than that of the constant--friction setup does on different fixed values of $\eta$. A detailed discussion and its intuition is explained in Section \ref{sub:NLSolvers}. The second practical issue is how to choose a good preconditioner. Again, rather than finding necessary constants by trial and error, we suggest what we call the \emph{averaged Newton} preconditioner, which computes the parameters involved in the preconditioner in such a way that it mimics the second variation of the objective functional among a certain type of linear operators. See Section \ref{sub:NLSolvers} for a detailed discussion.

The rest of this paper is organized as follows: Section \ref{sec:models} summarizes the mathematical formulations of the two models of interest, namely the FCH and PFC equations. In Section \ref{sec:discretization}, we detail how to discretize (in time) the PDEs in four different ways, whose resulting solvers are used in the numerical experiments. We also talk about how to construct the numerical solvers in the same section and explain the sweeping--friction strategy and the averaged Newton preconditioner in detail. Section \ref{sec:numerics} summarizes the benchmark problems, the results of the numerical experiments, and our interpretation of the results. Finally, we make concluding remarks in Section \ref{sec:conclusion}.

\section{The phase field crystal and functionalized Cahn-Hilliard equations}
\label{sec:models}

We begin our discussion by providing some details regarding the models that we shall be interested in.

\subsection{Phase field crystal equation}

There are several versions of the phase field crystal (PFC) equations~\citep{asadi2015PFCreview, elder2002, elder2004, elder07,provatas10}. We will use only the prototypical version, as presented in  \citep{elder2004}. The other variants of the model bring similar numerical challenges. The PFC model, at its heart, describes solidification of a unary crystal from its liquid phase.  The model captures atomic--scale features on a diffusive time scale. Suppose that $u:\Omega\to\mathbb{R}$ defines an atom density. The free energy of the system (at constant temperature) is
	\begin{equation}
\mathcal{E}_{\mathrm{PFC}}(u):=\int_\Omega \left[ \frac{1}{4} u^4 + \frac{1-\varepsilon}{2}u^2-|\nabla u|^2 + \frac{1}{2} (\Delta u)^2 \right] \diff x,
	\label{obj:PFC-egy}
	\end{equation}
where $\varepsilon$ is a parameter that mimics the temperature variation. We assume that $u$ satisfies periodic boundary conditions on $\Omega$, for simplicity, and that the dynamics for $u$ are mass conserving and free energy dissipative. This leads to the following system of equations:
	\begin{subequations}
	\begin{alignat}{3}
\partial_t u &= M \Delta \mu,
		\\
\mu &=\frac{\delta \mathcal{E}_{\mathrm{PFC}}}{\delta u} = u^3 +(1-\varepsilon)u+2\Delta u +\Delta^2 u,
	\end{alignat}
	\label{eqn:PFCsystem}
	\end{subequations}
which is an $H^{-1}$--gradient flow with a constant mobility $M>0$. Mass is conserved, i.e., $\diff_t\int_\Omega u(x,t) \, \diff x = 0$, and energy is dissipated at the rate $\diff_t\mathcal{E}_{\mathrm{PFC}}(u) = -\int_\Omega\left|\nabla \mu\right|^2\,\diff x$.

\subsection{The functionalized Cahn-Hilliard equation} 
The functionalized Cahn-Hilliard equation is a phase field model that describes  network formation of amphiphilic di--block co--polymer mixtures~\citep{gavish2011, gavish2012}.  As with the PFC, there are several versions of the FCH equations, as the model needs to be fine--tuned to the physical system of interest~\citep{gavish2011, gavish2012, christlieb2020benchmark, zhang2020benchmark}. We will use the same model as that used in the computational benchmark paper~\cite{zhang2020benchmark}. (See also~\citep{christlieb2020benchmark}.)  Let $u:\Omega\to \mathbb{R}$ denote the volume fraction of component $A$ in a binary mixture of molecules consisting of $A$ and the other component $B$. The free energy of the mixture (at constant temperature) is
	\begin{equation}
	\mathcal{E}_{\mathrm{FCH}}(u):=\int_{\Omega} \left[ \frac{1}{2}\left(\varepsilon^{2} \Delta u-F^{\prime}(u)\right)^{2}-\left(\frac{\varepsilon^{2}}{2} \eta_{1}|\nabla u|^{2}+\eta_{2} F(u)\right) \right] \diff x,
	\label{obj:FCHegy}
	\end{equation}
where $F$ is a double well potential, $\varepsilon$ is an interface thickness parameter, and $\eta_1,\eta_2>0$ are material parameters. We assume, for simplicity that $u$ is periodic on the square domain $\Omega$. The corresponding FCH equation, with constant mobility $M>0$, reads:
	\begin{equation}
  	\begin{aligned}
	\partial_{t}u&=M\Delta \frac{\delta \mathcal{E}_{\mathrm{FCH}}}{\delta u} \\
	&= M\Delta\left[\left(\varepsilon^{2} \Delta u-F^{\prime \prime}(u)\right)\left(\varepsilon^{2} \Delta u-F^{\prime}(u)\right)-\left(-\varepsilon^{2} \eta_{1} \Delta u+\eta_{2} F^{\prime}(u)\right)\right].	
  \end{aligned}
	\label{eqn:FCHPDE}
\end{equation}
Written as a system of three second--order equations, we have, equivalently,
	\begin{subequations}
	\begin{alignat}{3}
\partial_t u &= M \Delta \mu,
		\\
\mu &=\frac{\delta \mathcal{E}_{\mathrm{FCH}}}{\delta u} = (\varepsilon^2\Delta u -F''(u)+\eta_1)\omega +(\eta_1-\eta_2)F'(u),
		\\
		\omega&=\varepsilon^2\Delta u - F''(u).
	\end{alignat}
	\label{eqn:FCHsystem}
	\end{subequations}
We assume that $F:\RR\to\RR$ is a polynomial double--well potential of the form
\[
  F(\zeta)=\frac{1}{2}(\zeta+1)^{2}\left(\frac{1}{2}(\zeta-1)^{2}+\frac{2}{3} \tau(\zeta-2)\right),
\] 
whose symmetry can be tuned by adjusting $\tau\in\RR$.

Similar to PFC, the FCH system can be seen as a $H^{-1}$--gradient flow of the energy $\mathcal{E}_{\mathrm{FCH}}$. As before, mass is conserved, $\diff_t\int_\Omega u(x,t) \, \diff x = 0$, and energy is dissipated at the rate $\diff_t\mathcal{E}_{\mathrm{FCH}}(u) = -\int_\Omega\left|\nabla \mu\right|^2\,\diff x$.

\section{Discretization and numerical solvers}
\label{sec:discretization}
In this section, we describe our numerical approach. First, in Section~\ref{sub:TimeStepping}, we present our adaptive time discretizations: the fully and semi--implicit BDF2 and midpoint rule (MP). This reduces our problem to a sequence of time independent, sixth order, nonlinear elliptic equations, which are then discretized by a Fourier collocation method as detailed in Section~\ref{sec:FCH-sp-disc}. Finally, the ensuing nonlinear systems of equations are solved using the linear and nonlinear solvers described in Sections~\ref{sub:LinSolvers} and \ref{sub:NLSolvers}.

\subsection{Time discretization}
\label{sub:TimeStepping}
\subsubsection{Fully and semi-- implicit BDF2 and midpoint rule}

We choose four different schemes for time discretization: the fully implicit second order backward differentiation formula (BDF2), fully implicit midpoint rule (MP), a (linear) semi--implicit second order backward differentiation formula (LBDF2), and a (linear) semi--implicit midpoint rule (LMP). There are several reasons for these choices. First, we do not consider explicit schemes since, to be stable, they  require extremely stringent time step size restrictions of the form $\diff t\approx\diff x^6$; see \citep{Wise2018FCH}.  Second,  we want to compare the performance of fully implicit schemes and their semi--implicit versions. Both of these classes of schemes are known to be unconditionally stable and accurate.  However, at first glance, one might expect that fully implicit schemes will not be computationally efficient, as they require solving a nonlinear system every time step.  On the other hand, while \emph{semi--implicit} schemes, like the first order convex splitting scheme \citep{eyre1998, eyre1998unconditionally, backofen2019, cherfils2021}, are known to be often fast and stable, these properties always come at the expense of accuracy. Thus, a fair comparison between these two classes of schemes must be made by considering both speed and accuracy.

Let us now describe our schemes in more detail. We introduce a nonuniform time grid $\{t_{n}\}_{n \geq 0}$ with (variable) time step defined by $\dt_{n+1}=t_{n+1} - t_{n}$. The sequence of functions $\{u^n : \Omega \to \RR\}_{n\geq 0}$ is meant to be an approximation of $u$, the solution of \eqref{eqn:genericPDE}, at the time grid points, i.e., $u^n(\cdot) \approx u(t_n, \cdot)$ for all $n \geq 0$.

The fully implicit schemes are defined as follows: given the initial value $u^0 = u_0$, find $u^{n+1}$, for $n \geq 0$, as the solution of
	\begin{equation}
a_n  u^{n+1}+b_n  u^{n}+c_n  u^{n-1} = M\Delta \frac{\delta \mathcal E}{\delta u}(\breve u^{n+1}),
	\label{eqn:FCH-semidis-implicit}
	\end{equation}
where the coefficients $\{a_n\}_{n\geq 0}$, $\{b_n\}_{n\geq 0}$, $\{c_n\}_{n\geq 0}$, and the choice of the function $\breve u^{n+1}: \Omega \to \RR$ define the various fully implicit schemes. In particular, for the BDF2 scheme, we have
\[
  a_0 = \frac1{\dt_1}, \qquad b_0 = -\frac1{\dt_1}, \qquad c_0 = 0, \qquad \breve u^1 = u^1,
\]
and, for $n \geq 1$,
\begin{align}
	\begin{split}
	&a_n=\frac{1}{\dt_{n+1}}+\frac{1}{\dt_{n+1}+\dt_{n}}, \quad
	b_n=-\frac{1}{\dt_{n+1}}-\frac{1}{\dt_{n}},
	\\ 
	&c_n=\frac{1}{\dt_{n}}-\frac{1}{\dt_{n+1}+\dt_{n}}, \quad
	\breve u^{n+1}= u^{n+1}.
	\end{split}
	\label{obj:BDF2}
\end{align}
The MP scheme is defined by
\begin{align}
	&a_n =\frac{1}{\dt_{n+1}}, \quad
	b_n = - \frac{1}{\dt_{n+1}}, \quad 
	c_n = 0,\quad
	\breve u^{n+1}=\frac{u^{n+1}+u^{n}}{2},
	\label{obj:MP}
\end{align}
for all $n\ge 0$.

The semi--implicit schemes we shall use are of linear IMEX type (see, e.g., \citep{christlieb2020benchmark} for details). They only require, at each time step, the solution of a linear system of equations. To achieve this, these methods decompose the chemical potential, $\frac{\delta \mathcal{E}}{\delta u}$, into two parts
\begin{equation*}
	\frac{\delta \mathcal E}{\delta u}=\mathfrak L(u)+\mathfrak N(u)
\end{equation*}
where the \emph{linear} part $\mathfrak L$ is a linear, positive semi--definite operator. The remaining terms constitute the \emph{nonlinear} part. We have chosen the following decompositions. For the PFC model \eqref{eqn:PFCsystem}, we set
\begin{align*}
  \mathfrak L(u) &= (1+\Delta)^2u, \\ \mathfrak N(u) &= u^3 -\varepsilon u,
\end{align*}
For the FCH model \eqref{eqn:FCHPDE}, we add and subtract a second and a zeroth order term and obtain
\begin{align*}
  \mathfrak L(u) &= \varepsilon^4\Delta^2 u + \kappa_2 (-\Delta) u + \kappa_0 u \\
  \mathfrak N(u) &= -\kappa_0 u -\varepsilon^2\Delta F'(u) - (\varepsilon^2(F''(u)-\eta_1)-\kappa_2) \Delta u +(F''(u)-\eta_2) F'(u),
\end{align*}
where the parameters are set to $\kappa_0=(1-2\tau^2 +\eta_2)$ and $\kappa_2=1$, respectively. The parameter $\kappa_0$ is equal to the linear coefficient of the zeroth order term $(F''(u)-\eta_2) F'(u)$, which is a quintic polynomial in $u$. The value of the parameter $\kappa_2$ was found by trial and error. We note that these parameters are not optimized and that they differ from those in \citep{christlieb2020benchmark} because it turned out that, in our setting, the values in \citep{christlieb2020benchmark} made some of our schemes extremely slow.
However, we tried several reasonable options and have chosen a combination that yields an expected evolution of the FCH model. The semi--implicit schemes are obtained by treating the linear part, $\mathfrak L(u)$, implicitly and the nonlinear part, $\mathfrak N(u)$, explicitly, using a second-order extrapolation of the previous approximations. That is,
\begin{equation}
	a_n  u^{n+1}+b_n  u^{n}+c_n  u^{n-1}= M\Delta\left( \mathfrak L(\bar u^{n+1})+\mathfrak N(\breve u^{n+1}) \right)
	\label{eqn:FCH-semidis-simplicit}
\end{equation}
where $a_n$, $b_n$, and $c_n$ are the same as in \eqref{obj:BDF2} and \eqref{obj:MP}. The extrapolations, for LBDF2  are given by
\begin{align*}
  \bar u^{n+1} &= u^{n+1}, \\ \breve u^{n+1}&=u^{n}+\rho_{n+1} (u^n - u^{n-1}),
\end{align*}
whereas, for the LMP, they are
\begin{align*}
	\bar u^{n+1}&=\frac12 \left( u^{n}+u^{n+1}\right),
	\\
	\breve u^{n+1}&=
	\frac{(2+\rho_{n+1})u^{n}-\rho_{n+1} u^{n-1}}{2},
\end{align*}
with $\rho_{n+1} =\frac{\dt_{n+1}}{\dt_{n}}$. When $n=0$, an artificial time iterate $u^{-1}:=u^{0}$ is used, which reduces the explicit treatment using the extrapolation to a pure explicit one, $u^{0}$, without extrapolating. Observe that \eqref{eqn:FCH-semidis-simplicit} is linear in $u^{n+1}$.


\subsubsection{Adaptive time stepping}
To be able to accurately carry out long time simulations, we employ variable time step sizes,  which are chosen adaptively~\citep[Ch. III.5]{Hairer1993}. At every time step, after finding our numerical solution, we compute an error indicator and, if it is not smaller than our prescribed tolerance the current approximation is discarded, the step size reduced, and a new numerical solution is computed.

For the BDF2, LMP, or LBDF2 schemes, we follow the adaptive strategy detailed in \citep[Section 3.2]{christlieb2020benchmark}, which we refer to as \emph{AM3 stepping}. One exception is that the midAB2 stepping (see below) is used for PFC2 experiment (see Section \ref{sec:numerics} for details of PFC2 experiment) when it is solved by the LMP solver. This is because midAB2 stepping yields a way better result than AM3 for this computation. We now describe the AM3 stepping. See also Algorithm \ref{alg:solver}. We first introduce a predetermined stepping tolerance $\mathrm{TOL}>0$, as well as maximum and minimum time step sizes, denoted by $\dt_{min}$ and $\dt_{max}$, respectively. We begin by setting $n=0$, a tentative time step size $\tilde \dt_{n+1} = \dt_{min}$, and $t_0 =0$. while $t_n < T$:
\begin{enumerate}[1.]
  \item Compute $\tilde u^{n+1}$. This is a tentative solution at $\tilde t_{n+1}=t_n+\tilde \dt_{n+1}$, and is obtained using one of the main schemes (BDF2, LMP, or LBDF2). Also, copy current time step size $\tilde \dt_{tmp} = \tilde \dt_{n+1}$. This is used when a new tentative time step size is computed below.
  
  \item Compute $\hat u^{n+1}$. This is a solution of a higher order accuracy obtained using an explicit variant of the AM3 scheme:
  \begin{equation}
    \begin{aligned}
      \hat u^{n+1} =u^{n} &+\frac{\tilde \dt_{n+1}}{6} \left[\frac{3+2 \rho_{n+1}}{1+\rho_{n+1}} R\left(\tilde{u}^{n+1}\right) \right. \\
      &+ \left. (3+\rho_{n+1}) R\left(u^{n}\right)-\frac{\rho_{n+1}^{2}}{1+\rho_{n+1}} R\left(u^{n-1}\right)\right],
    \end{aligned}
      \label{fml:AM3-high-sol}
  \end{equation}
  with
  \[
    \rho_{n+1}=\frac{\tilde \dt_{n+1}}{\dt_{n}}, \qquad R(v)= M\Delta \frac{\delta \mathcal E}{\delta u}(v).
  \]
  
  \item Estimate the error with
  \begin{equation}
  \mathrm{ERR} =	
  \frac{\norm{\tilde u^{n+1}-\hat u^{n+1}}{L^2}}{\norm{\hat u^{n+1}}{L^2}}.
  \label{fml:AM3-err}
  \end{equation}
  
  \item If $\mathrm{ERR} \leq \mathrm{TOL}$ or $\tilde \dt_{n+1} \le \dt_{min}$:
  \begin{itemize}
    \item $\dt_{n+1} = \tilde \dt_{n+1}$,
    
    \item $t_{n+1} = t_n + \dt_{n+1}$,
    
    \item $u^{n+1} = \tilde{u}^{n+1}$,
    
    \item Increment $n$.
  \end{itemize}
    \item Compute a new \emph{tentative} time step by
  \begin{equation}
      \bar \dt=0.9 \left(\frac{\mathrm{TOL}}{\mathrm{ERR}}\right)^{1 / 3} \tilde \dt_{tmp},
      \label{fml:time-step}
  \end{equation}
  and
  \begin{equation}
      \tilde \dt_{n+1} = \max \left(\dt_{min}, \min \left(\bar \dt, \dt_{max}\right) \right),
      \label{fml:time-step-med}
  \end{equation}

\end{enumerate}

Note that $\tilde \dt_{n+1}$ computed in \eqref{fml:time-step-med} can be a tentative time step size for the next time marching (if $n$ is incremented in step 4) or a shrunken time step size for a recomputation during the current time marching (if $n$ is not changed in step 4). We also comment that, in \eqref{fml:time-step}, the number 0.9 is a so--called \emph{safety factor}, and that the power of a third in \eqref{fml:time-step} is related to the fact that the local truncation error of our schemes of interest (MP, BDF2, LMP, and LBDF2) is of order two.

Following a suggestion found in \citep{burkardt2020}, we use a different error estimator in the case of the MP scheme (and the LMP when solving PFC2 as mentioned above as an exception), which we call \emph{midAB2 stepping}. Using the computed values of the solution at $t_{n-2}$, $t_{n-1}$, and $t_{n}$ one can compute approximations at the midpoints $t_{n-1/2}$ and $t_{n-3/2}$, these are then used to construct a second order polynomial that is then evaluated at $\tilde t_{n+1}=t_n + \tilde\dt_{n+1}$ to obtain
\begin{equation}
	\begin{split}
		\hat u^{n+1}_{AB2}=&u^{n} \frac{\left(\tilde\dt_{n+1}+\dt_{n}\right)\left(\tilde\dt_{n+1}+\dt_{n}+\dt_{n-1}\right)}{\dt_{n}\left(\dt_{n}+\dt_{n-1}\right)}
		\\
		&-u^{n-1} \frac{\tilde\dt_{n+1}\left(\tilde\dt_{n+1}+\dt_{n}+\dt_{n-1}\right)}{\dt_{n} \dt_{n-1}}\\
		&+u^{n-2} \frac{\tilde\dt_{n+1}\left(\tilde\dt_{n+1}+\dt_{n}\right)}{\dt_{n-1}\left(\dt_{n}+\dt_{n-1}\right)}.
	\end{split}
	\label{fml:midAB2-high-sol}
\end{equation}
Then, the local truncation error can be computed by
\begin{equation}
	{T}_{n+1}=\left(\tilde u^{n+1}_{MP}-\hat u^{n+1}_{{AB2}}\right) \frac{1}{1-1 /\left(24 {R}_{n}\right)},
\end{equation}
where $\tilde u^{n+1}_{MP}$ is a tentative solution at $\tilde t_{n+1}$ using the MP and 
\begin{equation}
	{R}_{n}=\frac{1}{24}+\frac{1}{8}\left(1+\frac{\dt_{n}}{\tilde\dt_{n+1}}\right)\left(1+2 \frac{\dt_{n}}{\tilde\dt_{n+1}}+\frac{\dt_{n-1}}{\tilde\dt_{n+1}}\right).
\end{equation}
To match the scaling of the error with the AM3 stepping case, we use a $L^2$--normalized error estimator
\begin{equation}
\mathrm{ERR}= \frac{\norm{\tilde u^{n+1}_{MP}-\hat u^{n+1}_{{AB2}}}{L^2}}{\norm{\hat u^{n+1}_{AB2}}{L^2}} \frac{1}{1-1 /\left(24 {R}_{n}\right)}
\label{fml:midAB2-err}
\end{equation}
when determining the tentative step size. In the numerical experiments, the midAB2 stepping applies from the third time step $\dt_3 =t_3 -t_2$ because it requires three previous approximations. For $n=1, 2$, we use AM3 stepping. We refer the reader to \citep{burkardt2020} for further details on the midAB2 stepping strategy.

\subsection{Spatial discretization via the Fourier collocation method}\label{sec:FCH-sp-disc}
To take advantage of the fact that our PDEs are supplemented with periodic boundary conditions on a square, we use the Fourier collocation method for spatial differentiation and integration. 

We introduce $K \in \NN$, so that the grid resolution is $N = 2K+1$, and the grid spacing is $h=1/N$. We define
\begin{align*}
	\mathbb{N}^2_{0,N} &= \left\{ \bfm =(m_1,m_2) \in\ZZ^2 \ \middle| \ 0\le m_1,m_2\le N \right\},
	\\
	\mathbb{N}^2_{N} &= \left\{ \bfm =(m_1,m_2) \in\ZZ^2 \ \middle| \ 1\le m_1,m_2\le N \right\},
	\\
	\ZZ^2_K &= \left\{ \bfr =(r_1,r_2) \in\ZZ^2 \ \middle| \ -K\le r_1,r_2\le K\right\},
\end{align*}
and introduce the uniform grid domain
\begin{equation}
  \gm = [0,L] \cap h \NN_{0,N}^2.
\label{obj:grid-domain}
\end{equation}
We also define the \emph{trial space} of periodic grid functions
\begin{equation}
	\begin{split}
\gridfn = 
\left\{v_N:\gm\goesto\CC \ \middle| \ v_N(0,hm)=v_N(L,hm), \right.
\\
\qquad \left. v_N(h\ell,0)=v_N(h\ell,L), 
(m,\ell) \in \NN_{0,N}^2 \right\}.
	\end{split}
\label{obj:grid-fn}
\end{equation}
In the numerical experiment for the PFC, the domain is translated so that $\Omega=[-L/2,L/2]$ and the grid domain is also shifted accordingly. This is purely a cosmetic matter since we are dealing with the periodic boundary conditions. 

We omit the details of the case where $N=2K$ for brevity but, up to a slight difference in indexing, a similar construction can be carried out.

Finally, we replace the differential operators in our problems with so--called \emph{Fourier interpolation differentiation}; see \citep[pp. 123---124]{canuto2007spectral}, 
which we now describe in some detail for the FCH case. For the PFC, a shift of $\dl/2$ in each coordinate direction is necessary. First, we endow $\gridfn$ with the discrete $L^2$--inner product
\[ (u_N,v_N)_N=h^2\sum_{\bfs\in\mathbb{N}^2_{N}} u_N(\x_\bfs) \overline{v_N(\x_\bfs)},\]
where $\overline{v_N(\x_\bfs)}$ denotes the complex conjugate.

The Fourier interpolation differentiation can be defined and computed via its diagonalization using the \emph{discrete Fourier transform} (DFT) and the \emph{inverse discrete Fourier transform} (IDFT). 
The DFT $\hat w_K$ of $w_N \in \gridfn$ is defined by
\[ 
\hat w_K(\bfr)=(w_N,e^{\frac{2\pi}{\dl} \fraki \bfr\cdot(\cdot)})_N=h^2\sum_{\bfs\in\mathbb{N}^2_{N}} w_N(\x_\bfs) e^{-\frac{2\pi}{\dl} \fraki \bfr\cdot\x_\bfs},
\quad \bfr \in \ZZ^2_K.
\]
In particular, given $w_N\in\gridfn$ and $\alpha\in \{-1, 1, 2\}$, we set
\begin{equation}
	[(-\lap_N)^\alpha w_N](\x_\bfm)=\sum_{\bfr\in\ZZ^2_K} \left(\frac{4\pi^2 |\bfr|^2}{\dl^2} \right)^\alpha \hat w_K(\bfr) e^{\frac{2\pi}{\dl} \fraki \bfr\cdot\x_{\bfm}}.
	\label{def:frac-Lap}
\end{equation}
This defines the discrete Laplacian if $\alpha=1$, the discrete biharmonic operator if $\alpha = 2$, and the inverse Laplacian if $\alpha = -1$ and $(w_N,1)_N=0$.
In the last case, however, $\bfr=\mathbf{0}$ must be excluded in the summation though it is present for notational convenience. 
Define the following mesh-dependent \emph{negative norm}
\begin{equation}
	\norm{w_N}{-1,N}=\sqrt{\left( (-\Delta_N)^{-1}w_N,w_N\right)_N}.
	\label{def:neg-norm}
\end{equation}
In addition to replacing differential operators, the spatial integration is replaced with the \emph{composite trapezoidal rule}. This is indicated by the symbol $( \, \cdot \, , \, \cdot \,)_N$.

There are two significant advantages of using the Fourier collocation method. First, it is accurate. For smooth, periodic functions, the two aforementioned operations are known to be spectrally accurate (see \citep[pp. 53, 272]{canuto2007spectral} and \citep{trefethen2014trapezoidal}). Second, it is fast. We can take advantage of the fast Fourier transform (FFT) when computing the Fourier interpolation differentiation, which reduces the computational cost significantly (see \citep[pp. 52---54]{canuto2007spectral}). 

\subsection{Linear solvers for semi--implicit schemes}
\label{sub:LinSolvers}
As mentioned above, the semi--implicit schemes require us to solve a linear equation at every time step. It turns out that the coefficient matrix of the linear system results only from differentiation. Since we are using a Fourier collocation method, we apply the FFT to solve the linear equations involved in the LMP and LBDF2 schemes.

\subsection{Nonlinear solvers for fully-implicit schemes}
\label{sub:NLSolvers}
Let us now discuss solvers for the fully implicit schemes. We employ the \emph{preconditioned Nesterov's accelerated gradient descent method} (PAGD) as our main nonlinear solver, whose convergence theory and applications to nonlinear PDEs are found in \citep{psw2021PAGD}, and the \emph{preconditioned gradient descent method} (PGD) for comparison, which is studied in \citep{feng2017preconditioned}. To summarize how these solvers work, let us explicitly state the fully discrete problem required for time marching, where we drop the superscript for the new time marching for ease of notation and so that it can be viewed as a time--independent problem on its own: given $u^{n-1}_N, u^{n}_N \in \HH_N$, find $u_N\in \HH_N$ such that 

\begin{equation}
	a_n  u_N+b_n  u_N^{n}+c_n  u_N^{n-1}=M\Delta_N \frac{\delta \mathcal E_{{N}}}{\delta u}(\breve u_N),
	\label{eqn:FCH-fulldis-implicit}
\end{equation}
where, as before, $\breve u_N=u_N$ or $\frac{u_N+u_N^{n}}{2}$ if the BDF2 or MP is used, respectively.  By $\mathcal E_N$ we denote the discrete version of either the PFC \eqref{obj:PFC-egy} of FCH \eqref{obj:FCHegy} energy. Namely, the one that is defined by replacing the differential operators by Fourier interpolation differentiation, and integrals by the trapezoidal rule.

Since the problem is nonlinear, we need to employ an iterative method. To this end, we recast \eqref{eqn:FCH-fulldis-implicit} as a minimization problem.

\begin{prop}[minimization problem]
	Let $\gm$ be  given by \eqref{obj:grid-domain}. Then, $u_N\in\gridfn$ solves the discrete PDE \eqref{eqn:FCH-fulldis-implicit} if and only if it is a critical point of the objective
	\begin{align}
		G_N(v_N)=\frac{1}{2Ma_n}\norm{a_n v_N+b_n u_N^n+c_n u_N^{n-1}}{-1,N}^2 +  {\widetilde{\mathcal{E}}_{N}}(\breve v_N),
		\label{obj:FCH-mini-objv}
	\end{align}
	where 
	\begin{align}
		 {\widetilde{\mathcal{E}}_{N}}(\breve v_N) = 
		 	\begin{cases}
		 		\mathcal E_{N}(v_N) & \text{if BDF2 is used},
		 		\\
		 		2 \mathcal E_{N}( \frac{v_N+u_N^{n}}{2}) & \text{if MP is used}.
		 	\end{cases}
	\end{align}
\end{prop}
\begin{proof}
	First, if $v_N$ solves \eqref{eqn:FCH-fulldis-implicit}, the inverse Laplacian of $a_n v_N+b_n u_N^n+c_n u_N^{n-1}$ is well--defined since $v_N$ has zero mean. To see this, we take the discrete inner product $(\cdot,1)_N$ on both sides of \eqref{eqn:FCH-fulldis-implicit} and  use \eqref{def:frac-Lap} to conclude that the discrete Laplacian of any periodic grid function must have zero mean.
	
	Next, an explicit calculation shows
	\begin{equation}
		\pairing{G_N'(v_N)}{w_N}=\frac{1}{M} \left( a_n v_N+b_n u_N^n +c_n u_N^{n-1}, w_N \right)_{-1,N} + \pairing{\frac{\delta {{\widetilde{\mathcal{E}}_{N}}}}{\delta u}(\breve v_N)}{w_N},
		\nonumber
	\end{equation}
or, in other words, 
	\begin{equation}
	G'_N(v_N)=\frac{1}{M}(-\Delta_N)^{-1} (a_n v_N +b_n u_N^n +c_n u_N^{n-1})+\frac{\delta { {\widetilde{\mathcal{E}}_{N}}}}{\delta u}(\breve v_N).
	\label{misc07}
	\end{equation}
The factor $2$ in the MP case comes from the chain rule.
\end{proof}

\begin{rmk}[convexity of $G_N$]
The objective functional $G_N$, given in \eqref{obj:FCH-mini-objv}, is not convex in general. For this reason, we speak of critical points rather than minimizers.\ermk
\end{rmk}

\begin{rmk}[discrete mass conservation]
\label{rmk:meanzero}
The iterative solvers that we consider use a slight variant of the gradient \eqref{misc07}. Namely, they use the mean--zero projection of intermediate grid functions to ensure mass conservation at the discrete level and to properly compute the discrete inverse Laplacian. More specifically, we use the negative of the following gradient as the residual
\begin{align}
	G_N'(v_N)&=\frac{1}{M} (-\Delta_{N})^{-1}[\Pi_0 (a_n v_N+b_n u_N^n +c_n u_N^{n-1})] + \Pi_0\frac{\delta {\widetilde{\mathcal{E}}_{N}}}{\delta u}(\breve v_N). 
	\label{obj:G'}
\end{align}
The first mean--zero projection is not needed if we have an infinite precision since they must be mean--zero. However, due to  round--off error, we need it to keep the mass conservation at the discrete level. The second projection really changes the discrete chemical potential. However, as can be seen from \eqref{eqn:FCH-fulldis-implicit}, adding a constant to the chemical potential does not change the main unknown $u_N$, i.e., the discrete Laplacian annihilates the constant added to the discrete chemical potential. Moreover, the second mean--zero projection allows the inversion of the preconditioner (see below) to be well--defined, which involves the discrete inverse Laplacian. \ermk
\end{rmk}

\subsubsection{The averaged Newton preconditioner}
As the names indicate, the iterative solvers used in this work involve a preconditioner. There is no general way to construct a suitable preconditioner. For this reason, we make use the energy structure of the PFC and FCH models to develop what we call an \emph{averaged Newton preconditioner}. 

To present the idea without introducing unnecessary technicalities, consider the following problem. Let $m \in \NN$ and suppose that $G : \RR^m \to \RR$ is a smooth, convex objective functional with a unique minimizer $x=\argmin_{\tilde x \in \RR^m} G(\tilde x)$. To find it, our starting point is to view the Newton's method as a generalization of the preconditioned gradient descent method, where the preconditioner is the second variation of the objective. That is, 
\[x_{n+1} =   x_{n}  - G''(x_n)^{-1}G'(x_{n}).\] 
Now, the goal is to come up with a preconditioner $\mathcal{G}$ (that is independent of $n$) that resembles $G''(x_n)$ but is easier to invert. For that, we exploit the structure of the objective function $G$. We know that, necessarily, $G''(x_n)$ is a linear operator, but that it may depend on the entries of $x_n$. To make it even simpler to invert, we remove this dependence by replacing these by averaged quantities.

Let us now move on to our case of interest. The second variation for the PFC model reads
\begin{align}
	G_N''(v_N)w_N=\frac{a_n}{M}  (-\Delta_N)^{-1}w_N + (3u^2 + 1-\varepsilon)w_N +2\Delta_N w_N +\Delta_N^2 w_N ,
	\label{eql:PFC-2nd-var-str}
\end{align}
whereas the one for the FCH model is
\begin{align}
	G_N&''(v_N)w_N=\frac{a_n}{M}  (-\Delta_N)^{-1}w_N
	\nonumber
	\\
	&+\left[F''(v_N)^2 - \eta_2 F''(v_N)  -\varepsilon^2\Delta_N F''(v_N)-F'''(v_N)(\varepsilon^2\Delta_N v_N-F'(v_N))  \right]w_N   
	\nonumber
	\\
	& + (F''(v_N)\varepsilon^2-\eta_1\varepsilon^2)(-\Delta_N w_N)
	\nonumber
	\\
	&  +\varepsilon^4\Delta_N^2w_N.
	\nonumber
\end{align}   
Both have the form
\begin{align*}
	G_N''(v_N)w_N= \beta_{-2} (-\Delta_N)^{-1}w_N +\tilde \beta_{0} w_N +\tilde \beta_2(-\Delta_N) w_N +\beta_4\Delta_N^2 w_N,
\end{align*}
where the parameters $\beta_{-2}$ and $\beta_4$ are constants while $\tilde\beta_{0}$ and $\tilde\beta_{2}$ are functions of $v_N$ for the FCH equation. The same is true for the PFC model except $\tilde\beta_{2}$ is also a constant. Based on this observation, we consider a preconditioner of the same form, but where all the parameters are constants, that is,
\begin{align*}
	\LL w_N := \beta_{-2} (-\Delta_N)^{-1}w_N +\beta_{0} w_N +\beta_2(-\Delta_N) w_N +\beta_4\Delta_N^2 w_N.
\end{align*}
The question that remains then, is how to choose these constants. Our approach, for the PFC models, is that if the parameter is constant, then we keep the value, whereas for those that are variable we set them to be the absolute value of its average across the domain. For the FCH model, however, we have chosen to drop several terms from $\beta_0(v_N)$ when computing its average for practical purposes. More specifically, $-\varepsilon^2\Delta_N F''(v_N)$ and $- \varepsilon^2F'''(v_N) \Delta_N v_N$ are not included since their contributions are small (they involve $\varepsilon^2$) and to save computations (they involve computing a Laplacian). At every time--step, these constants are computed using the initial guess, and kept fixed throughout the iterative process. They are only recomputed once we advance in time.

\subsubsection{PGD and PAGD}

The PGD method is given in Algorithm \ref{alg:PGD}. This method works the same way as usual gradient descent methods to find a critical point of  \eqref{obj:FCH-mini-objv} except that, as mentioned above in Remark \ref{rmk:meanzero}, we take the mean--zero projection on some parts of the gradient when computing the residual and apply a preconditioner to get the search direction.

\begin{algorithm}[h]
	\SetAlgoLined
	\KwData{$x_0$, $s>0$, $\text{TOL}_{iter}>0$}
	\CommentSty{initial guess, step size, tolerance}\;
	\KwResult{$x_{\infty}$}
	\CommentSty{approximate solution/critical point/minimzer}\;
	$i=0$ \CommentSty{initialization}\;
	\While{$\|d_i\|_{\infty} < \mathrm{TOL}_{iter}$}{
		$r_i = -G_N'(x_i)$ \CommentSty{find the residual using \eqref{obj:G'}}\;
		$d_i=\LLinv (r_i)$ \CommentSty{find the search direction, i.e., solve $\LL d_i= r_i$}\;
		$x_{i+1}=x_{i} + s d_i$ \CommentSty{descent step}\;
		$i \gets i+1$\;
	}
	\caption{Preconditioned gradient descent method (PGD)}
	\label{alg:PGD}
\end{algorithm}

The PAGD is an accelerated version of the PGD. As explained in \citep{psw2021PAGD}, it requires an additional parameter $\eta$. This corresponds to the \emph{friction coefficient} of the rolling ball system associated to the PAGD. That is, the PAGD is nothing but a discretization of a second order ordinary differential equation (ODE) describing the motion of a ball as it rolls down to the bottom of  a well (the graph of the objective functional) in the presence of constant friction. For more details about this relation, we refer interested readers to \citep{psw2021PAGD}. As shown in \citep{psw2021PAGD}, the choice of friction parameter $\eta$ can be justified theoretically if the objective is strongly convex and the strong convexity constant is known. However, this is not the case for the PFC nor FCH equations and a different approach is needed. We could have tried to find an optimal value for the friction coefficient $\eta$ by trial and error. Instead, we propose the scheme given in Algorithm \ref{alg:PAGD}, which is a slight variant of the PAGD and we call the \emph{sweeping--$\eta$} or \emph{sweeping--friction} strategy. Here, the friction coefficient $\eta$, takes a different value from a prescribed list of values at every iteration.

We note that this strategy, upon choosing a reasonable range of values, works very well in practice. In fact, it is almost as efficient as the optimally tuned, fixed choice in terms of the number of iterations required to  reach a prescribed tolerance. Moreover, the efficiency is less sensitive to a change of the range than it is to that of the fixed value. Based on the convergence theory of the PAGD in \citep{psw2021PAGD}, an equally spaced range
starting from some small number (e.g., 0.1) ending at $1/\sqrt s$ is possible.  However, from our experience, the right endpoint can be slightly larger than $1/\sqrt s$. The intuition behind this can be explained by the rolling ball analogy. As opposed to the constant friction case, the sweeping--$\eta$ strategy corresponds to ``putting on the brake'' repeatedly, say, from softly to hard. If the ball is rolling down to the bottom in this manner, our experience from driving a vehicle suggests that it will effectively stop the ball near the bottom of the valley of the landscape, an analogy to converging to a local minimum. Finally, when the PAGD or PGD is used as a nonlinear solver for time marching equations, a good option for the initial guess is the extrapolation of the previous two histories, and in fact, this is adopted in the numerical experiments that follow. 

\begin{algorithm}[h]
	\SetAlgoLined
	\KwData{$x_0$, $x_{-1}:=x_0$, $s>0$, $(\eta_0, \eta_1, \ldots, \eta_{k-1})\in\RR^{k}$, $\text{TOL}_{iter}>0$\;}
	\CommentSty{initial guess, fictitious previous iterate, step size, range of sweeping-$\eta$  with $0<\eta_j$ ($0\le j\le k-1$), tolerance\;}
	\KwResult{$x_{\infty}$\;}
	\CommentSty{approximate local minimizer}\;
	$i=0$ \CommentSty{initialization}\;
	\While{$\|d_i\|_{\infty} < \text{TOL}_{iter}$}{
        $j \gets i \mod k$ \CommentSty{sweep over the possible choices of $\eta$}\;
		$\lambda_i = \frac{1-\eta_{j}\sqrt{s}}{1+\eta_{j}\sqrt{s}}$ \CommentSty{compute the extrapolation coefficient}\;
		$y_i= x_i + \lambda_i (x_i - x_{i-1})$ \CommentSty{compute the extrapolated position}\;
		$r_i = -G_N'(y_i)$ \CommentSty{find the residual using \eqref{obj:G'}}\;
		$d_i=\LLinv (r_i)$ \CommentSty{find the search direction, i.e., solve $\LL d_i= r_i$}\;
		$x_{i+1}=y_{i} + s d_i$ \CommentSty{gradient step}\;
		$i \gets i+1$\;
	}
	\caption{Preconditioned Nesterov's accelerated gradient descent method (PAGD) with sweeping--$\eta$.}
	\label{alg:PAGD}
\end{algorithm}

\begin{algorithm}[h]
	\SetAlgoLined
	\KwData{Model, $u^0$, $T$, time discretization, $\dt_{min}$, $\dt_{max}$, $\mathrm{TOL}$\;}
	\CommentSty{PFC or FCH; initial condition; final time; MP, BDF2, LMP, or LBDF2; minimum and maximum time step sizes; stepping tolderance\;}
	
	\KwResult{$(t_1, t_2, t_3,\ldots)$, $\left(u^{1}, u^{2}, u^{3},\ldots\right)$\;}
	\CommentSty{a sequence of times, a sequence of approximate solutions at the prescribed times\;}
	$n=0$, 
	$t_0=0$ \CommentSty{Initialization. Set initial time}\;
	$\tilde \dt_1$=$\dt_{min}$  \CommentSty{Initialization. Set step size}\;
	
	\While{$t_n<T$ }{
        $\tilde t_{n+1}=t_{n}+\tilde \dt_{n+1}$ \CommentSty{Tentative new time}\;
		$\tilde u^{n+1} $ \CommentSty{Tentative solution at $\tilde t_{n+1}$. The solution depends on the Model and Solver}\;
		$\hat u^{n+1}$ \CommentSty{The higher order solution: \eqref{fml:AM3-high-sol} or \eqref{fml:midAB2-high-sol}}\;
		$\mathrm{ERR}$
		\CommentSty{The error estimator, computed via \eqref{fml:AM3-err} or \eqref{fml:midAB2-err}}\;
		\If{$\mathrm{ERR}\le\mathrm{TOL}$ or $\tilde \dt_{n+1}\le\dt_{min}$}{
			\CommentSty{Copy history, and time advances}\;
			$\dt_{n+1}\gets\tilde\dt_{n+1}$\; 
			$t_{n+1} \gets t_{n}+\dt_{n+1}$\;
			$u^{n+1}\gets \tilde u^{n+1}$\;
			$n\gets n+1$\;
		}
		$\tilde \dt_{n+1}$ \CommentSty{Recompute using \eqref{fml:time-step} and \eqref{fml:time-step-med}}\;
	}
	\caption{Solvers. Since BDF2 and LBDF2 require two previous solutions, the backward Euler scheme or its semi--implicit version is used for the first time--step, respectively.}
	\label{alg:solver}
\end{algorithm}
 
\section{Numerical experiments}
\label{sec:numerics}

\subsection{Benchmark problems}
\label{sec:problems}

	Our numerical experiments can be divided into two parts: in the first part (Section \ref{sec:PAGDperform}), the performances of the PGD and PAGD are compared, while in the second part (Section \ref{sec:scheme-comparision}), comparisons of fully implicit schemes and semi--implicit schemes are made. For such comparisons, we have chosen five benchmark problems, where we measure the computational cost taken by each of the proposed solvers for each problem (see corresponding sections for details on how to measure the cost). For the first part, we have ten combinations: five problems are numerically solved by two solvers, PGD and PAGD, respectively. For the second part, we have twenty combinations: four different numerical solvers, i.e., fully implicit MP and BDF2 (both equipped with the PAGD iterative solver), and semi--implicit LMP and LBDF2, are used to solve five benchmark problems. Three of the problems are evolutions according to the FCH model with different combinations of initial condition and parameters of the model, which we refer to as FCH1, FCH2, and FCH3, respectively, and the other two are according to the PFC model.

	The details of the initial conditions of FCH1---3 are found in \citep[section 5.2 (FCH1), 5.3 (FCH2), and 5.5.1 (FCH3)]{zhang2020benchmark} and the significance of evolutions similar to FCH2 and FCH3 is studied in \citep{Doelman2014}. The plots of these three are displayed in Figure \ref{figsub:FCH1-1} (FCH1), Figure \ref{figsub:FCH2-1} (FCH2), and Figure \ref{figsub:FCH3-1} (FCH3).
	
	The other two benchmark problems are simulations of crack propagation in a crystal strip system and that of crystal growth in a supercooled liquid using the PFC model for both, which we refer to as PFC1 and PFC2, respectively. Similar computations are common in the literature, since they vividly illustrate interesting physical phenomena and highlight the versatility of the model: by tuning some model parameters this single model is able to describe many different experimentally observed states and transformations. For more discussions of these simulations, see \citep{elder2002, elder2004, gomez2012PFCnumerical, provatas2007, emmerichPFC, provatas10}. 
	
	In the following two paragraphs, we detail the initial conditions of PFC1---2 since they are not exactly the same as in the literature that they are based on. The initial condition of PFC1 (see Figure \ref{figsub:PFC1-1} for its plot) is set by 
\begin{equation}
	\begin{split}
		u_0(x, y)=\phi_{s} &\cdot \psi(x, y)+\phi_{\ell}(1-\psi(x, y))
		\\
		&+A \cdot \psi(x, y) \cdot\left[\cos \left(q_{t} \frac{x}{1.2}\right) \cos \left(\frac{q_{t} y}{\sqrt{3}}\right)-\frac{1}{2} \cos \left(\frac{2 q_{t} y}{\sqrt{3}}\right)\right],
	\end{split}
	\label{obj:IC-PFC1}
\end{equation} 
where $A$ and $q_t$ are constants needed to describe the steady state density field of the solid type (see \citep[II.  C]{elder2004}) and set to 
\begin{equation}
	A=\frac{4}{5} \phi_{s}+\frac{4}{15} \sqrt{15 \epsilon-36 \phi_{s}^{2}}, \quad q_{t}=\frac{\sqrt{3}}{2}.
\end{equation}
Following \citep[III. D. 2]{elder2004}, we set $\phi_s = 0.49$, and $\phi_\ell=0.79$. $\phi_s$ and $\phi_\ell$ represent the temporal average of the number density of atoms for the solid and liquid state of a crystal strip system in a liquid bath, respectively. The function $\psi$ is a smoothed Heavyside function defined by 
\begin{equation}
	\begin{split}
		\psi(x, y):=&\left(\frac{1}{2}-\frac{1}{2} \tanh \left(\frac{|y|-\gamma_{1}}{4}\right) \right)
		\\
		&\quad\cdot\left(\frac{1}{2}+\frac{1}{2} \tanh \left( \frac{\sqrt{\left(x-x_{0}\right)^{2}+\left(y-y_{0}\right)^{2}}-\gamma_{2}}{4} \right)\right).
	\end{split}
\end{equation}
The function $\psi$ is used to create the chip (or a hole) on the strip as well as where the liquid regions are. The parameters $x_0$ and $y_0$ determine the location of the chip, $\gamma_{2}$ how big it is, and $\gamma_{1}$ the liquid regions. 

Lastly, the initial condition for the crystal growth simulation (PFC2) used for our experiments is a miniature version of the one implemented in \citep[Section 4.1]{gomez2012PFCnumerical}. The simulation of the whole domain, i.e., the same one as in \citep{gomez2012PFCnumerical} except the locations and sizes of crystallites, is reproduced in Figures \ref{figsub:PFC2big-1}--\ref{figsub:PFC2big-5}, and the PFC2 simulation considered here is displayed in Figures \ref{figsub:PFC2-1}--\ref{figsub:PFC2-4}. The detailed setting of PFC2 is the same as in \citep[Section 4.1]{gomez2012PFCnumerical} except that the spatial domain and the final time are reduced to $[-L/2,L/2]^2$ with $L=200$ and $T=300$, respectively. Also, to smooth out the initial condition, it is filtered using a Gaussian filtering that is used in \citep[p.~15]{zhang2020benchmark}, which is used also for FCH2---3. In the case of PFC2, however, an eight--times--finer resolution is used $\hat N = 8N$ while $\hat N = 2N$ for FCH2---3, where $N$ is the original resolution set for the main experiment. See \citep[p.~15]{zhang2020benchmark} for the details about the filtering.

The snap shots of the evolutions displayed in Figure \ref{fig:evol1}---\ref{fig:evol5}
are generated using the best performing scheme in terms of CPU time when applied to the experiment conducted in Section \ref{sec:scheme-comparision}. However, all solvers produce visually the same evolution with their differences detected only through numerical errors. 

To help put things into perspective about overall range of parameters in pursuit of accuracy (e.g., time step sizes, time stepping tolerance, etc.), Figure \ref{fig:FCH3CmpByResln} illustrates how a milder time stepping restriction (BDF2 with stepping tolerance $1.3\times10^{-4}$) leads to a different evolution in FCH3 than the one achieving the 5--digit objective (BDF2 with stepping tolerance $10^{-6}$; see Section \ref{sec:scheme-comparision} for details on 5--digit objective): the connectivity of the level curves of the solutions at $t=18$ is different.  When the tolerance is set to $10^{-6}$, the maximum time step size that is actually used by the algorithm is 0.1028 while when the tolerance is $1.3\times10^{-4}$, it hits the maximum time step size 0.5. And it becomes as large as 0.5482 when the maximum time step size is to 1. If the time stepping tolerance is slightly larger, say $1.5\times10^{-4}$ (with the maximum time step size being 1), near $t=10.1561$, our algorithm does not reach the iteration tolerance before the maximum number of iterations, which is set to 1000.

\begin{figure}
	\centering
	\subfloat[t=0]{\label{figsub:FCH1-1} \includegraphics[width=0.5\linewidth]{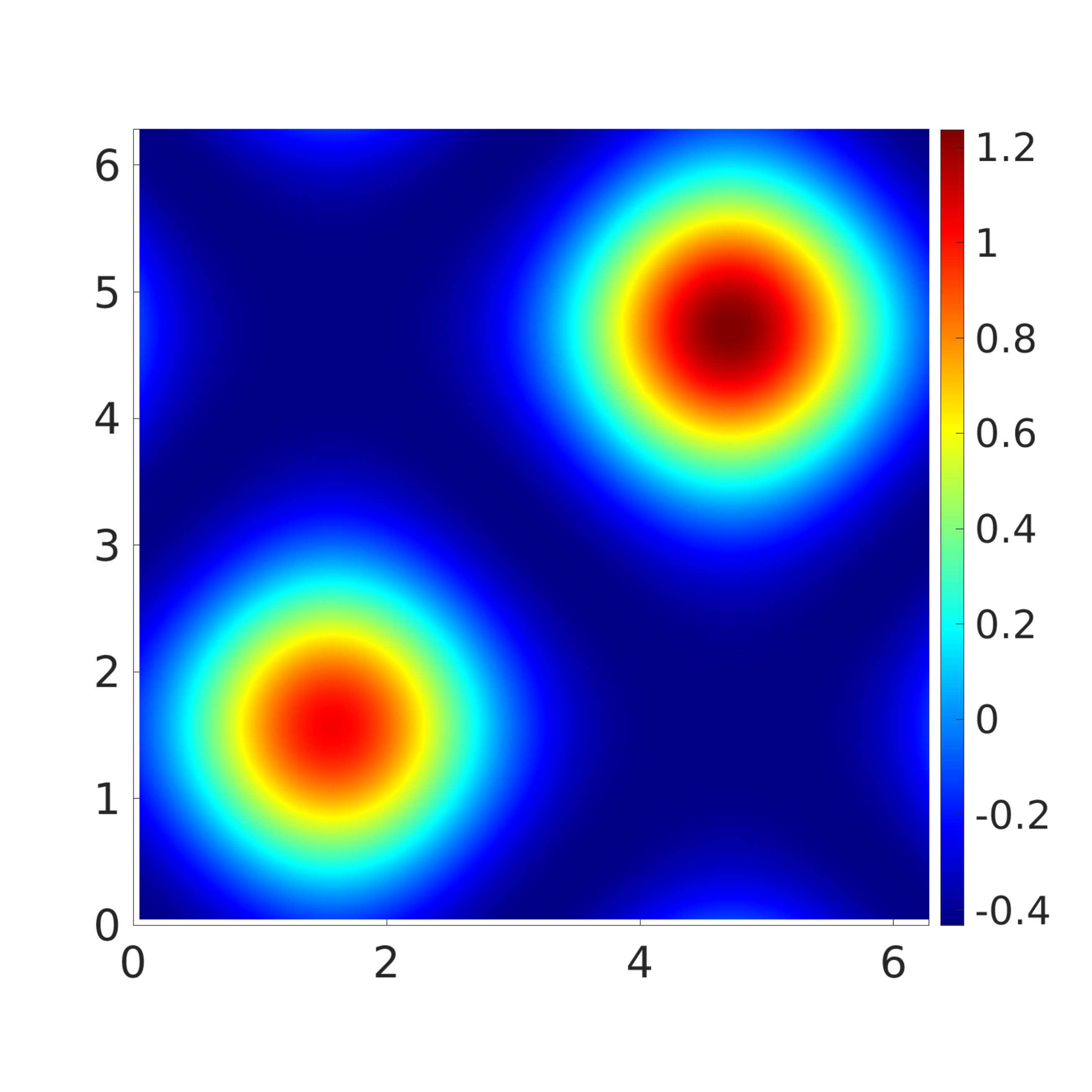}} 
	\subfloat[t=2]{\label{figsub:FCH1-2} \includegraphics[width=0.5\linewidth]{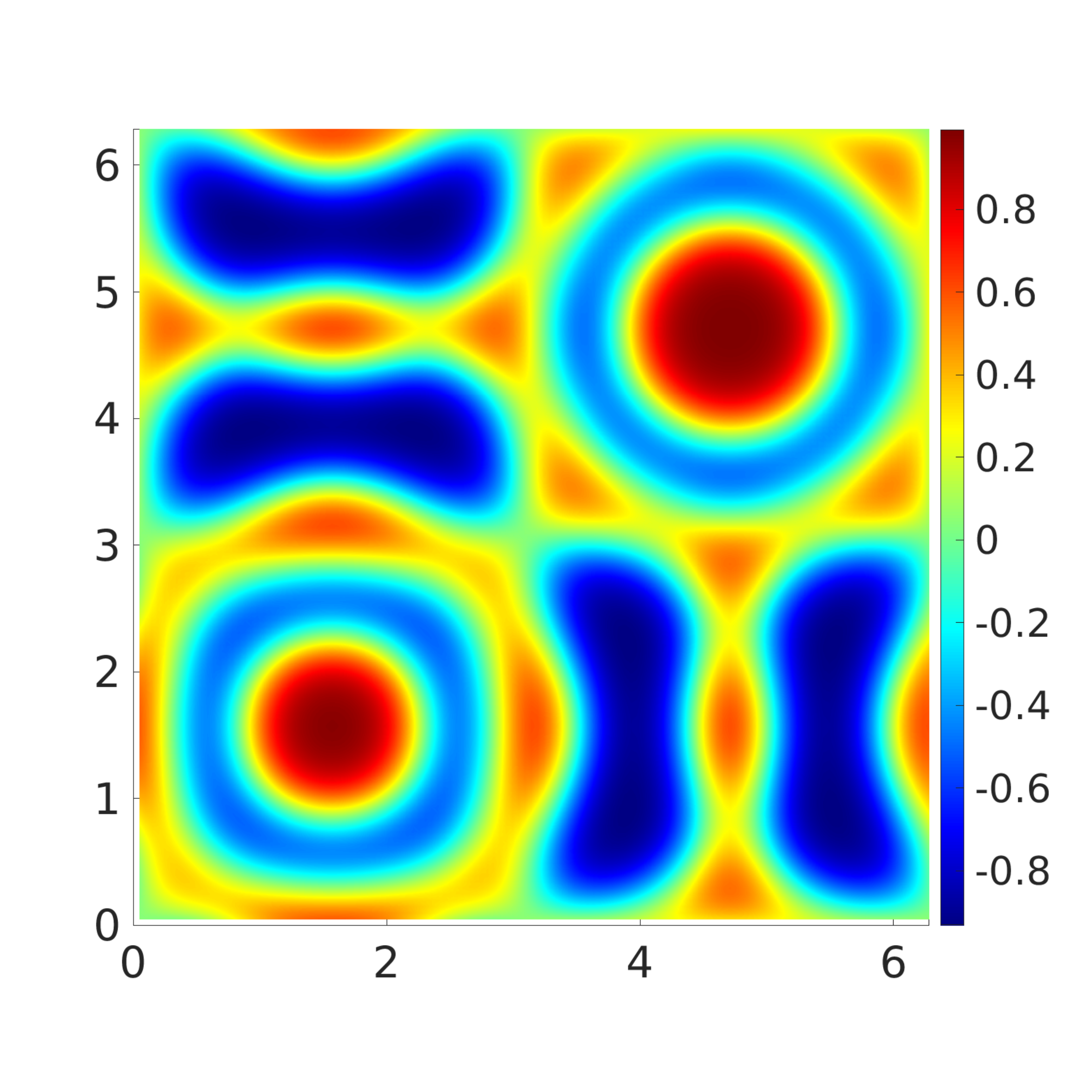}}
	\\
	\subfloat[t=10]{\label{figsub:FCH1-3} \includegraphics[width=0.5\linewidth]{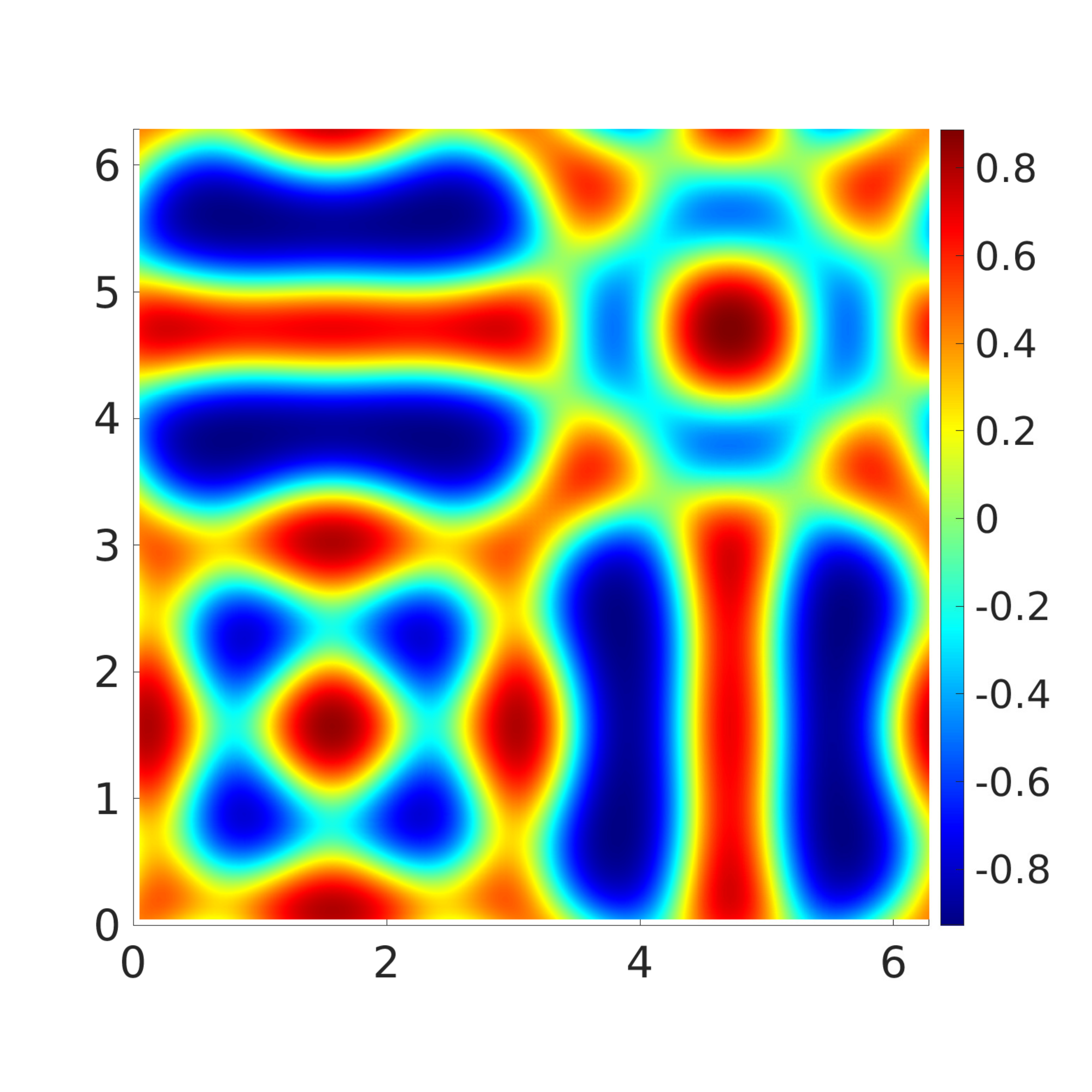}} 
	\subfloat[t=100]{\label{figsub:FCH1-4} \includegraphics[width=0.5\linewidth]{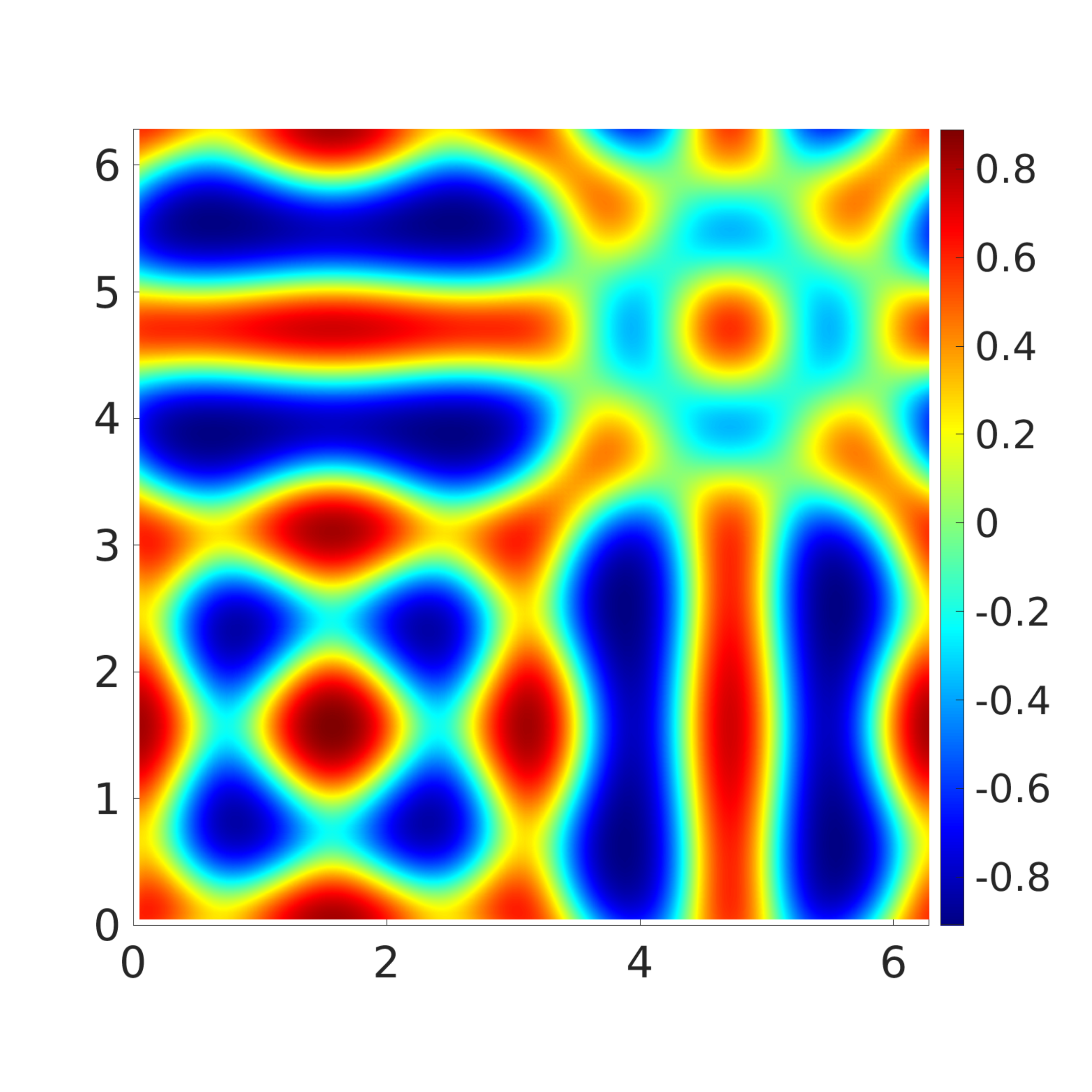}}
	\caption{FCH1 evolution.}
	\label{fig:evol1}
\end{figure}

\begin{figure}
	\centering
	\subfloat[t=0]{\label{figsub:FCH2-1} \includegraphics[width=0.5\linewidth]{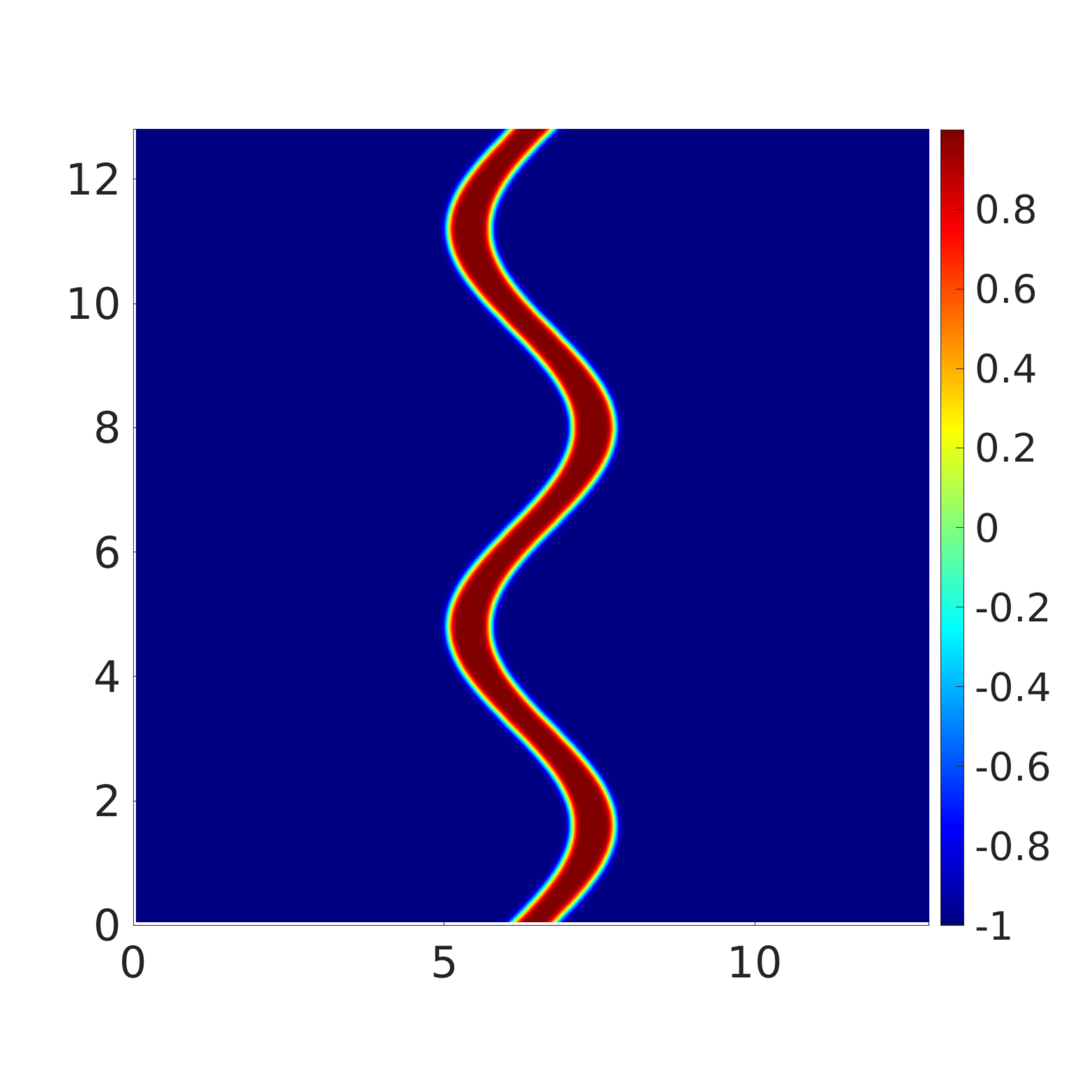}} 
	\subfloat[t=2]{\label{figsub:FCH2-2} \includegraphics[width=0.5\linewidth]{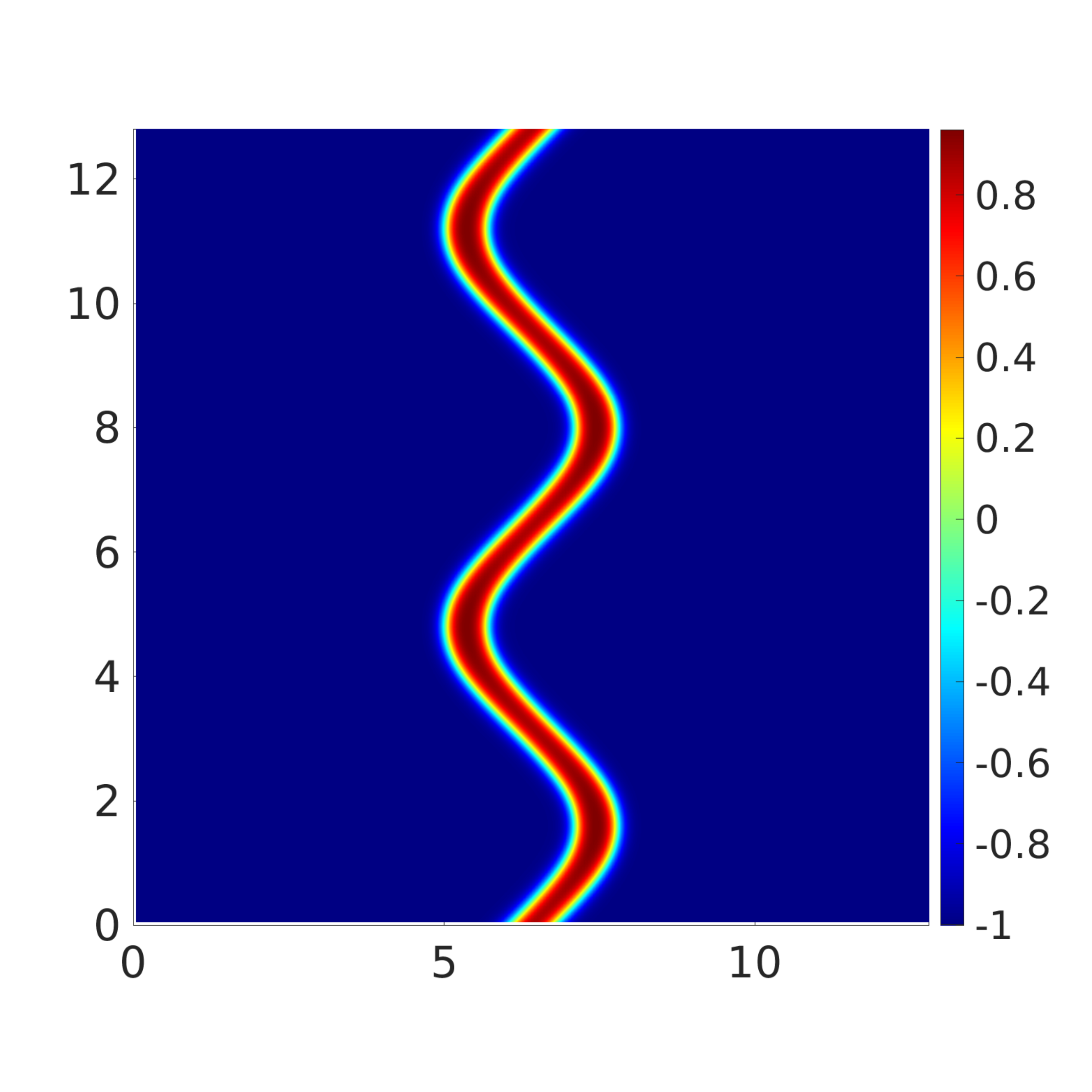}}
	\\
	\subfloat[t=10]{\label{figsub:FCH2-3} \includegraphics[width=0.5\linewidth]{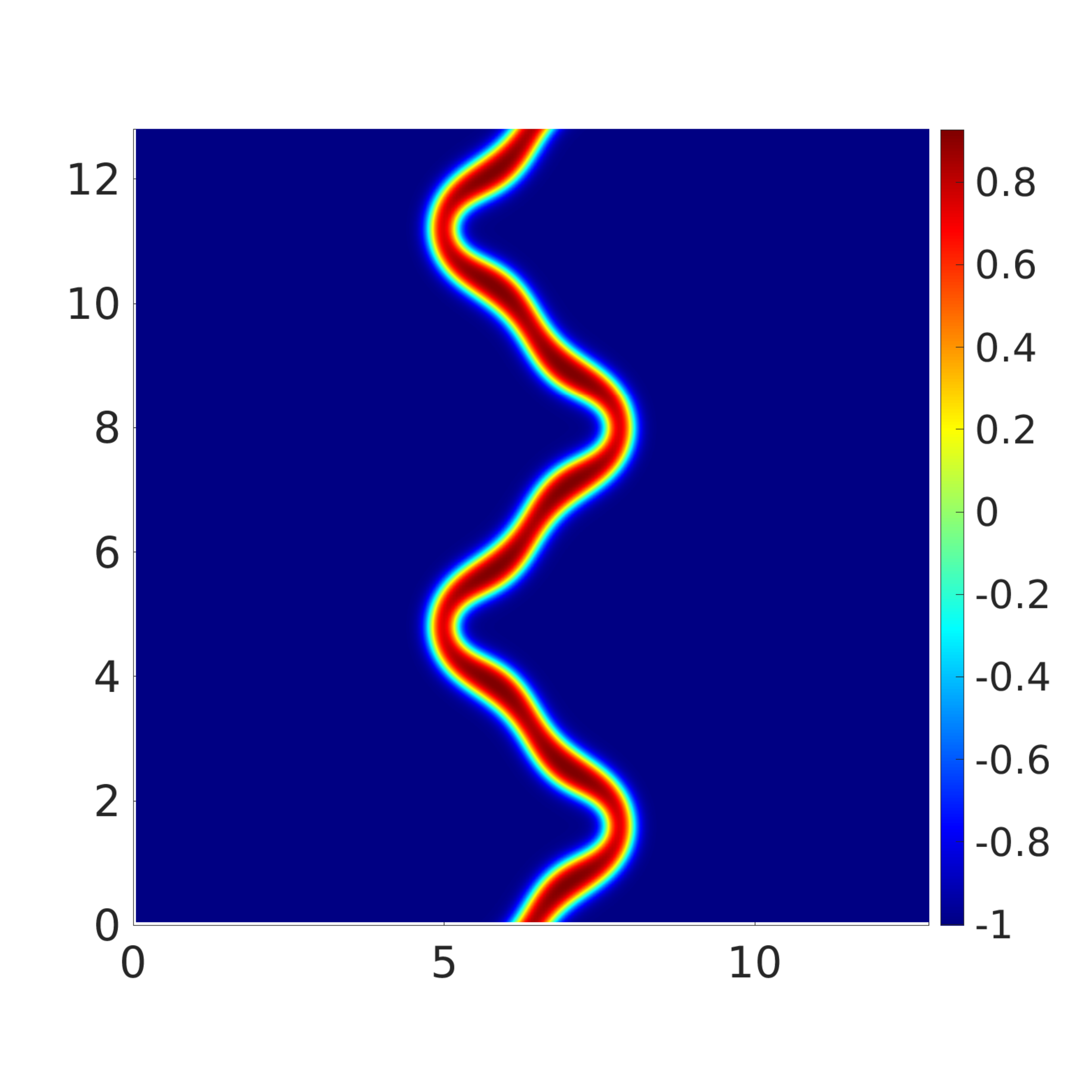}} 
	\subfloat[t=100]{\label{figsub:FCH2-4} \includegraphics[width=0.5\linewidth]{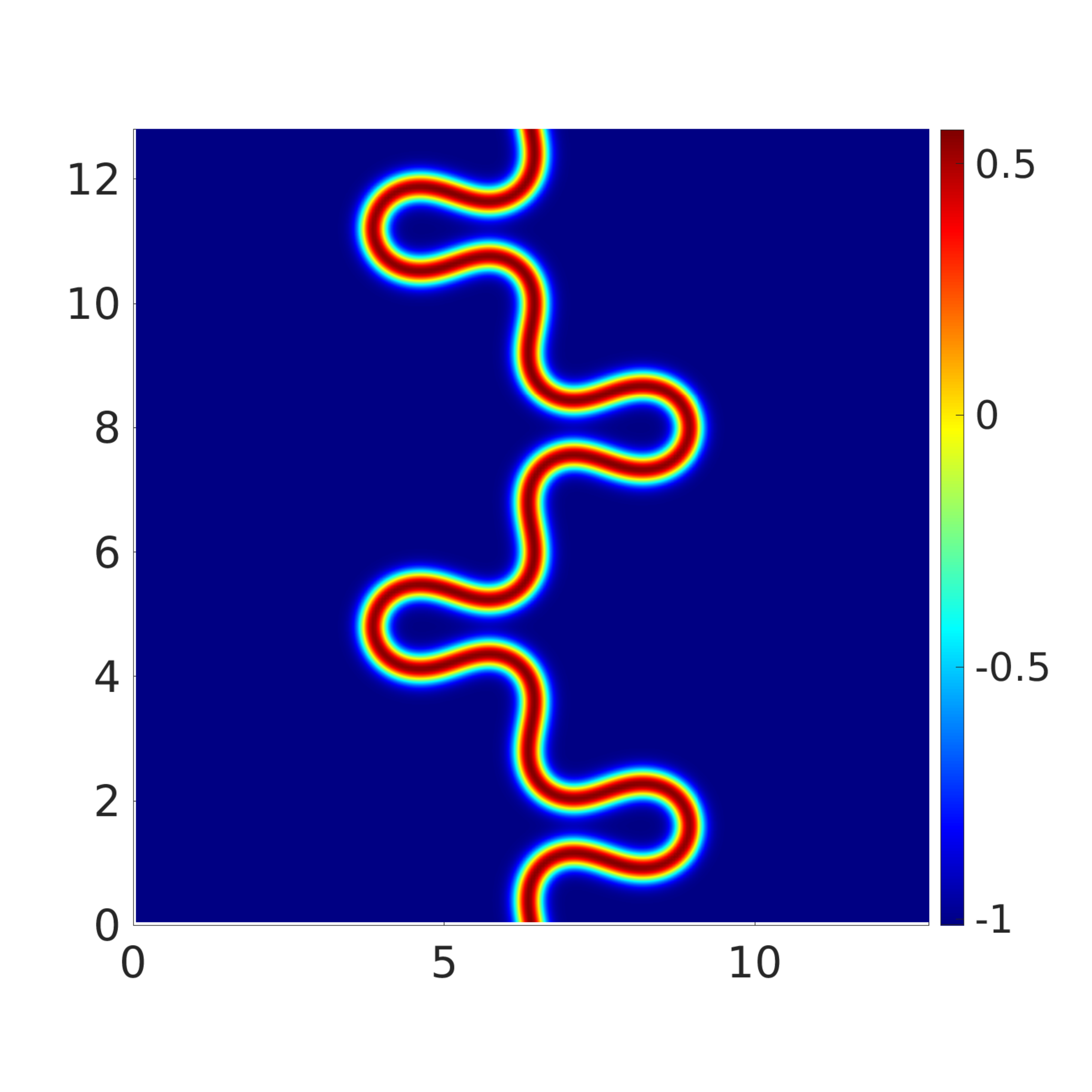}}
	\caption{FCH2 evolution.}
	\label{fig:evol2}
\end{figure}

\begin{figure}
	\centering
	\subfloat[t=0]{\label{figsub:FCH3-1} \includegraphics[width=0.5\linewidth]{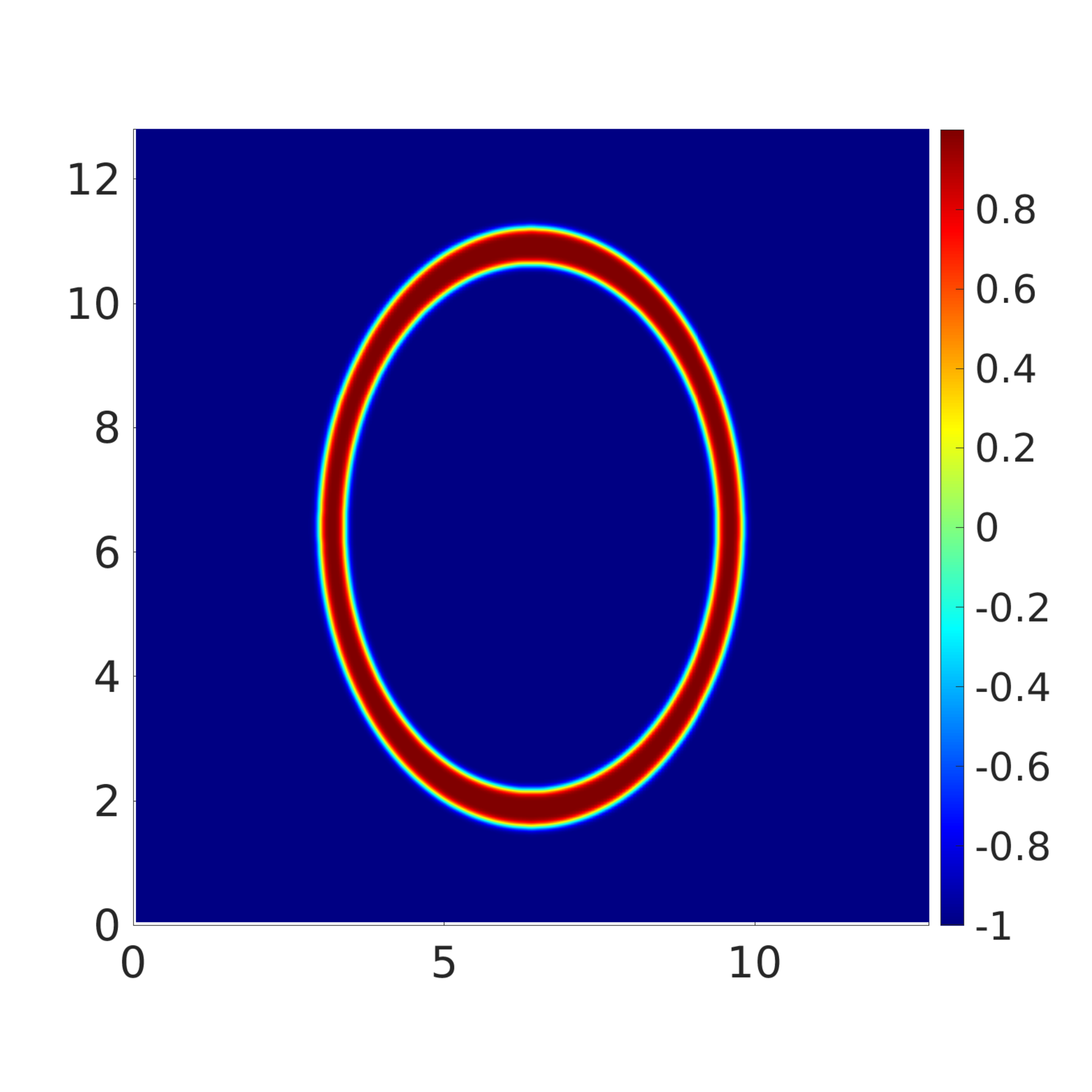}} 
	\subfloat[t=2]{\label{figsub:FCH3-2} \includegraphics[width=0.5\linewidth]{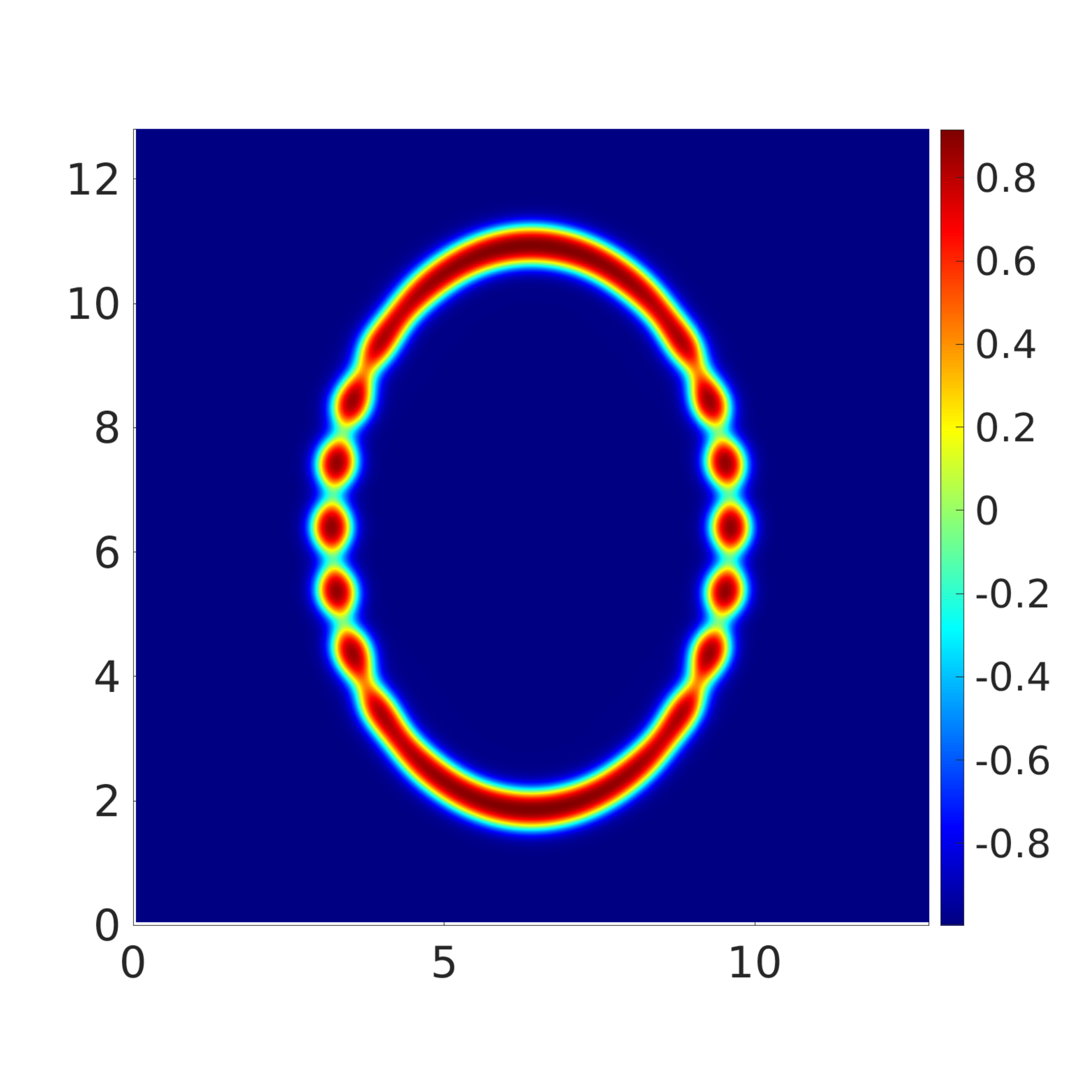}}
	\\
	\subfloat[t=10]{\label{figsub:FCH3-3} \includegraphics[width=0.5\linewidth]{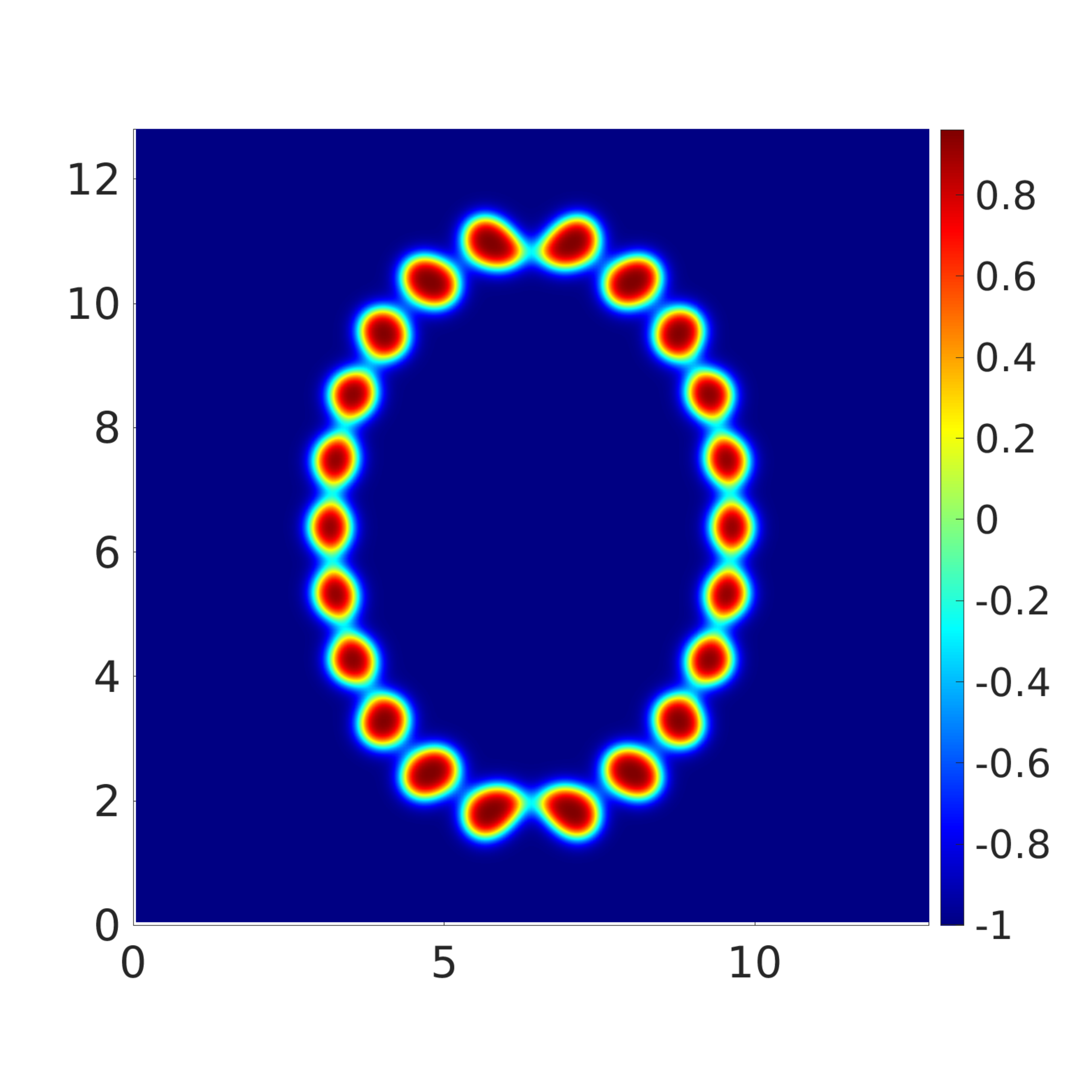}} 
	\subfloat[t=100]{\label{figsub:FCH3-4} \includegraphics[width=0.5\linewidth]{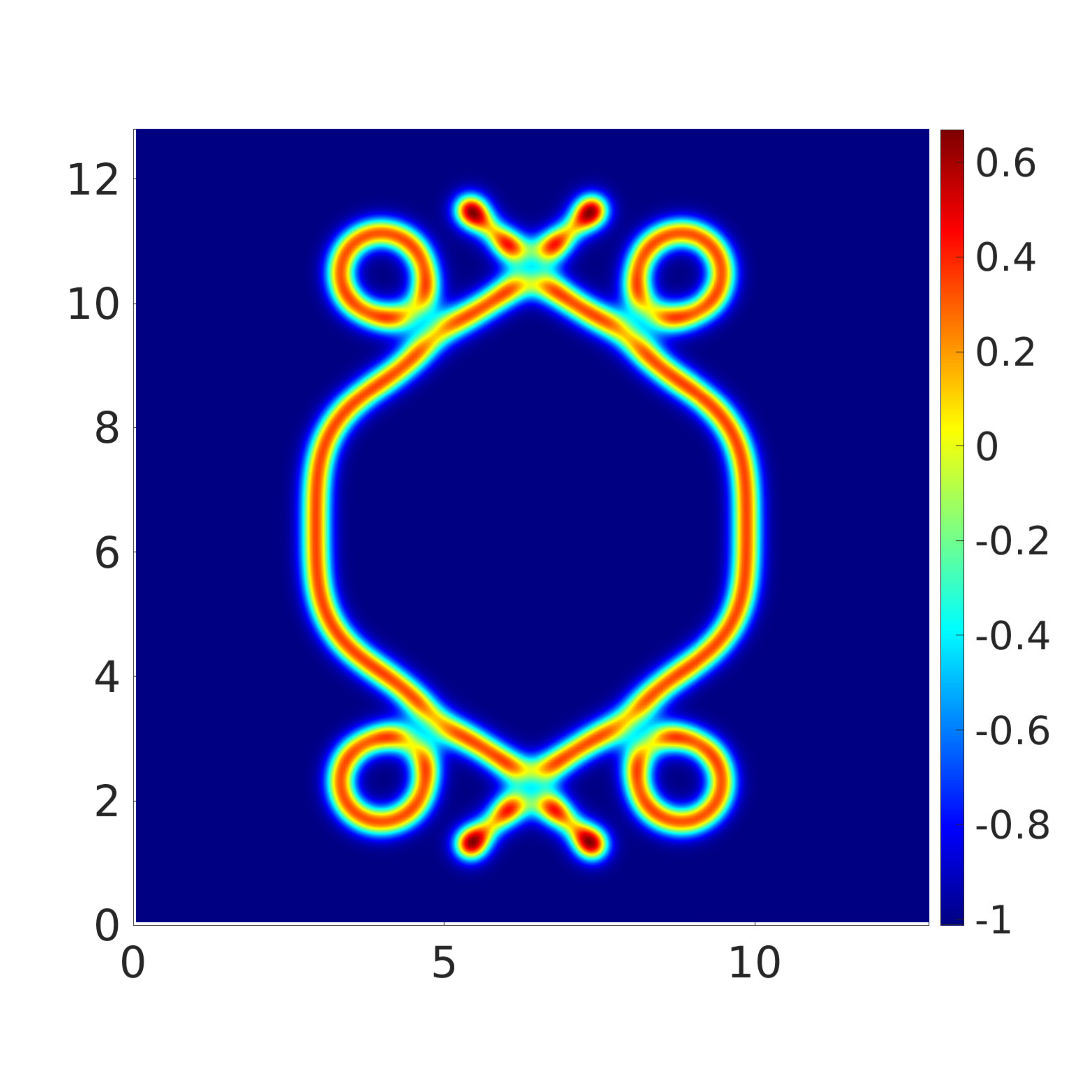}}
	\caption{FCH3 evolution.}
	\label{fig:evol3}
\end{figure}

\begin{figure}
	\centering
	\subfloat[t=0]{\label{figsub:PFC1-1} \includegraphics[width=0.44\linewidth]{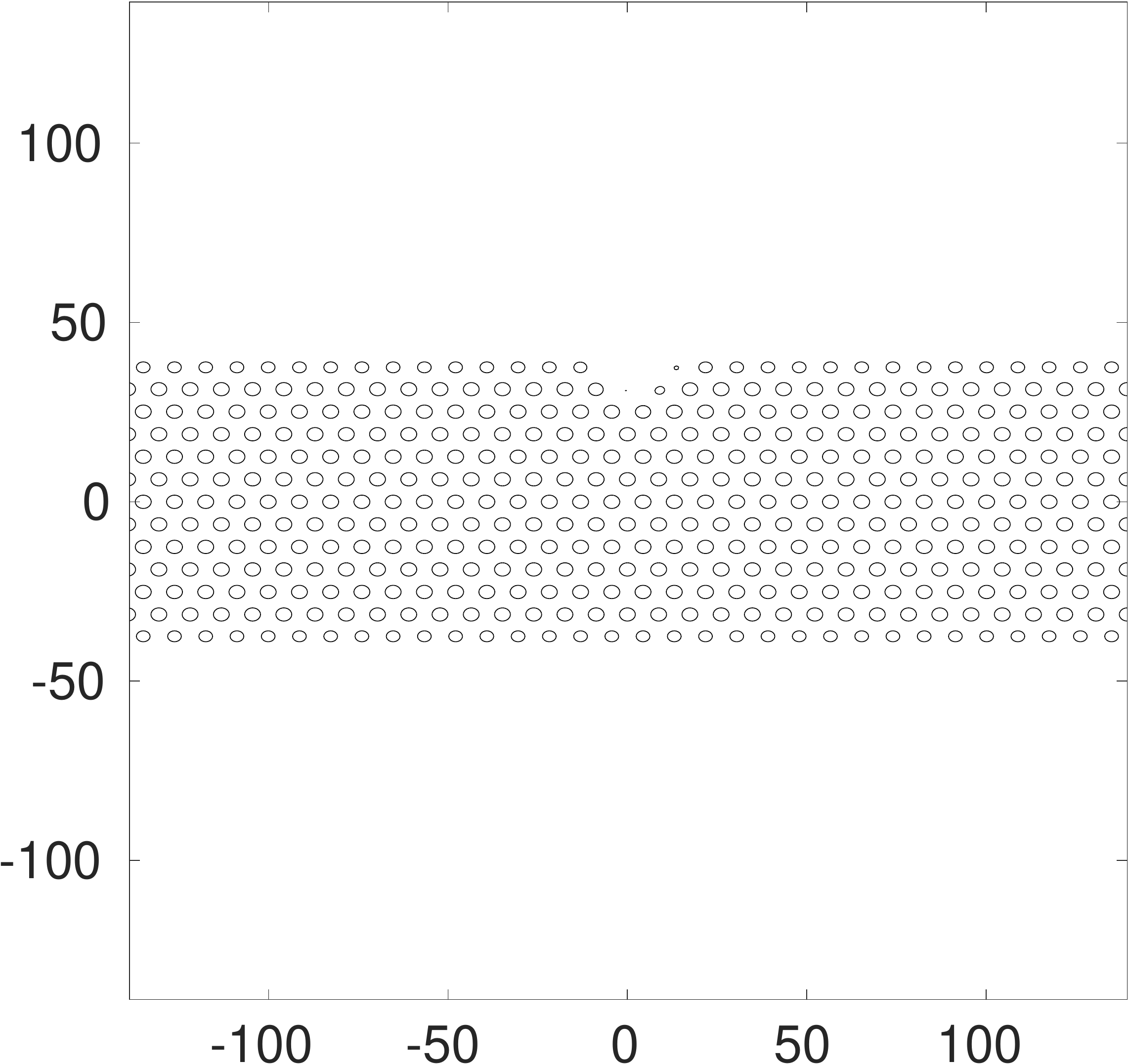}} 
	\subfloat[t=200]{\label{figsub:PFC1-2} \includegraphics[width=0.44\linewidth]{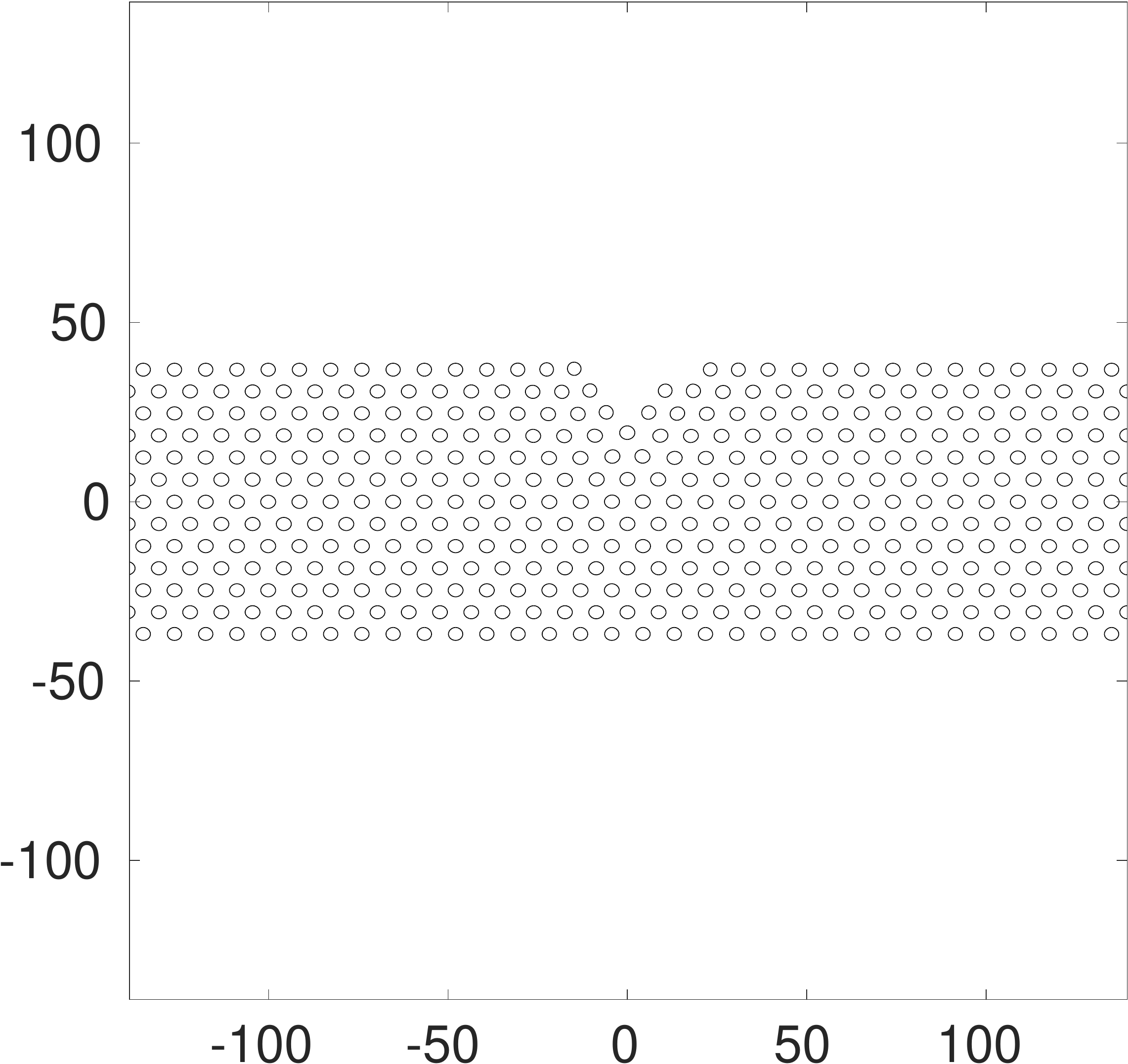}}
	\\
	\subfloat[t=1000]{\label{figsub:PFC1-3} \includegraphics[width=0.44\linewidth]{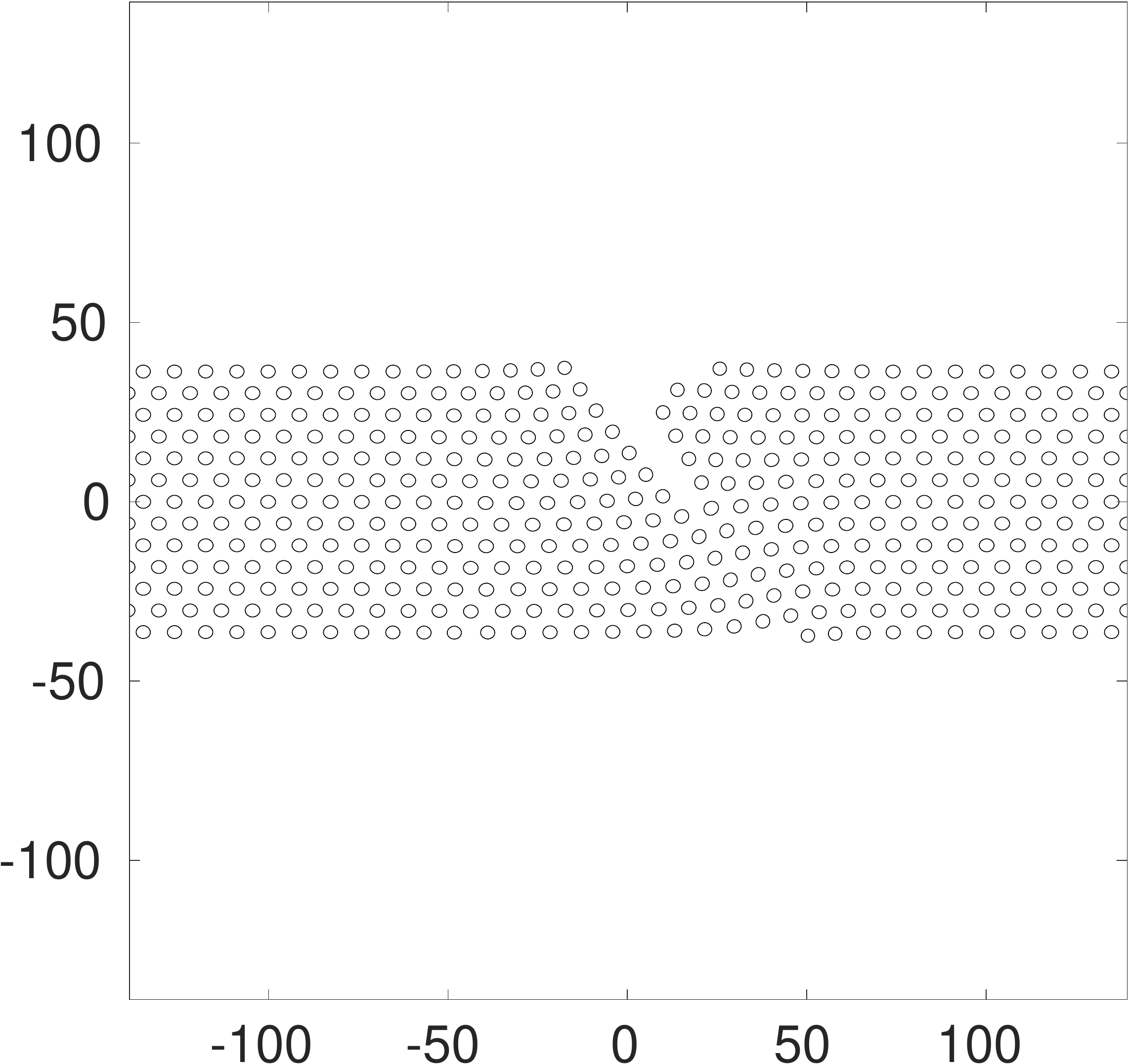}} 
	\subfloat[t=3000]{\label{figsub:PFC1-4} \includegraphics[width=0.44\linewidth]{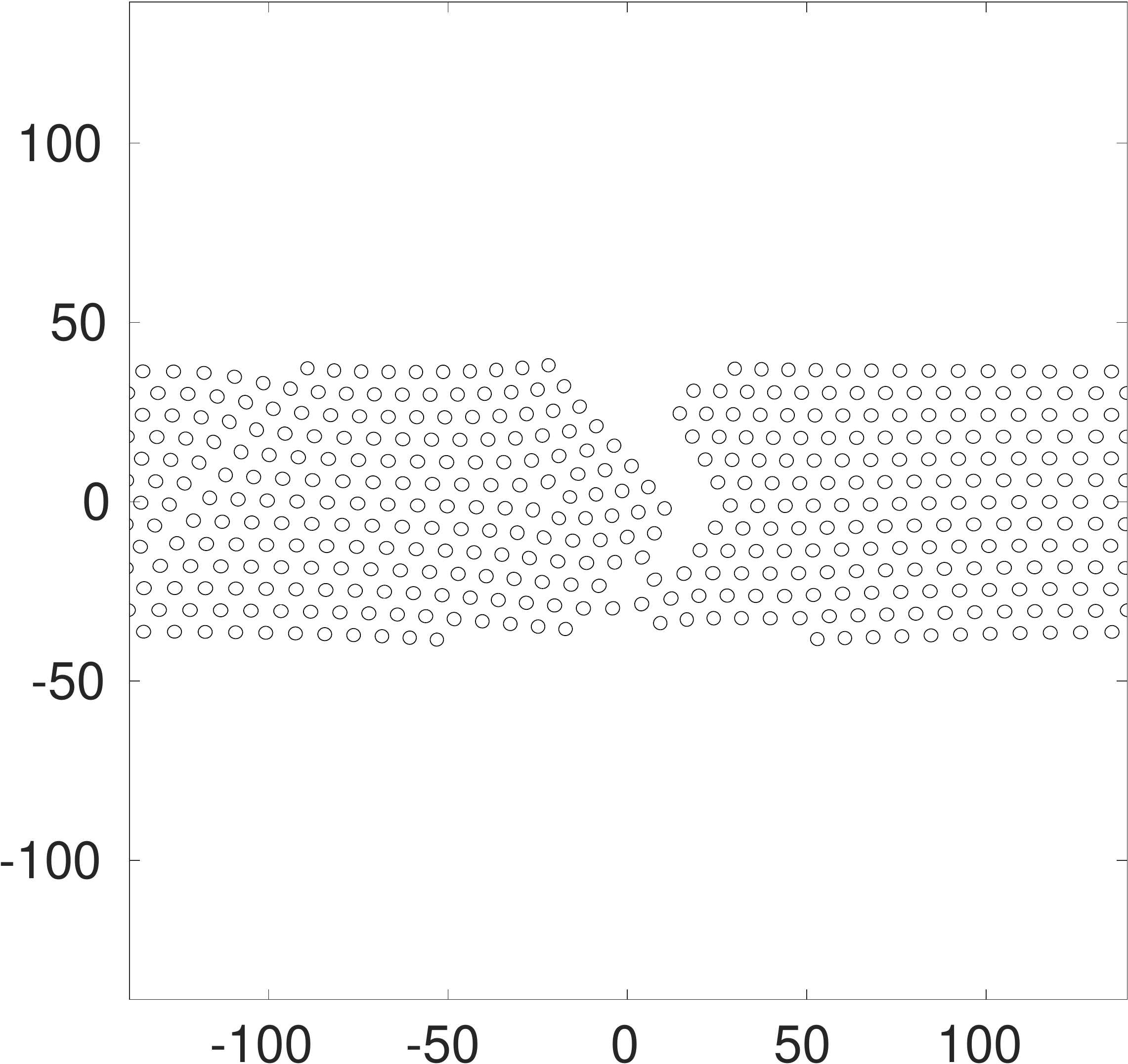}}
	\\
	\subfloat[t=10000]{\label{figsub:PFC1-5} \includegraphics[width=0.44\linewidth]{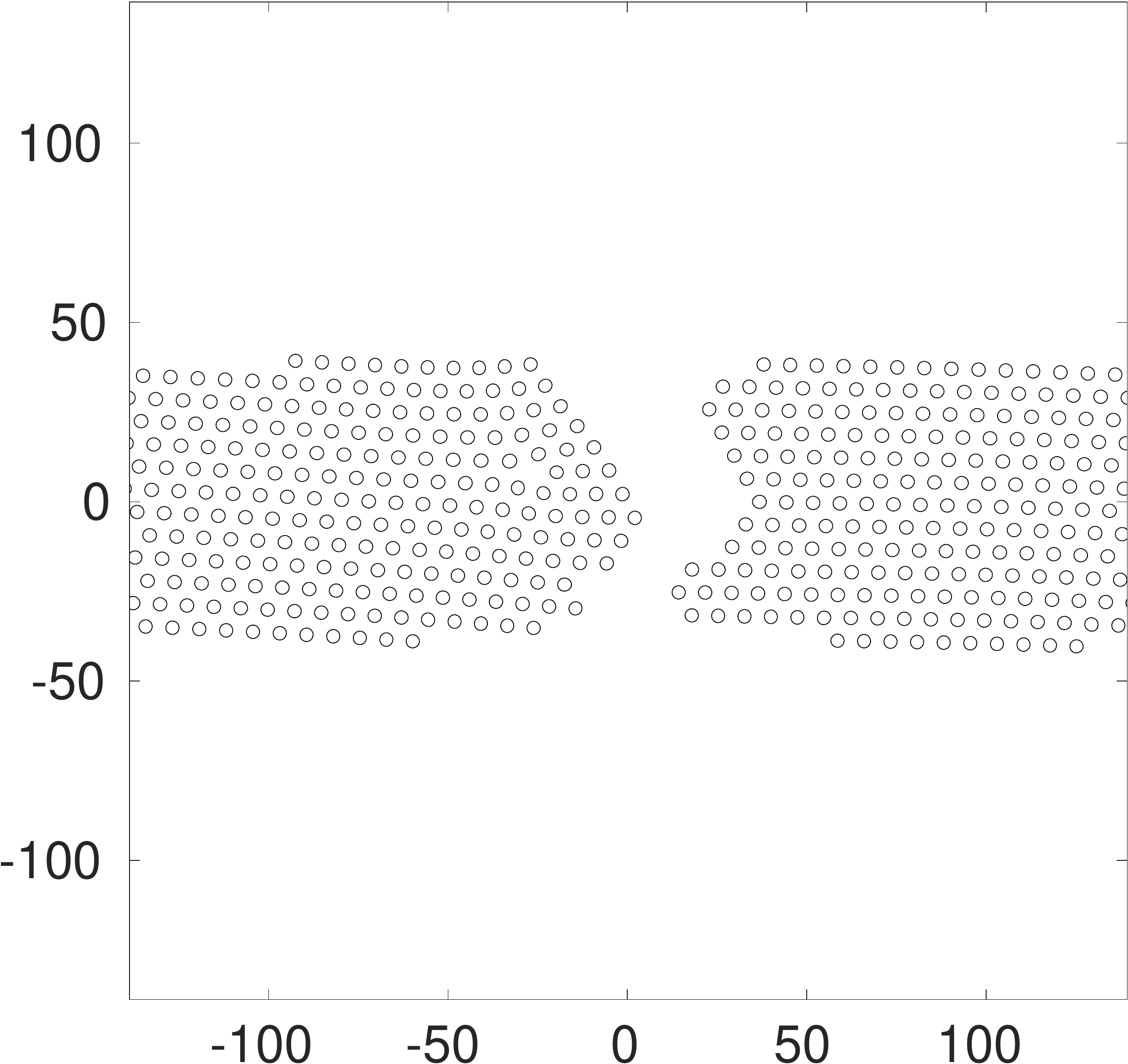}} 
	\caption{PFC1 evolution: zero level curves of the phase variable.}
	\label{fig:evol4}
\end{figure}

\begin{figure}
	\centering
	\subfloat[t=0]{\label{figsub:PFC2-1} \includegraphics[width=0.5\linewidth]{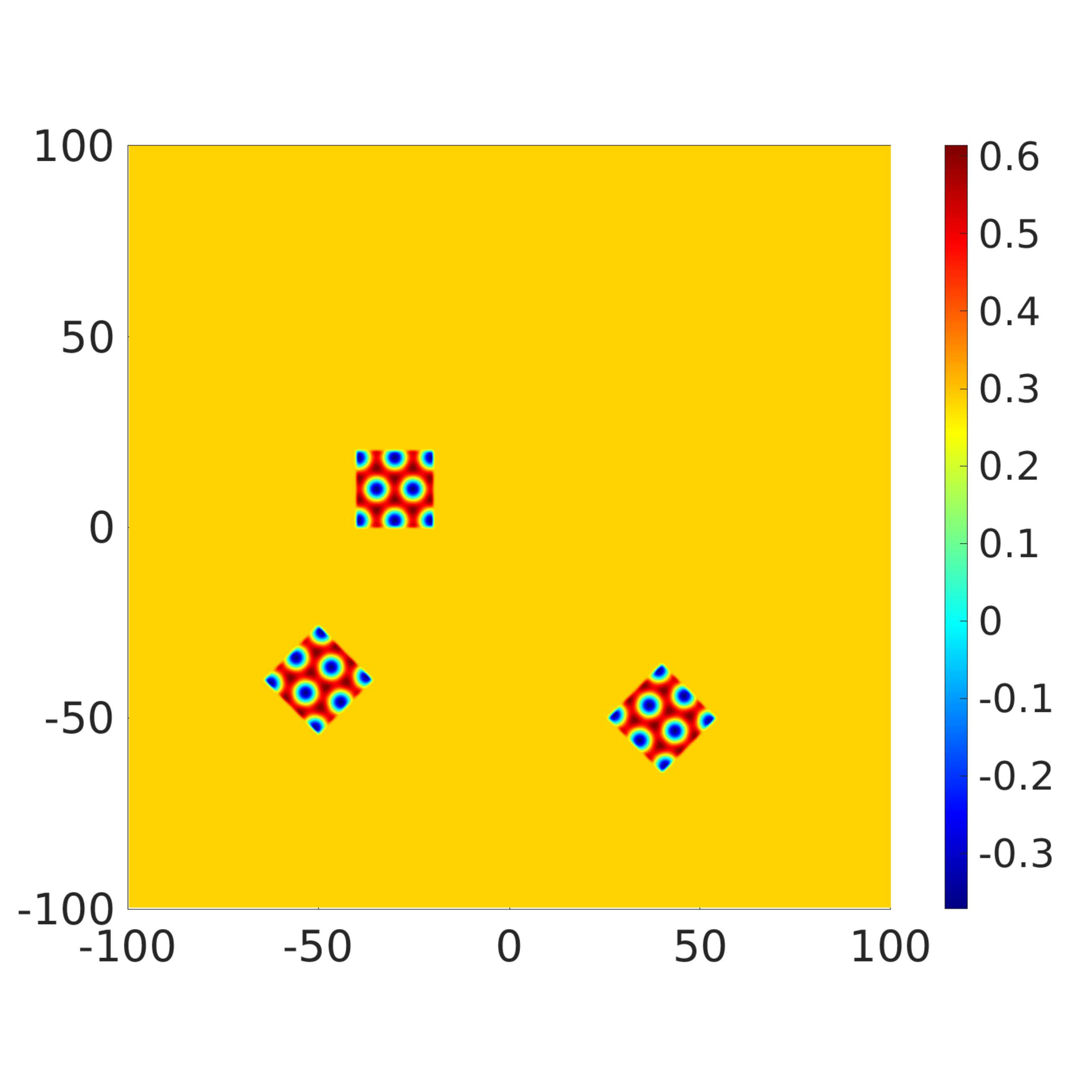}} 
\subfloat[t=12]{\label{figsub:PFC2-2} \includegraphics[width=0.5\linewidth]{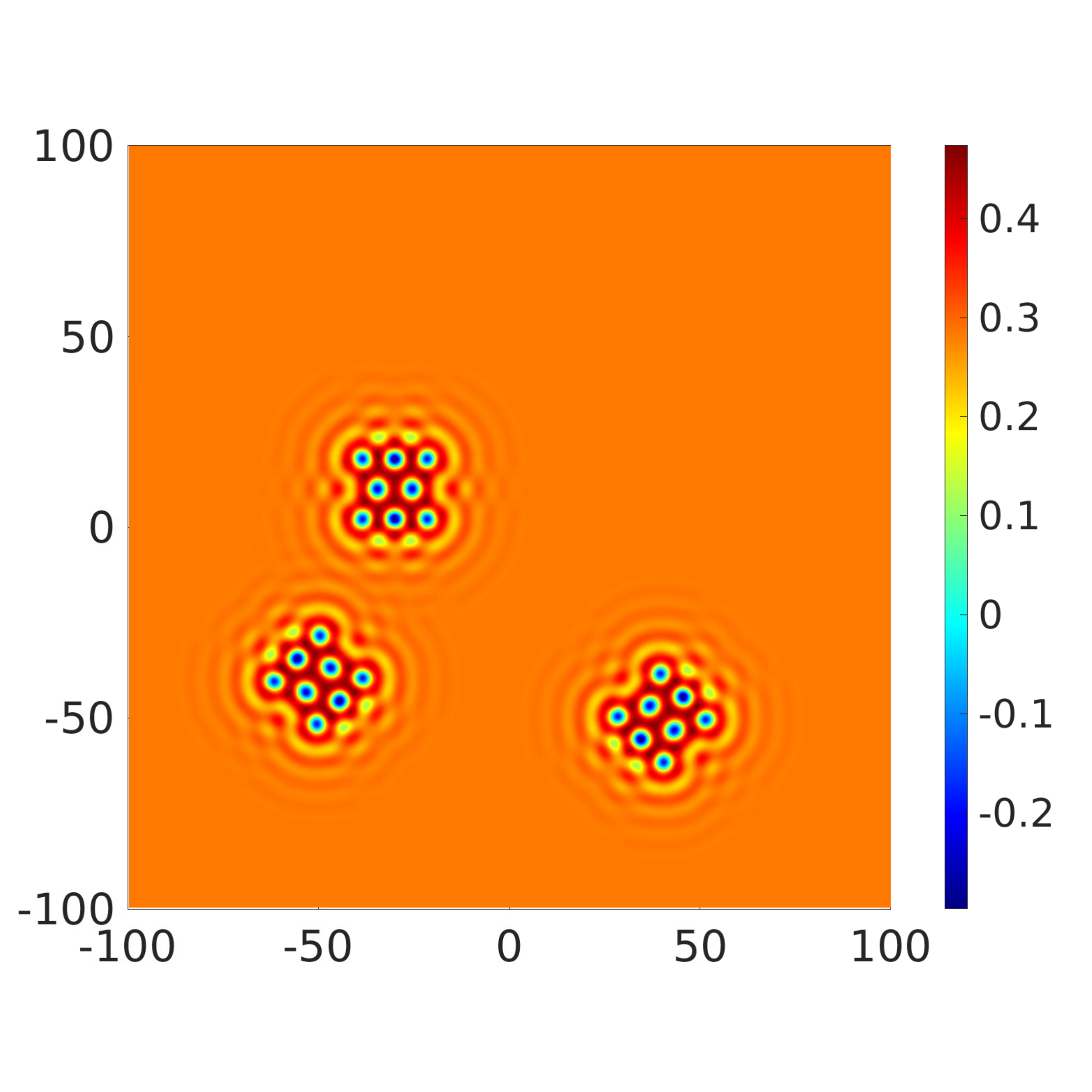}}
\\
\subfloat[t=30]{\label{figsub:PFC2-3} \includegraphics[width=0.5\linewidth]{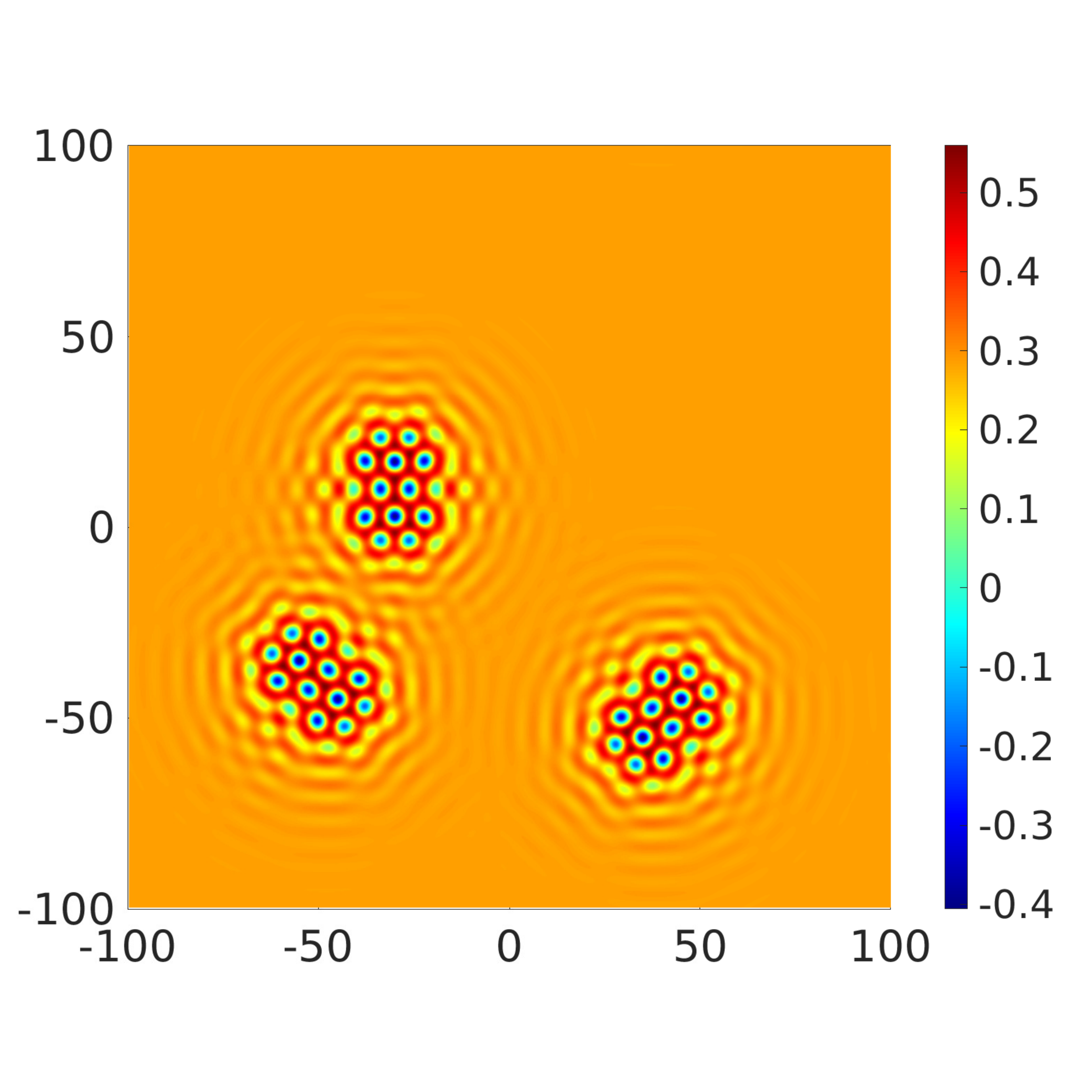}} 
\subfloat[t=300]{\label{figsub:PFC2-4} \includegraphics[width=0.5\linewidth]{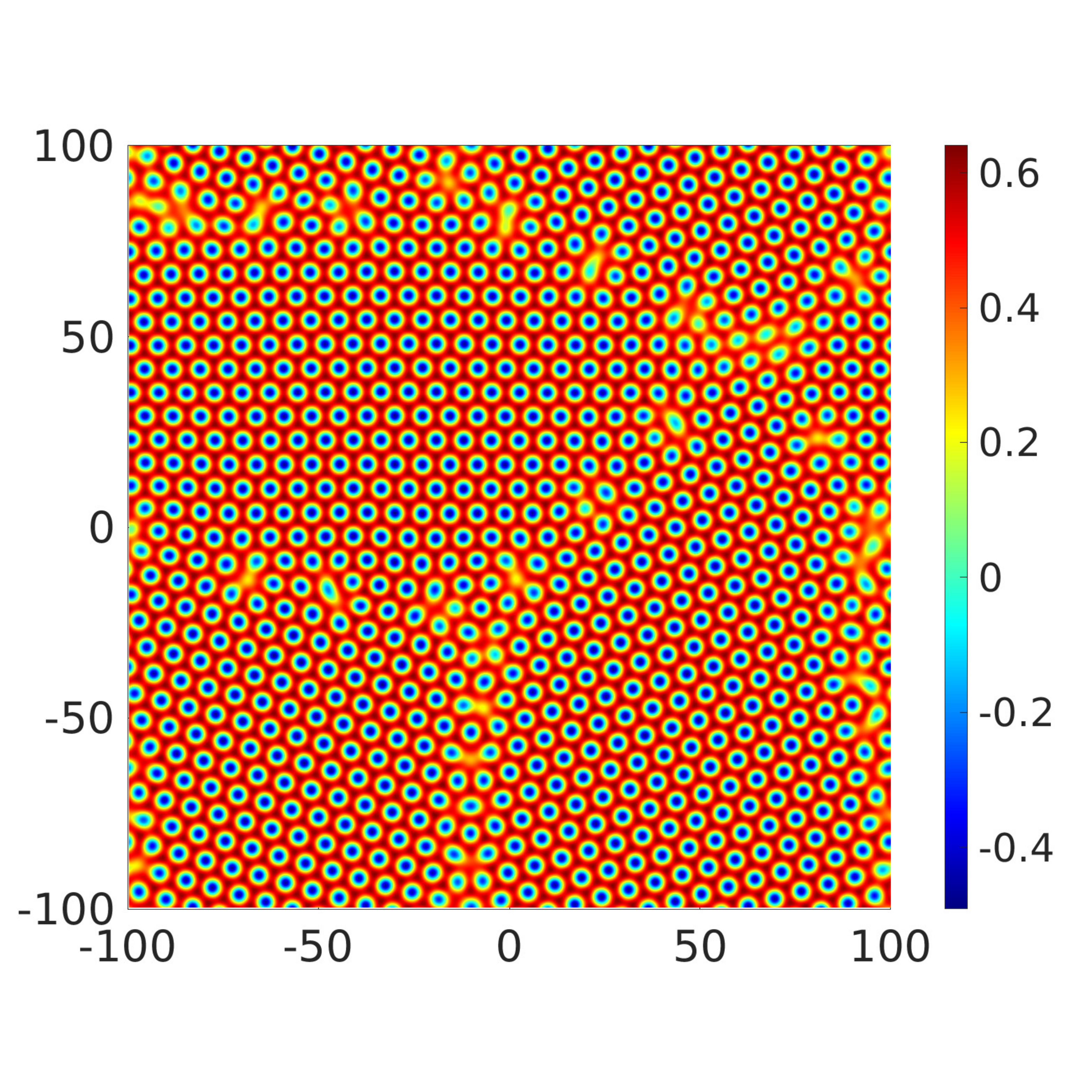}}
\caption{PFC2 evolution.}
\label{fig:evol5}
\end{figure}

\begin{figure}
	\centering
	\vspace{-5ex}
	\subfloat[t=0]{\label{figsub:PFC2big-1} \includegraphics[width=0.47\linewidth]{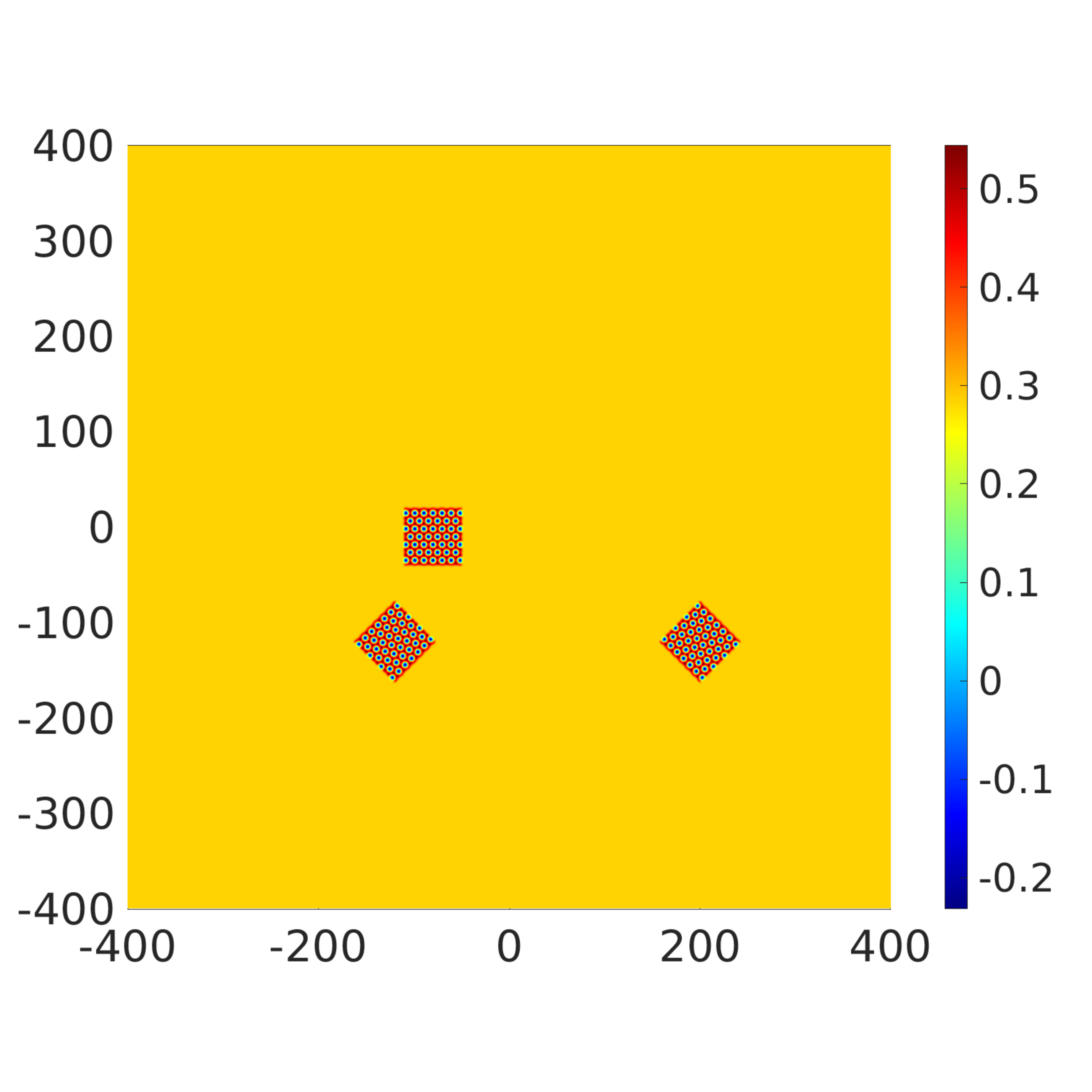}} 
\subfloat[t=20]{\label{figsub:PFC2big-2} \includegraphics[width=0.47\linewidth]{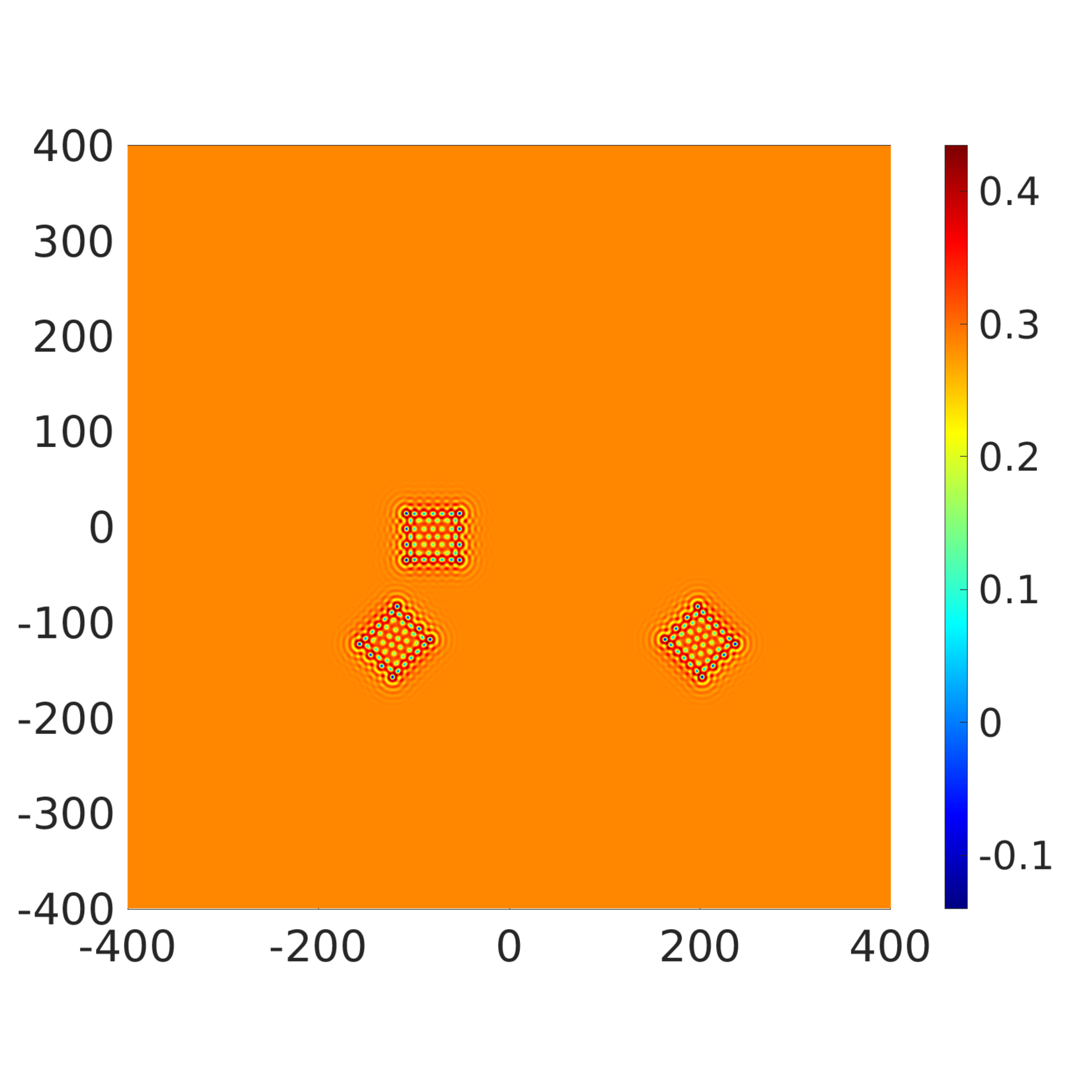}}
\\[-3ex]
\subfloat[t=150]{\label{figsub:PFC2big-3} \includegraphics[width=0.47\linewidth]{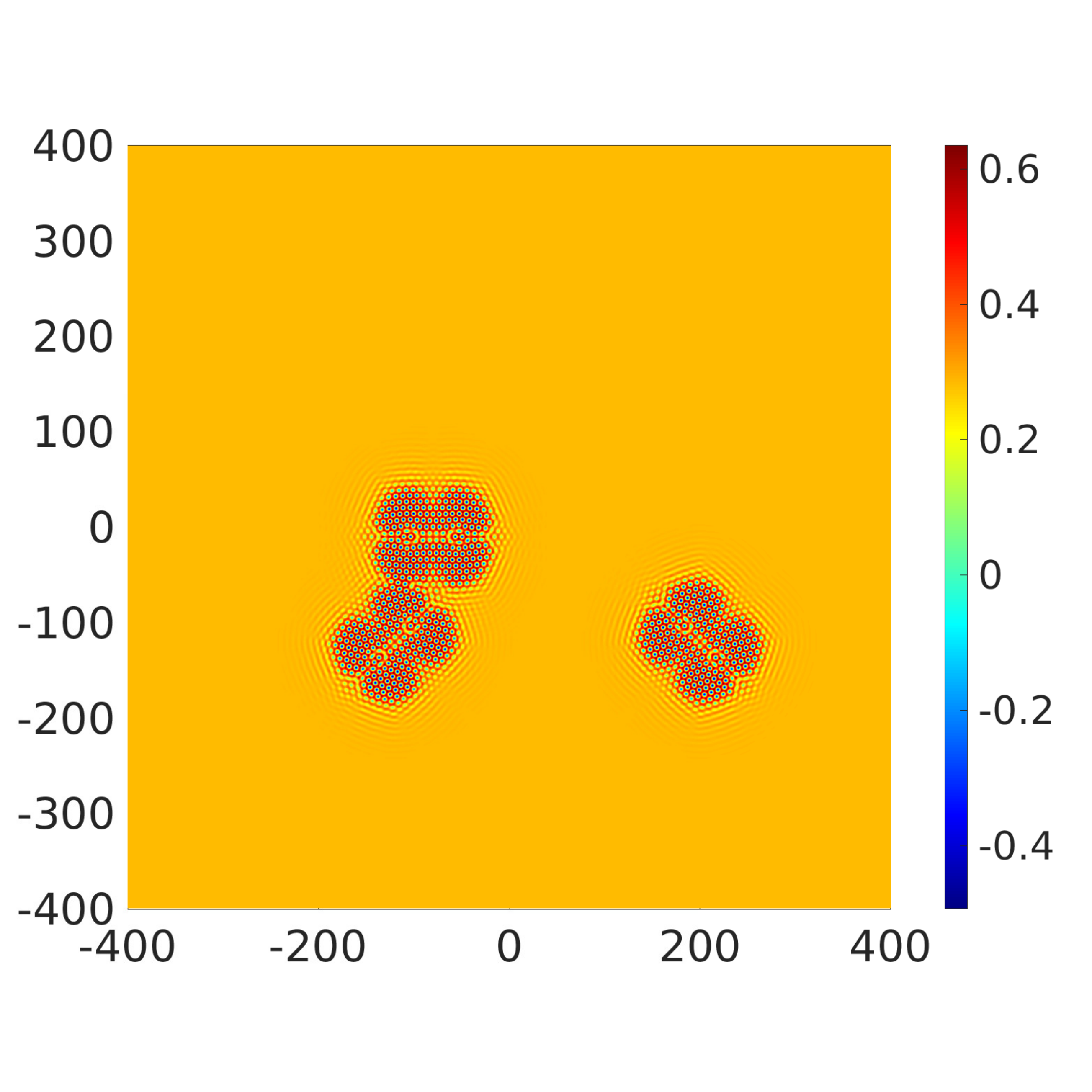}} 
\subfloat[t=400]{\label{figsub:PFC2big-4} \includegraphics[width=0.47\linewidth]{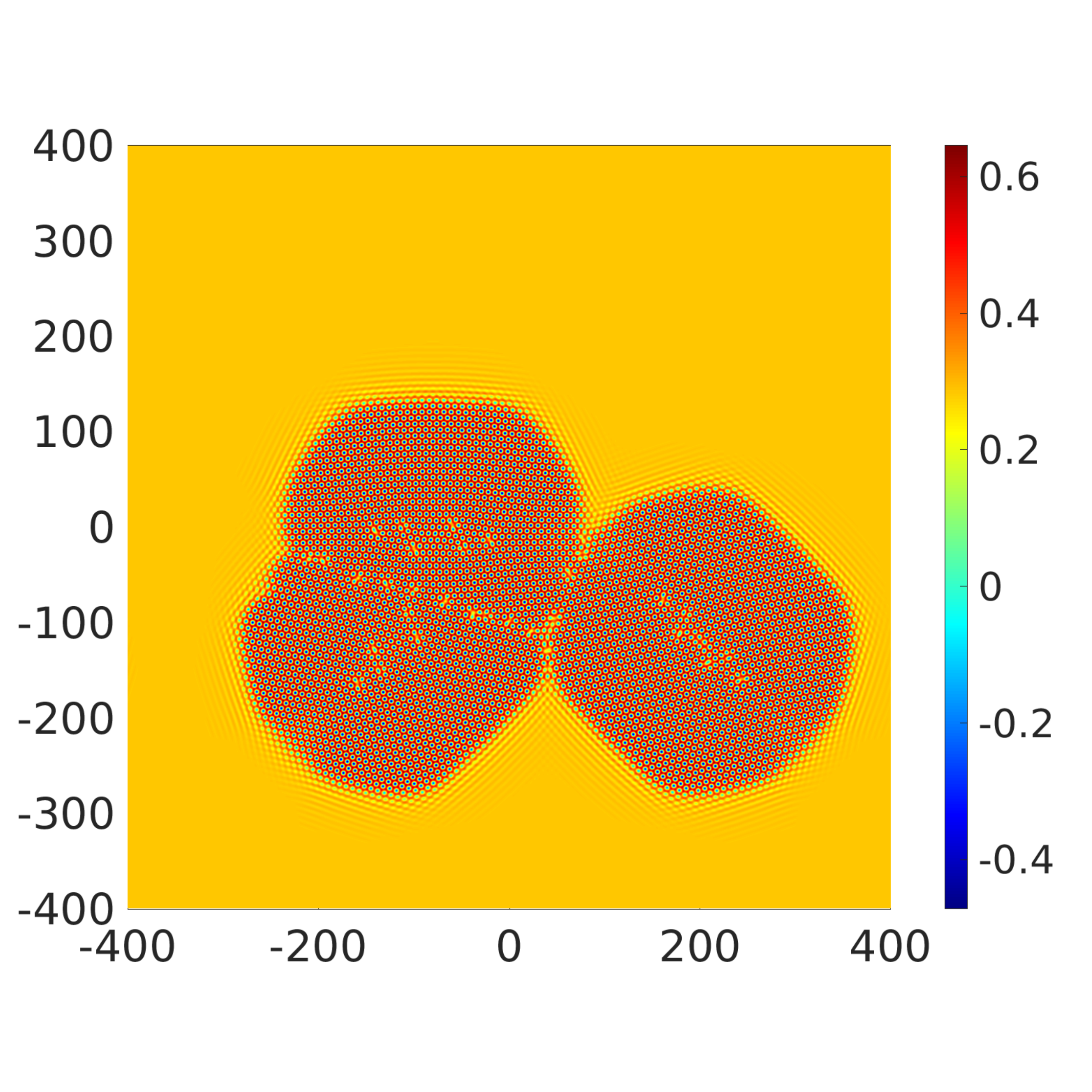}}
\\[-3ex]
\subfloat[t=1000]{\label{figsub:PFC2big-5} \includegraphics[width=0.47\linewidth]{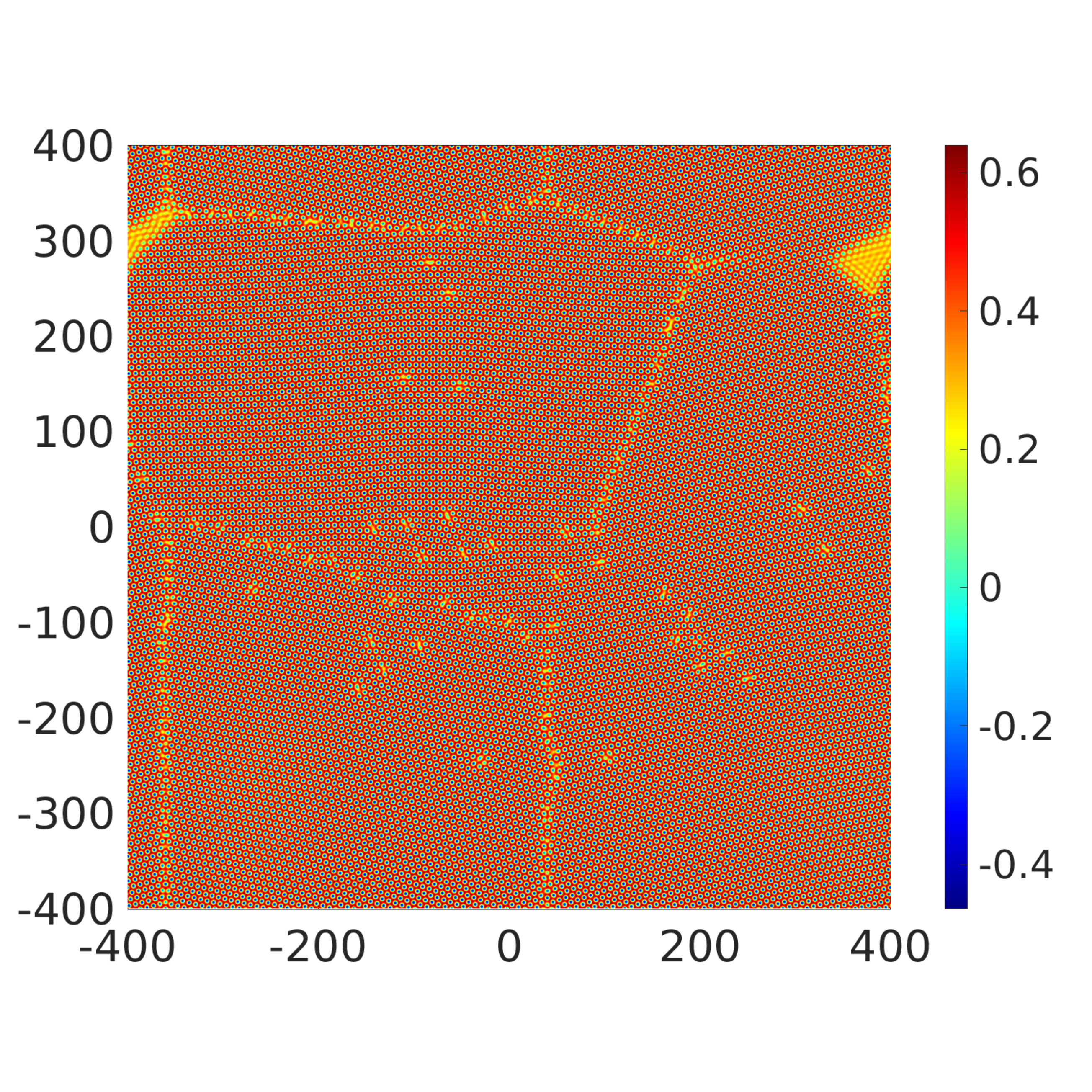}} 
\caption{Simulation of crystal growth in a supercooled liquid, as originally implemented in \citep{gomez2012PFCnumerical}. Its miniature version is used for PFC2 benchmark computations.}
\label{fig:evol6}
\end{figure}

\begin{figure}
	\centering
	\subfloat[t=18]{\label{figsub:FCH3Cmp1} \includegraphics[width=0.5\linewidth]{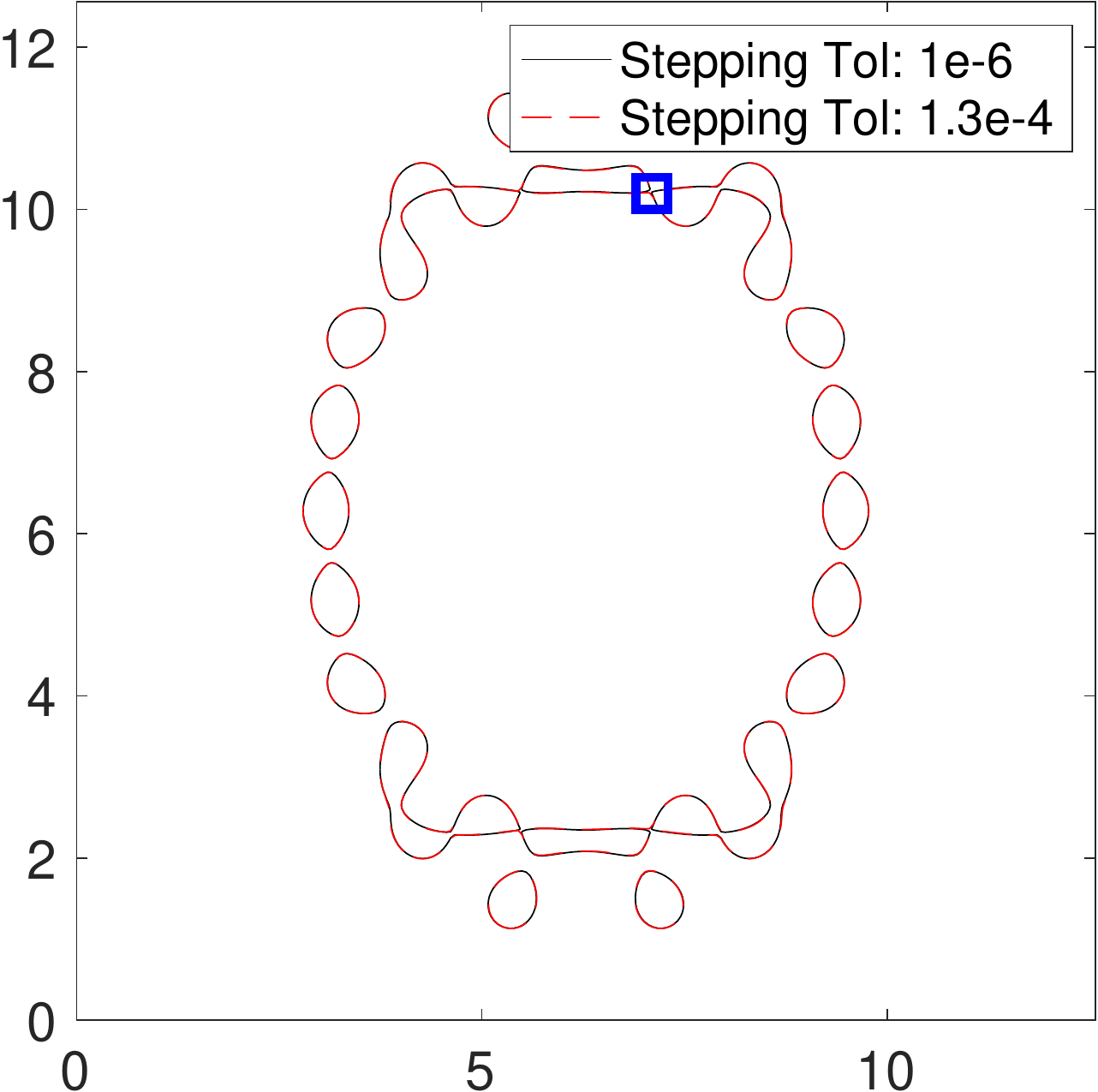}} 
	\subfloat[t=18 zoom-in]{\label{figsub:FCH3Cmp2} \includegraphics[width=0.5\linewidth]{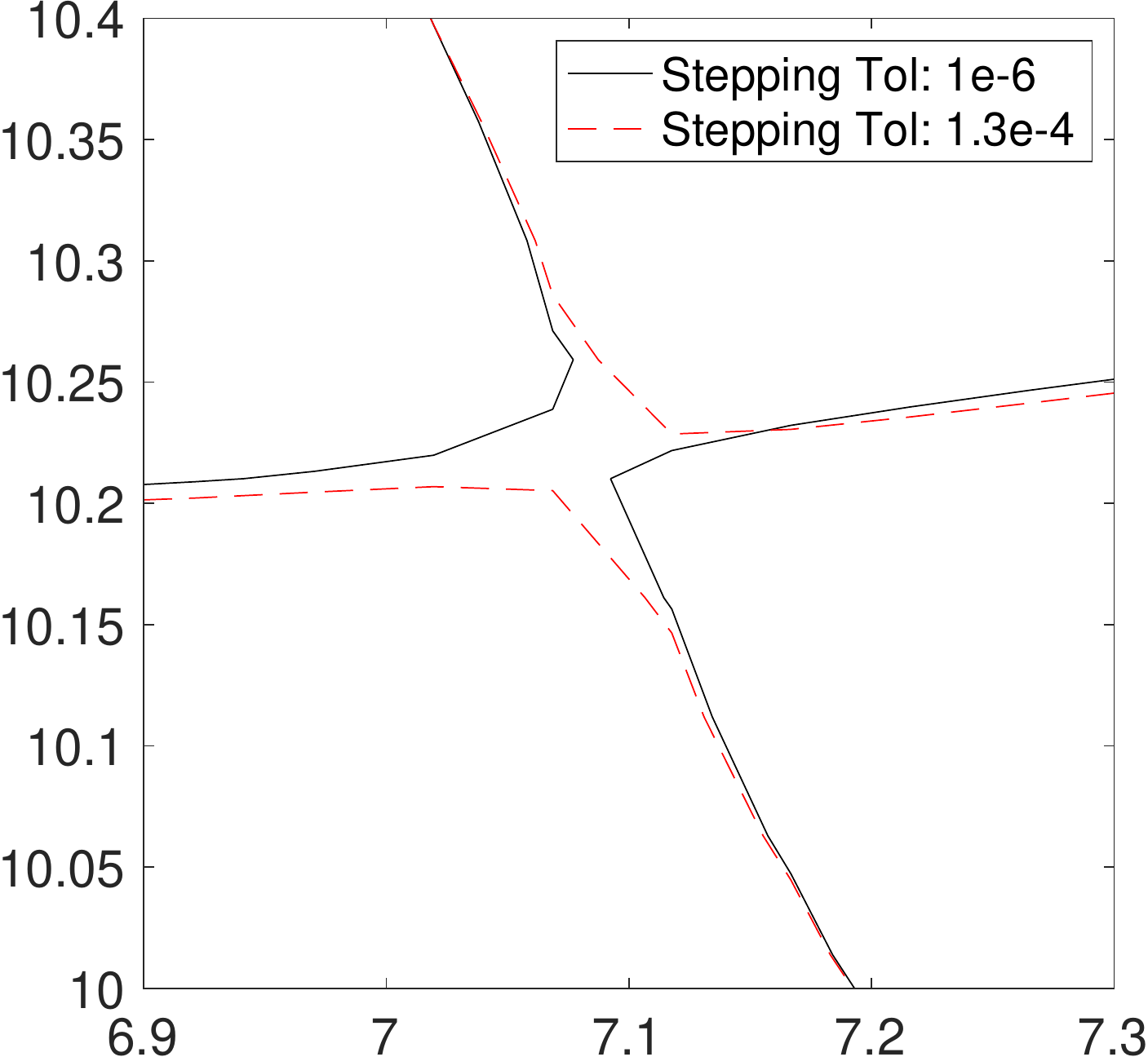}} 
	\caption{Different evolutions of FCH3 depending on different time stepping tolerances. Both are level curves of $u(x,y,t=18)=-0.01$  obtained  by BDF2 with time stepping tolerance $10^{-6}$ (solid black) and $1.3\times10^{-4}$ (dashed red), respectively. The right figure is a zoom--in around the point $(x,y)=(7.1, 10.2)$ (blue box).}
	\label{fig:FCH3CmpByResln}
\end{figure}

Detailed settings for the benchmark comparisons are summarized in Table \ref{tab:setting-benchmark}. Let us make a comment regarding the chosen range for the sweeping--$\eta$ strategy. This is obtained by taking square root of a collection of equally spaced $5$ numbers in the interval $[0.1,2]$. The square root is taken due to PAGD convergence theory. If the minimization problem is applied to a $\mu$--strongly convex functional, $\eta=\sqrt \mu$ is the optimal choice from the convergence analysis; see \citep{psw2021PAGD} for more details.

\begin{table}[h]
	\centering
	\resizebox{\textwidth}{!}{%
		\begin{tabular}{|l|l|l|}
			\hline
			& PDE setting                                                                                                                                                & Solver setting                                                                                                                                                                            \\ \hline
			\multicolumn{1}{|c|}{Common} &                                                                                                                                                            & \begin{tabular}[c]{@{}l@{}}$\dt_{max} = 0.5$, $\mathrm{TOL}_{iter} = 10^{-10}$, \\ sweeping--$\eta$ = ($\sqrt{0.1}$, $\sqrt{0.575}$, $\sqrt{1.05}$, $\sqrt{1.525}$, $\sqrt{2}$)
			\end{tabular} \\ \hline
			FCH1                         & \begin{tabular}[c]{@{}l@{}}$L=2 \pi$, $\Omega=(0,L)^2$, $\epsilon=0.18$,\\ $\eta_1= \epsilon^2$,      $\eta_2=\epsilon^2$,    $\tau = 0$\end{tabular}      & $N=2^7$, $s=0.4$, $\dt_{min} = 10^{-5}$                                                                                                                                                      \\ \hline
			FCH2                         & \begin{tabular}[c]{@{}l@{}}$L=12.8$, $\Omega=(0,L)^2$, $\epsilon=0.1$,\\ $\eta_1= 0.2$,      $\eta_2=0.2$,    $\tau = 0$\end{tabular}                      & $N=2^8$, $s=0.9$, $\dt_{min} = 10^{-5}$                                                                                                                                                      \\ \hline
			FCH3                         & \begin{tabular}[c]{@{}l@{}}$L=4 \pi$, $\Omega=(0,L)^2$, $\epsilon=0.1$,\\ $\eta_1= 1.45\epsilon$,      $\eta_2=2\epsilon$,    $\tau = 0.125$;\end{tabular} & $N=2^8$, $s=0.9$, $\dt_{min} = 10^{-5}$                                                                                                                                                      \\ \hline
			PFC1                          & \begin{tabular}[c]{@{}l@{}}$L=1.2\times32\times4\times \frac{\pi}{\sqrt 3}$, \\ $\Omega=(-L/2, L/2)^2$, $\varepsilon = 1.0$\end{tabular}                   & $N= 2^9$, $s=0.9$, $\dt_{min} = 10^{-4}$
			\\ \hline
			PFC2                          & \begin{tabular}[c]{@{}l@{}}$L=200$, \\ $\Omega=(-L/2, L/2)^2$, $\varepsilon = 0.25$\end{tabular}                   & $N= 2^9$, $s=0.9$, $\dt_{min} = 10^{-4}$                                                                                                                                                                                                                                                                          \\ \hline
		\end{tabular}%
	}
	\caption{Parameter settings for the benchmark comparison.}
	\label{tab:setting-benchmark}
\end{table}

\subsection{Performance of PAGD and PGD}
\label{sec:PAGDperform}
One of our main goals in this work is to show that the PAGD is a viable solver for certain types of challenging PDEs such as the PFC and FCH equations. To achieve this goal, for each one of these models, we compare the total computational cost of using PAGD vs.~PGD, the latter of which is known to be an efficient solver for such problems; see \citep{feng2017preconditioned, christlieb2020benchmark}. As a measure of cost, we count the number of FFTs needed to finish the evolution. Each of the benchmark problems is simulated using the solver described in Algorithm \ref{alg:solver} with the BDF2 time discretization scheme except that the MP scheme is used for PFC2,\footnote{BDF2 somehow makes both solvers extremely slow for PFC2.} once equipped with the PGD and another time with the PAGD as a solver. The AM3 adaptive time stepping is used for FCH1---3 and PFC1 with a stepping tolerance $10^{-4}$ while the midAB2 stepping is utilized for PFC2 with the same stepping tolerance. For the FCH evolutions, the final time is set to $T=100$, whereas, for PFC1 and PFC2, it is set to $T=10,000$ and $T=300$, respectively. All remaining parameter settings are the same as described in the previous section. The results are shown in Table \ref{tab:PGDvsPAGD}. 

\begin{table}[]
	\centering
	\begin{tabular}{|c|r|r|r|r|r|}
		\hline
		FFT  & \multicolumn{1}{c|}{FCH1} & \multicolumn{1}{c|}{FCH2} & \multicolumn{1}{c|}{FCH3} & \multicolumn{1}{c|}{PFC1} & \multicolumn{1}{c|}{PFC2}
		\\ \hline
		PGD  & 139841                    & 367246                    & 883543                    & 330908.5                  & 24811.5
		\\ \hline
		PAGD & 87097                     & 240374                    & 440807                    & 176887.5                  & 25284.0
		\\ \hline
	\end{tabular}
	\caption{Number of FFTs needed for the PGD and PAGD solvers to complete the evolution. The time stepping is carried out via BDF2.}
	\label{tab:PGDvsPAGD}
\end{table}

As Table \ref{tab:PGDvsPAGD} shows, the PAGD solver takes as few as half the number of FFTs needed for the PGD to carry out the same simulations for most of the cases. The exception of PFC2 is discussed in the next paragraph. We emphasize once more that PGD itself is known to be an efficient solver for the FCH model (e.g.,  \citep{christlieb2020benchmark, zhang2020benchmark}) and that PAGD needs only one more vector addition per iteration than PGD. Upon further inspection, we observe that, when the time step size is small, as is the case at the beginning of the evolution, the two solvers perform almost equally. However, when the time step size is relatively large, the PAGD costs much less than the PGD does for each time marching. This is in line with what is reported in \citep{psw2021PAGD} in the sense that the acceleration comes into play when the problem is ``hard''. 

The discussion in the preceding paragraph does not explain the results obtained for PFC2, where the computational cost (number of FFTs) of both solvers is similar. We speculate that this peculiarity is due to a well--behaving landscape of the physical energy functional associated to PFC2. From an intuitive perspective, the evolution described by PFC2 does not involve many possible bifurcations since the crystallites only grow as portions of a supercooled liquid (i.e., regions of constant phase variable) coagulate and continue the crystal pattern near the boundary of the grains. On the other hand, other simulations bear a certain symmetry in the system so that there are many possible bifurcations. As a result, there can be many more local minima in the energy functional, making solving them harder than PFC2. As reported in \citep{psw2021PAGD}, the acceleration of PAGD (in comparison to the PGD) tends to play a bigger role in ``harder" problems. The same tendency mentioned in the last two sentences of the previous paragraph is also observed in PFC2 case. However, a slightly better performance of the PGD in the beginning of the evolution (when the time step size is small) outweighs a marginally better performance of the PAGD towards the end of the simulation (when the time step size is big).

\subsection{Computational cost}
\label{sec:scheme-comparision}
In order to make our comparisons as fair as possible, we take into account both accuracy and efficiency. To this end, the experiment starts by preparing preliminary data. To be specific, we choose time $t$ of evolution (long enough for a certain morphological change to emerge) for each problem, then we choose a point $(x,y)$ in our domain, and the triple $(x,y,t)$ is formed. Then, we find a highly accurate solution, which is computed by the implicit Euler method with a constant time step that is so small that the difference between the computed values of $u(x,y,t)$ with the current and a ten times smaller time step is no larger than $10^{-6}$, while keeping a fixed spatial grid spacing. 

The implicit Euler method was chosen because, in our experience, this scheme is very robust. It is possible that the results of the numerical experiment may be different if one uses another method. However, our preliminary computations showed little difference in the point values at the reference coordinates. Following the same scheme, the difference in the reference point values computed by the implicit Euler and the LMP or LBDF2 method ranges from $1.86\times10^{-9}$ to $7.90\times10^{-7}$ across the five simulations,\footnote{For FCH1---3 and PFC1, the implicit Euler is compared with the LMP while the LBDF2 is compared for PFC2} suggesting we obtain a
6--digit precision.

The procedure we now describe aims at a 6--digit precision for the highly accurate solution at the reference point $(x,y,t)$. Let $u_{[\ell]}(x,y,t)$  be the point value at the reference point computed by setting the constant step size to $0.1^\ell$ ($\ell=1,2,3,\cdots$). Suppose that we have obtained that $|u_{[\ell]}(x,y,t)-u_{[\ell+1]}(x,y,t)| < 10^{-6}$. The second smallest constant time step size, i.e. $0.1^{\ell}$, is considered ``small enough'' for a 6--digit precision. Then, the reference point value is computed once again on a finer grid that has spacing smaller by a factor of $0.5$ than before, and with this small enough time step, i.e., $0.1^{\ell}$. If this point value still differs by less than $10^{-6}$ from the original approximation, i.e., the one before refining the grid spacing, then it is selected as the highly accurate solution, and its point value is used in the main experiment. In fact, in all problems, this is the case.  The reference coordinates $(x,y,t)$ for each benchmark problem are as listed below and also shown in Figure \ref{fig:ref-pts} (the center of the black dot). 

\begin{tabular}{cllr}
{FCH1:}& (4.71239 	&, 4.71239 	&, \hfill 10)  
\\
{FCH2:}& (7.1 		&, 8.85 	&, \hfill 10) 
\\
{FCH3:}& (6.92132, 	&, 10.7501 	&, \hfill 10)
\\
{PFC1:}& (20.6773 	&, 5.4414 	&, \hfill 1000)
\\
PFC2: & (43.3594 &, 14.4531 &, \hfill 300)
\end{tabular}

\begin{figure}
		\centering
	\subfloat[FCH1: $(x,y,t)$ = (4.71239, 4.71239, 10)]{\label{figsub:ref-pt-FCH1} \includegraphics[width=0.5\linewidth]{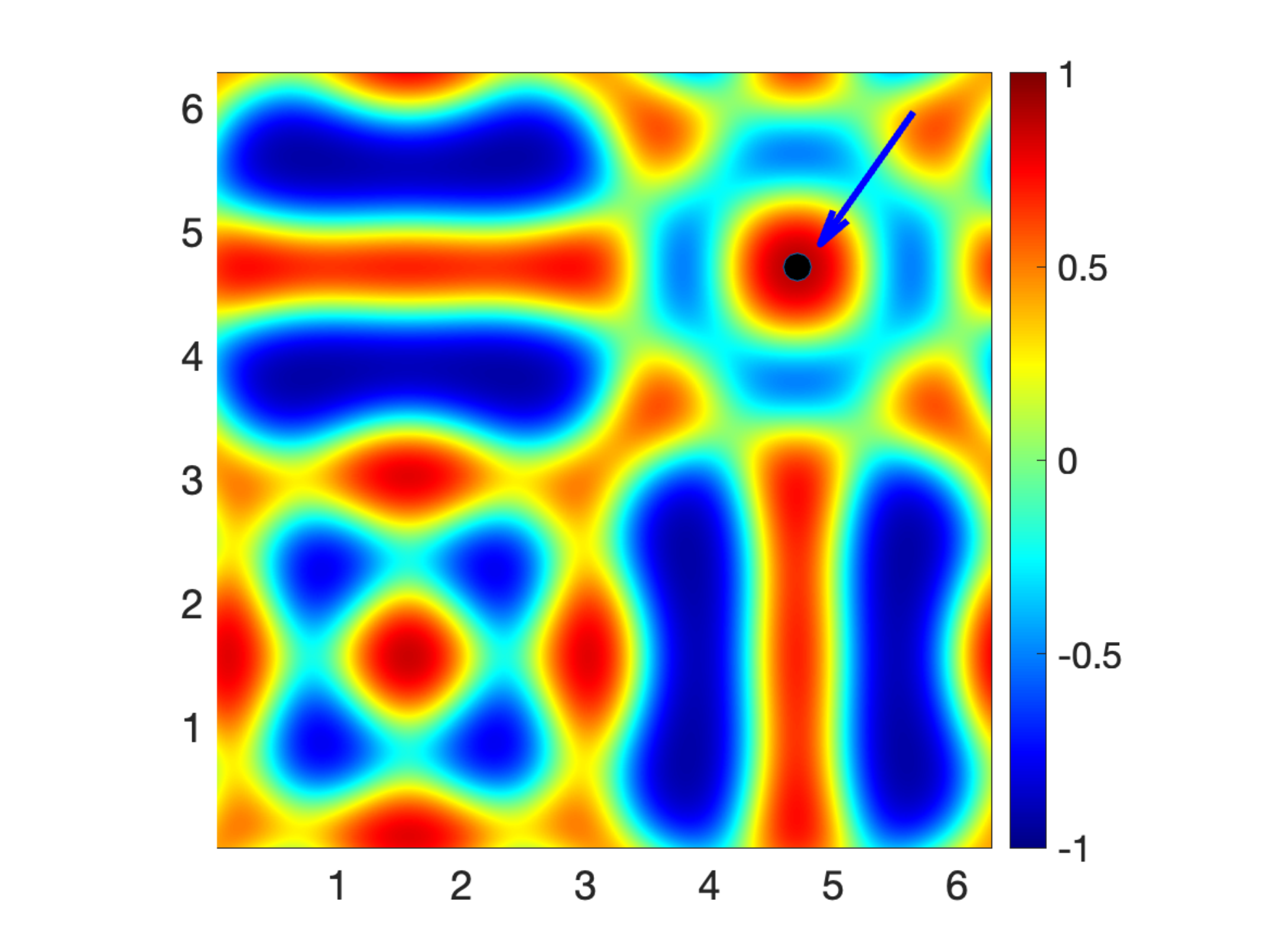}} 
	\subfloat[FCH2: $(x,y,t)$ = (7.1, 8.85, 10)]{\label{figsub:ref-pt-FCH2} \includegraphics[width=0.5\linewidth]{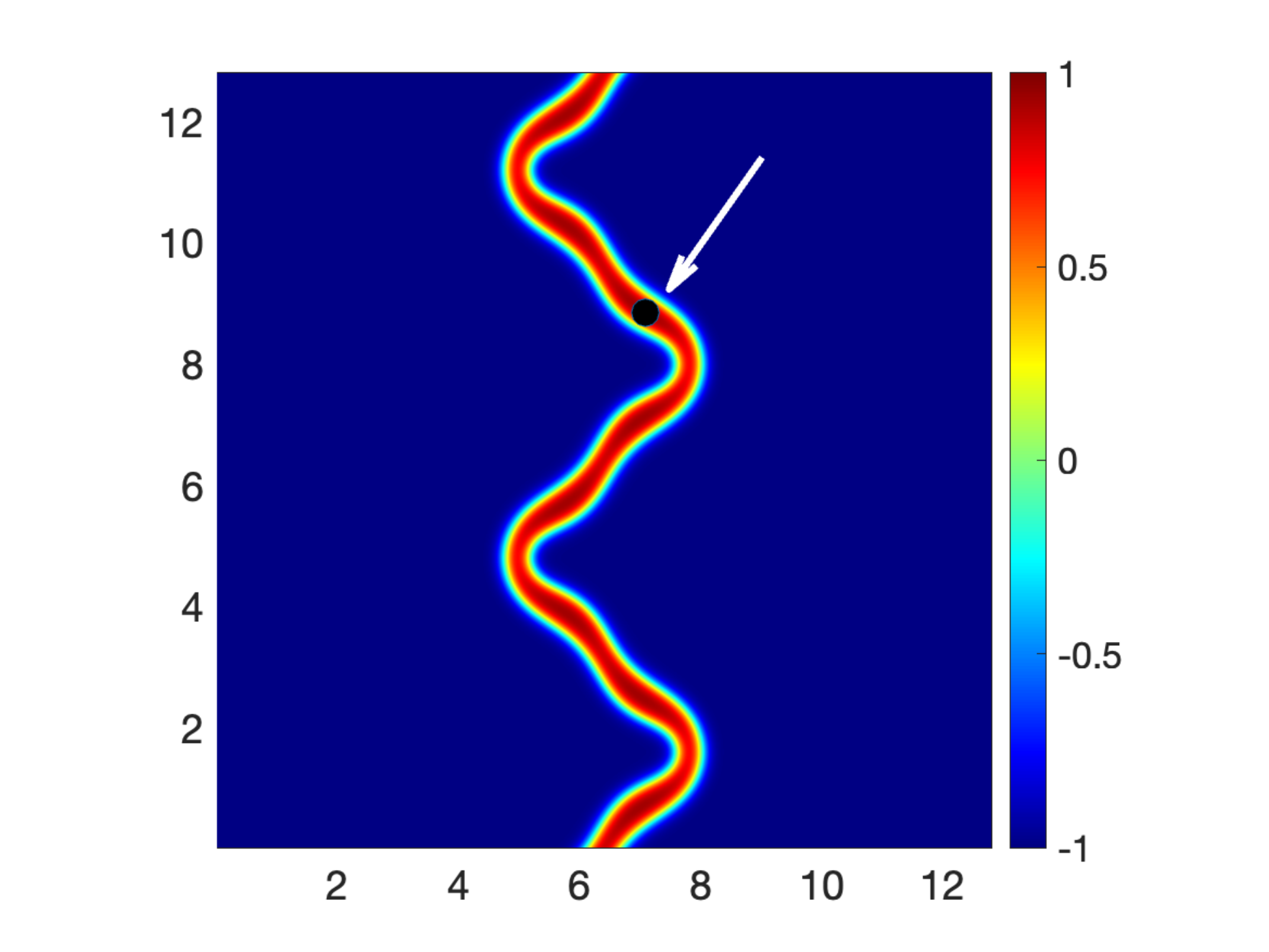}}
	\\
	\subfloat[FCH3: $(x,y,t)$ = (6.92132, 10.7501, 10)]{\label{figsub:ref-pt-FCH3} \includegraphics[width=0.5\linewidth]{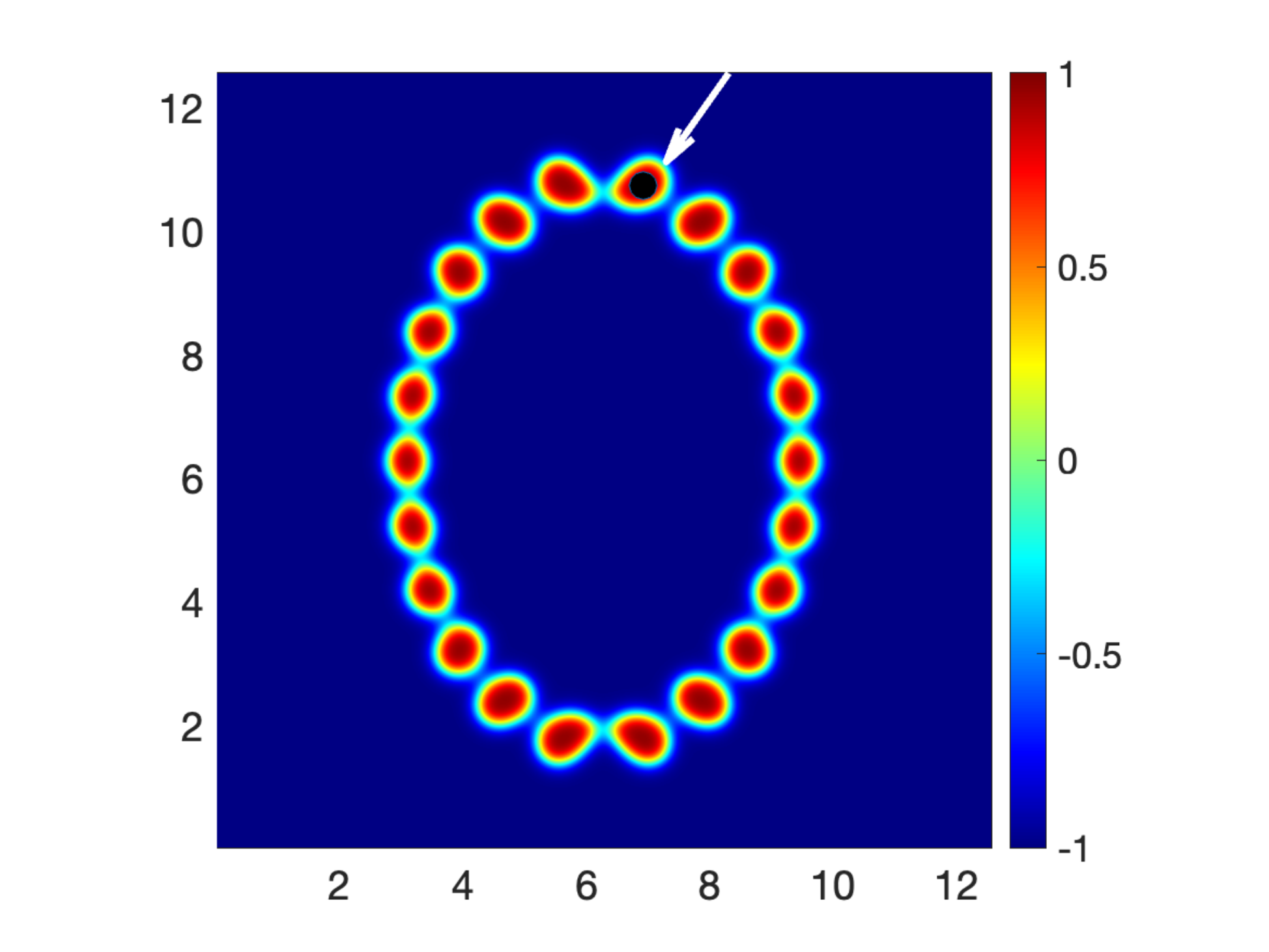}} 
	\subfloat[PFC1: $(x,y,t)$ = (20.6773,  5.4414, 1000)]{\label{figsub:ref-pt-PFC1} \includegraphics[width=0.4\linewidth]{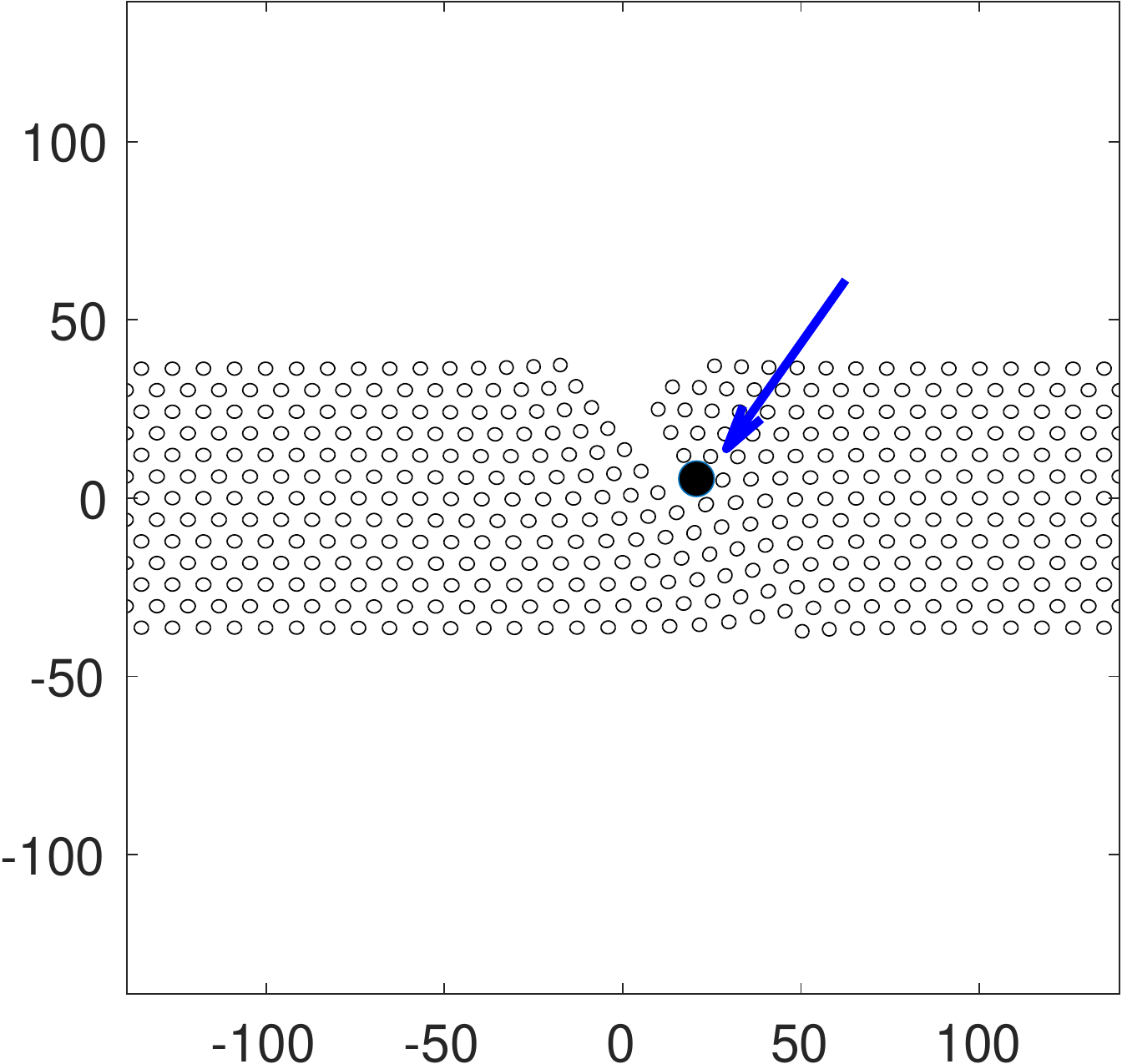}}
	\\
	\subfloat[PFC2: $(x,y,t)$ = (43.3594, 14.4531, 300)]{\label{figsub:ref-pt-PFC2} \includegraphics[width=0.5\linewidth]{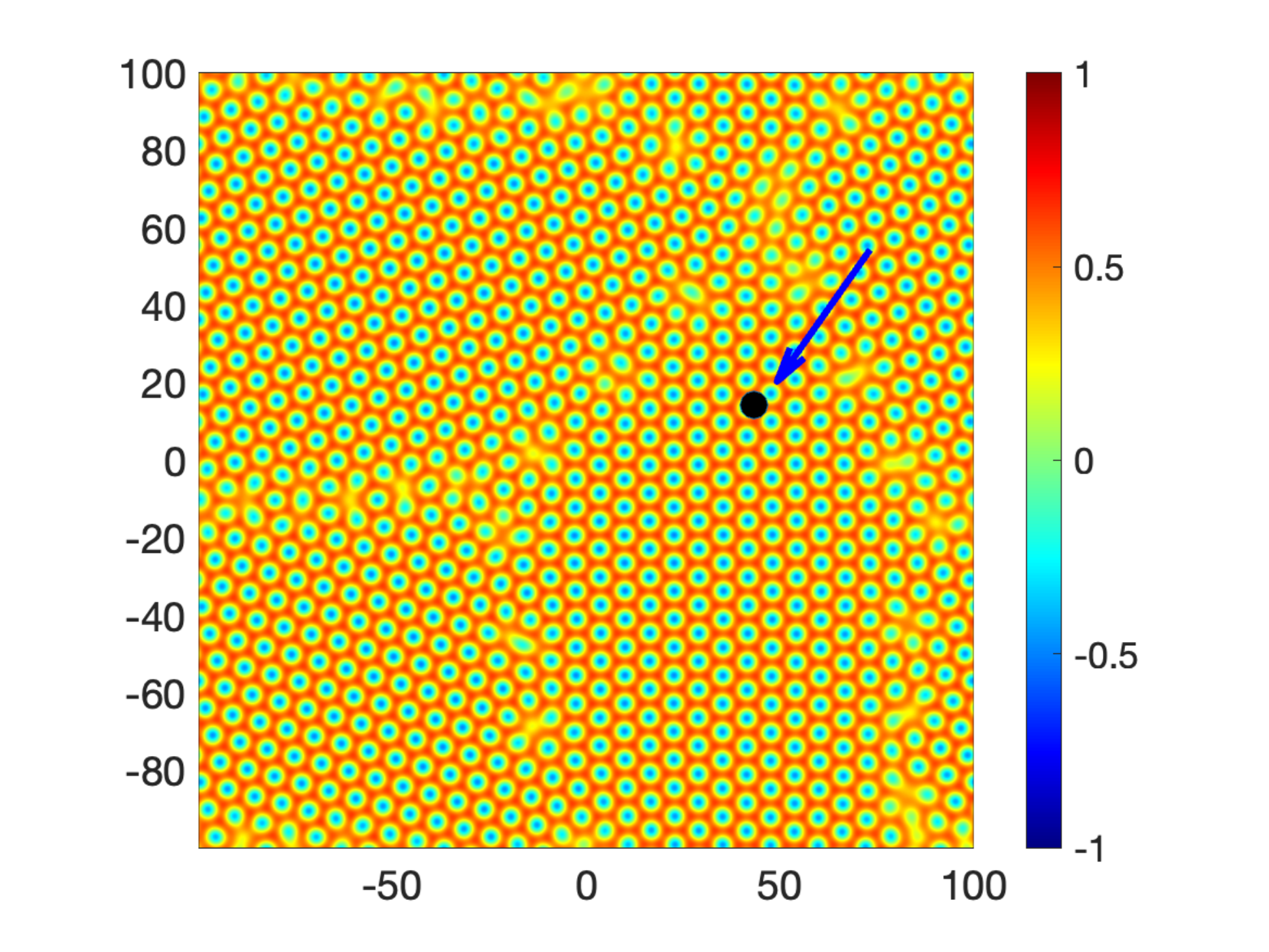}}
	\caption{Reference coordinates (the center of the black dot) for determining accuracy in the benchmark comparisons.}
	\label{fig:ref-pts}
\end{figure}

Now, the main experiment is done as Figure \ref{fig:exp-design} shows. For each combination of a problem and a solver, we start with a generous time stepping tolerance $1$, which will suggest rather large time step sizes through the adaptive time stepping. We simulate the evolution with this setting until it reaches the reference time $t$, and obtain the point value at the spatial reference coordinate $(x,y)$. Then, the \emph{objective error} is computed by the difference between $u(x,y,t)$ and the value of the precomputed, highly accurate solution at our reference point. If this error is larger than the \emph{objective tolerance} $10^{-5}$, which aims at a 5--digit precision, we restart the experiment with a time stepping tolerance that is reduced by a factor of $0.1$. This process is repeated until the objective error is smaller than the objective tolerance. When this occurs, we record the cost: FFT count (total number of the FFT and iFFT divided by two), the wall--clock time, and the CPU time consumed.

\begin{figure}[h]
	\begin{tikzpicture}[node distance=4cm, auto, >=latex', thick]
		\path[->] node[format] (steptol) {\textbf{Start} \\ time stepping \\ tolerance \\ ($1\sim10^{-10}$)}; 	
		\path[->] node[format, right of=steptol] (curptval) {point value\\ at $(x,y,t)$\\(current scheme)}
		(steptol) edge node {simulate} (curptval); 		
		\path[->] node[decision, right of=curptval] (errexam) {\tiny $\text{error}\stackrel{(?)}{<} \begin{pmatrix}\text{objective}\\ \text{tolerance}\end{pmatrix}$}
		node[format, below of=curptval] at (4,2) 
		(refptval) {point value at $(x,y,t)$ \\ highly accurate solution}
		(curptval) edge node (temp) {} (errexam); 		
		
		\path[->] node[format, below of =errexam] at (7,0) (record) {record\\cost}
		(errexam) edge node {Yes} (record)
		(refptval) edge (errexam);
		
		\path[<-, draw] (steptol) -- + (0,1.5) -| node[near start] {more restrictive stepping tolerance ($\times 0.1$)} node[near end] {No} (errexam);
	\end{tikzpicture}
	\caption[design of experiment]{A flowchart to determine computational cost.}
	\label{fig:exp-design}
\end{figure}
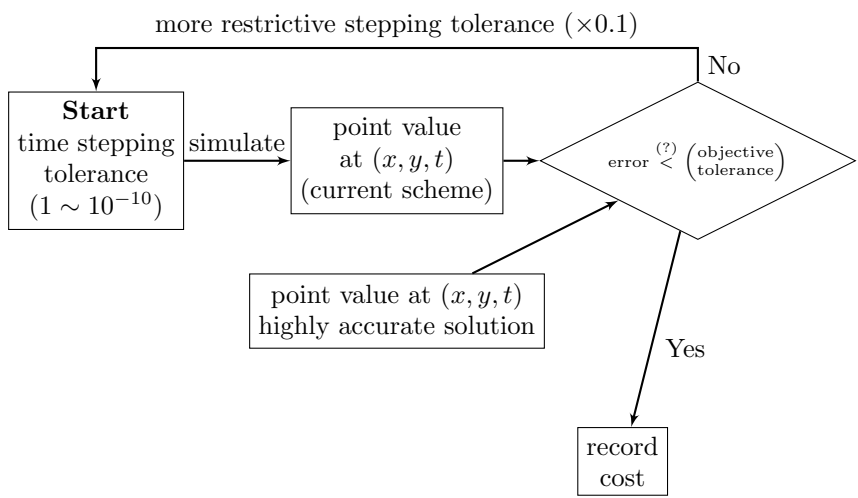

\begin{table}[]
	\centering
	\resizebox{\textwidth}{!}{%
		\begin{tabular}{|c|c|r|r|r|r|r|r|}
			\hline
			\bfseries Prob &
			\bfseries Scheme &
			\multicolumn{1}{c|}{\bfseries Step tol.} &
			\multicolumn{1}{c|}{\bfseries Point value} &
			\multicolumn{1}{c|}{\bfseries Obj. err.} &
			\multicolumn{1}{c|}{\bfseries FFT} &
			\multicolumn{1}{c|}{\bfseries Clock (sec)} &
			\multicolumn{1}{c|}{\bfseries CPU (sec)} 
			\begin{filecontents*}{BenchmarkPFCH.csv}
Eqn,Scheme,TadpTOL,objvval,objverr,FFT,ClockTime,CPUTime
FCH1,MP,0.0001,0.8886817507,-2.44E-07,43769,14,81
FCH1,BDF2,1E-06,0.8886738923,-8.1E-06,77385,23,128
FCH1,LMP,1E-08,0.8886791836,-2.81E-06,189479,102,691
FCH1,LBDF2,1E-08,0.8886785837,-3.41E-06,221987,116,791
FCH2,MP,1E-05,0.9254124244,7.77E-06,122049,116,639
FCH2,BDF2,1E-07,0.9254110119,6.36E-06,106425,90,480
FCH2,LMP,1E-08,0.9254006926,-3.96E-06,260438,488,2871
FCH2,LBDF2,1E-08,0.9253997437,-4.91E-06,297146,537,3169
FCH3,MP,1E-05,0.9597908022,-8.85E-06,113089,112,611
FCH3,BDF2,1E-06,0.9598029022,3.25E-06,142293,120,643
FCH3,LMP,1E-09,0.9597965984,-3.05E-06,727975,1366,7946
FCH3,LBDF2,1E-09,0.959796049,-3.6E-06,864888,1639,9519
PFC1,MP,1,-1.190098311,-1.38E-06,78289,595,5980
PFC1,BDF2,1,-1.190094247,2.68E-06,82485,437,4178
PFC1,LMP2,0.0001,-1.190101571,-4.64E-06,616922,5847,55413
PFC1,LBDF22,0.001,-1.190105174,-8.24E-06,16785,141,1319
PFC2,MP,1,0.609267849,9.07E-08,25551,77,505
PFC2,BDF2,0.01,0.609267701,-5.75E-08,237792,747,5174
PFC2,LMP2,0.001,0.609276681,8.92E-06,6026,53,291
PFC2,LBDF22,0.001,0.609269078,1.32E-06,517858,2896,20668
\end{filecontents*}
\csvreader[head to column names]{BenchmarkPFCH.csv}{}
			{\\\hline \Eqn & \Scheme & \TadpTOL & \objvval	& \objverr &	\FFT &	\ClockTime	& \CPUTime 
		}
		\\\hline
		\end{tabular}%
	}
	\caption{Results of benchmark experiment.}
	\label{tab:expbenchmark_csv}
\end{table}

\begin{figure}
	\subfloat[FFT count]{\includegraphics[height=0.26\textheight]{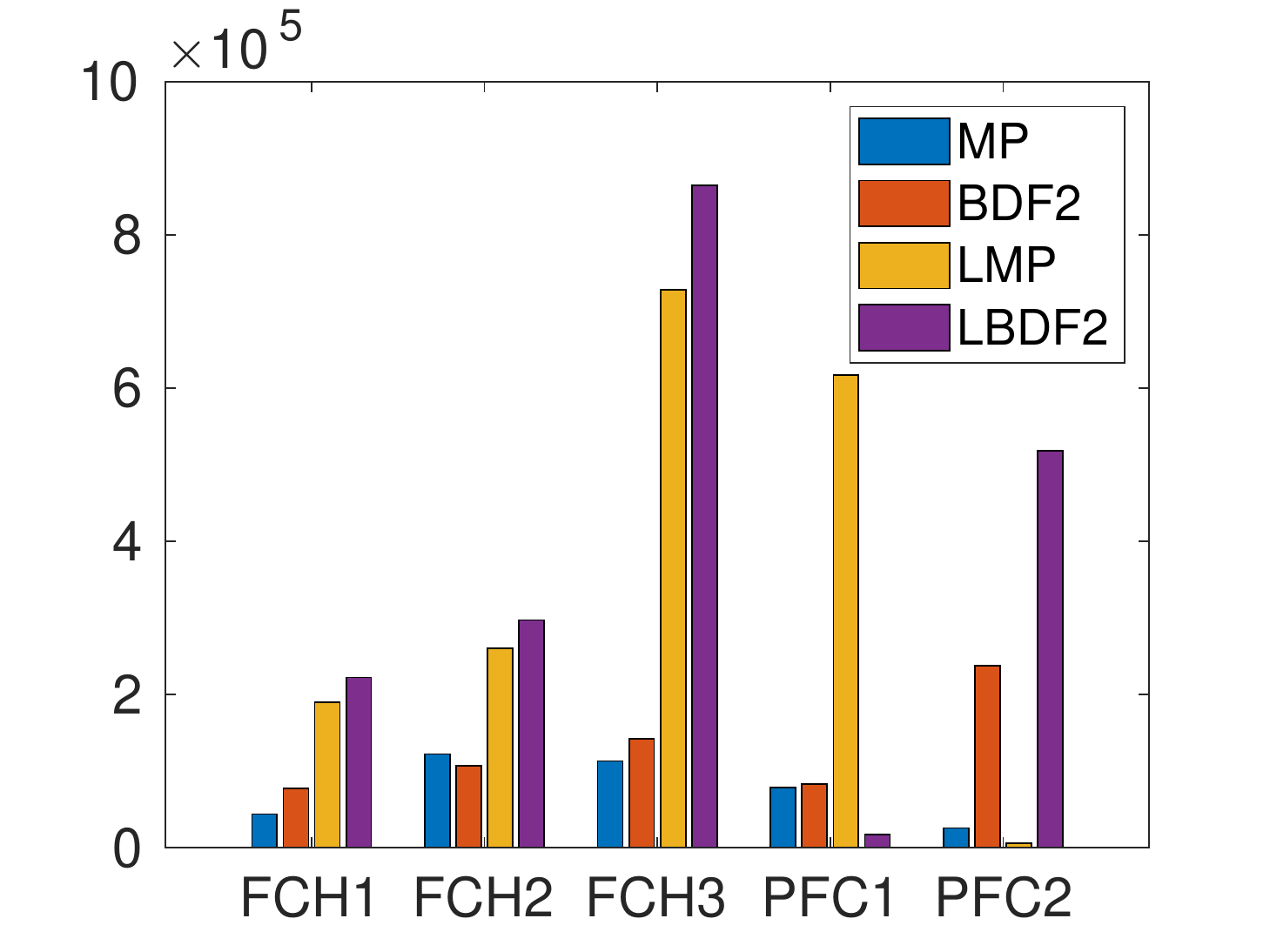}}
	\quad
	\subfloat[Clock Time (sec)]{\includegraphics[height=0.26\textheight]{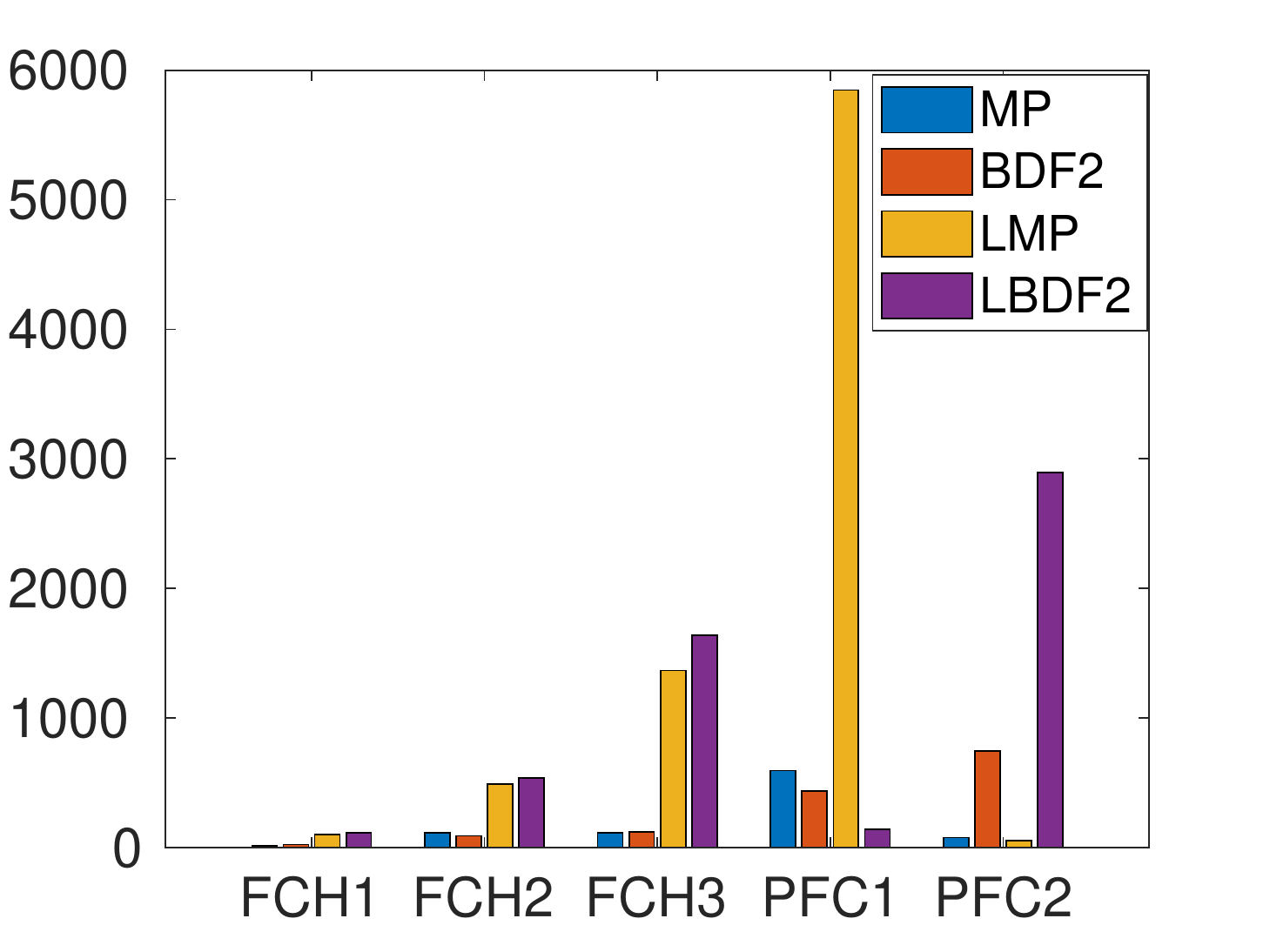}}
	\\
	\subfloat[CPU Time (sec)]{\includegraphics[height=0.26\textheight]{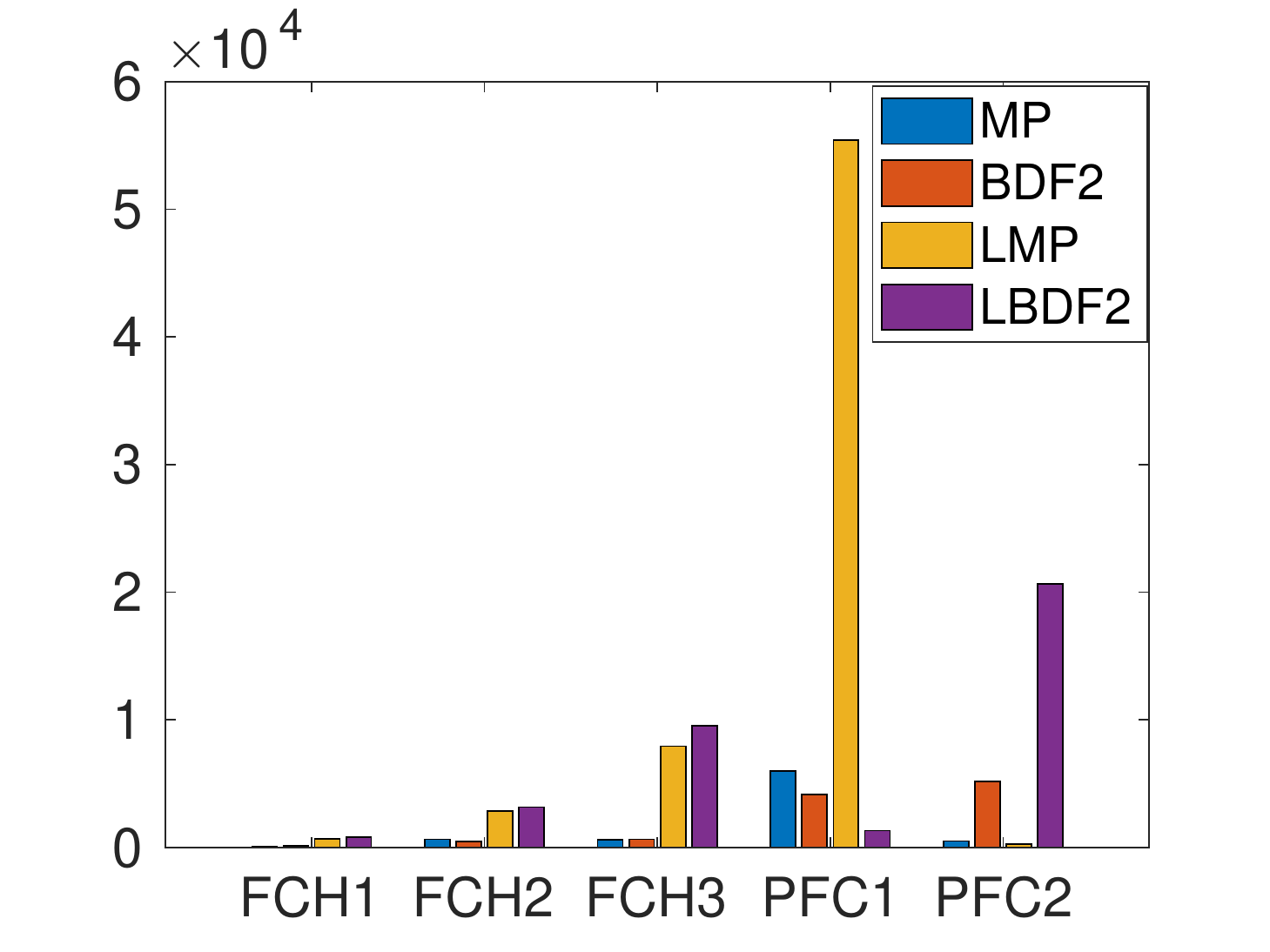}}
	\caption{Comparisons between number of FFTs, the wall--clock time, and CPU time of the MP, BDF2, LMP, and LBDF2. Implicit schemes perform better than semi--implicit schemes.}
	\label{fig:perf-cmp}
\end{figure}

The results of this experiment are summarized in Table \ref{tab:expbenchmark_csv}. Charts that reorganize our findings are also given in Figure \ref{fig:perf-cmp} so that one can compare the performance easily. 
As can be seen in Figure \ref{fig:perf-cmp}, for FCH1---3, implicit schemes perform significantly better than semi--implicit schemes. In an extreme case, the MP takes less than a fifteenth CPU time than the LBDF2 does for the FCH3 benchmark problem. Interestingly, for PFC1, semi--implicit schemes perform either very poorly (LMP) or very well (LBDF2) while the performance of implicit schemes is somewhere in the middle. Even more interestingly, for PFC2, the difference in performance diverges according to the time discretization, MP--based vs.~BDF2--based, rather than fully-- vs.~semi--implicit nature. And the worst performing solver for PFC1, namely LMP, shows an extremely good result. We suspect that the fact that the nonlinearity of the PFC equation is milder than that of the FCH equation (as one can see from their chemical potentials) explains the good performance of implicit solvers on FCH1---3. Also, although further investigations are needed, the well--behaving nature of the PFC2 evolution mentioned in the last paragraph of Section \ref{sec:PAGDperform} seems to make MP--based schemes more efficient than the BDF2--based ones.

\section{Conclusion}
\label{sec:conclusion}
In this work, we have introduced an efficient, time--adaptive solver for the FCH and PFC equations featuring the PAGD as solver. We observed that, if solver parameters are appropriately chosen, a PAGD-based solver outperforms a PGD-based solver, the latter of which has been recently developed and proven to be efficient on its own.

We have also conducted an experiment so that both the accuracy and the efficiency are measured in some way, and compared the performance of two fully implicit schemes, the MP and BDF2 equipped with the PAGD, and those of two semi--implicit schemes, the LMP and LBDF2. Our results show that the implicit schemes can be a good choice from a practical point of view, particularly for highly nonlinear problems. For such problems, although some desirable properties (e.g., unique solvability) may be not available for highly nonlinear, nonconvex problems, implicit schemes tend to take less cost to yield similarly accurate simulations than the semi--implicit schemes considered in this work, provided the implicit schemes are equipped with an efficient nonlinear solver such as the PAGD. This efficient solver can be harnessed without much of tedious paramter tuning with the help of averaged Newton preconditioner and the sweeping-friction strategy. Semi--implicit schemes can be very efficient for solving equations of a milder nonlinearity if an appropriate time discretization is chosen and a \emph{good} decomposition of the linear part and nonlinear part is chosen.

\section*{Acknowledgement}
The work of JHP was partially supported by NSF grants DMS-1720213, DMS-1719854, and DMS-2012634. The work of AJS was partially supported by NSF grants DMS-1720213 and DMS-2111228. The work of SMW was partially supported by  DMS-1719854 and DMS-2012634.

\bibliographystyle{plainnat}
\bibliography{PFCH}

\end{document}

%% file: JHSymbol.tex
\newcommand{\norm}[2]{\bigl\| #1 \bigr\| _{#2}}
\newcommand{\pairing}[2]{\bigl\langle #1 , #2 \bigr\rangle}
\newcommand{\argmin}{\mathop{\mathrm{argmin}}}

\newcommand{\LL}{\mathcal{L}}
\newcommand{\LLinv}{{\mathcal{L}^{-1}}}
\newcommand{\lap}{\Delta}

\newcommand{\HH}{\mathbb{H}}

\newcommand{\RR}{\mathbb{R}}

\newcommand{\NN}{\mathbb{N}}
\newcommand{\ZZ}{\mathbb{Z}}
\newcommand{\CC}{\mathbb{C}}

\theoremstyle{plain}
\newtheorem{thm}{Theorem}[section]
\newtheorem{prop}[thm]{Proposition}

\theoremstyle{definition}

\newtheorem{rmk}[thm]{Remark}

\newcommand{\ermk}{\hfill\ensuremath{\blacksquare}}

\newcommand{\goesto}{\rightarrow}


%% file: JHSettingArticle.tex
\usepackage{amsmath, amsfonts, amssymb, amsthm}
\usepackage{mathtools}
\usepackage{graphicx}
\usepackage{hyperref}
\usepackage{subfig}
\usepackage{algorithm2e}
\usepackage{tablefootnote}
\usepackage{csvsimple} 
